\documentclass[10pt, a4paper, reqno]{amsart}

\usepackage{amsmath,amssymb,amscd,amsthm,mathtools}

\usepackage[T1]{fontenc}

\usepackage{nicematrix}
\usepackage{tikz}

\usepackage[margin=2.5cm]{geometry}

\title{A De Rham Perspective on the Symplectic Geometry of Teichm\"uller Space}
\author{Antoine Ablondi}
\date{}

\usepackage{hyperref}
\hypersetup{colorlinks, linkcolor={newred}, citecolor={newbrown}, 
 pdftitle={A De Rham Perspective on the Symplectic Geometry of Teichm\"uller Space},
	pdfauthor={Antoine Ablondi},
 bookmarksdepth=2}
\usepackage{pdfpages}

\address{Antoine Ablondi: IMAG, Universit\'e de Montpellier, CNRS, Montpellier, France.} 
\email{antoine.ablondi@umontpellier.fr}
\date{}

\usepackage{setspace}

\usepackage{enumitem}

\usepackage[english]{babel}

\usepackage[justification=centering]{caption}

\usepackage{tikz-cd}
\usepackage{pgfplots}
\pgfplotsset{compat=1.16}
\usetikzlibrary{arrows}

\usetikzlibrary{tqft}

\usepackage{xcolor}
\definecolor{newred}{RGB}{200, 30, 45}
\definecolor{newyellow}{RGB}{255, 222, 23}
\definecolor{newblue}{RGB}{25, 48, 126}
\definecolor{newgrey}{RGB}{115, 110, 95}
\definecolor{newbrown}{RGB}{160, 62, 45}
\definecolor{sienna}{RGB}{160, 62, 45}

\newtheorem{maintheorem}{Theorem}

\newtheorem{maincorollary}[maintheorem]{Corollary}

\newtheorem{theorem}{Theorem}[section]
\newtheorem*{theorem*}{Theorem}
\newtheorem{corollary}[theorem]{Corollary}
\newtheorem{lemma}[theorem]{Lemma}
\newtheorem{proposition}[theorem]{Proposition}

\theoremstyle{definition}
\newtheorem{definition}[theorem]{Definition}

\theoremstyle{remark}
\newtheorem{remark}[theorem]{Remark}

\newtheorem*{remark*}{Remark}

\usepackage[labelsep=period]{caption}
\numberwithin{equation}{section}
\numberwithin{figure}{section}

\newcommand{\R}{\mathbf{R}} 

\newcommand{\Z}{\mathbf{Z}}
\renewcommand{\H}{\mathbf{H}}
\newcommand{\Rmink}{\R^{2,1}}
\newcommand{\bilmink}[2]{\left\langle {#1},{#2} \right\rangle_{\mbox{\tiny $2,1$}}}

\newcommand{\Id}{\mathrm{Id}} 

\newcommand{\newoplus}{\, \oplus \,}

\newcommand{\dif}{\mathrm{d}} 
\newcommand{\T}{\mathrm{T}} 

\newcommand{\SO}{\mathrm{SO}} 
\renewcommand{\O}{\mathrm{O}} 
\newcommand{\so}{\mathfrak{so}}

\newcommand{\dev}{\mathrm{dev}} 

\newcommand{\st}{\; \middle\vert \;} 

\newcommand{\namelessfunction}[4]{ \left\lbrace \begin{matrix}
{#1} &\longrightarrow& {#2}\\
{#3} & \longmapsto & {#4}
\end{matrix} \right. }
\newcommand{\function}[5]{{#1}: \namelessfunction{#2}{#3}{#4}{#5}}

\newcommand{\vp}[1]{{\left( {#1} \right)}} 
\newcommand{\bp}[1]{{\bigl( {#1} \bigr)}} 

\renewcommand{\S}{\mathcal{S}}
\newcommand{\pant}{\mathcal{P}}

\newcommand{\upant}{\tilde{\pant}}
\newcommand{\teich}{\mathcal{T}}
\newcommand{\diffo}{\mathrm{Diff}_0}

\newcommand{\hypcan}{\mathbf{h}}

\newcommand{\bundle}{E_{\rho}}
\newcommand{\bundlerest}[1]{E_{\rho,#1}}

\DeclareMathOperator{\Hol}{\mathrm{Hol}}
\newcommand{\Isom}{\mathrm{Isom}}
\newcommand{\isom}{\mathfrak{isom}}
\DeclareMathOperator{\Ad}{Ad}

\newcommand{\chifd}{\chi_{\mbox{\tiny \textit{fd}}}^+}

\newcommand{\Hisom}{H^1_{\mbox{\tiny $\Ad \rho$}}(\pi_1\S,\isom(\H^2))}
\newcommand{\Zisom}{Z^1_{\mbox{\tiny $\Ad \rho$}}(\pi_1\S,\isom(\H^2))}
\newcommand{\Bisom}{B^1_{\mbox{\tiny $\Ad \rho$}}(\pi_1\S,\isom(\H^2))}

\newcommand{\Hisomeq}{H^1_{\mbox{\tiny $\Ad \rho$}}\bigl(\pi_1\S,\isom(\H^2)\bigr)}
\newcommand{\Zisomeq}{Z^1_{\mbox{\tiny $\Ad \rho$}}\bigl(\pi_1\S,\isom(\H^2)\bigr)}
\newcommand{\Bisomeq}{B^1_{\mbox{\tiny $\Ad \rho$}}\bigl(\pi_1\S,\isom(\H^2)\bigr)}

\newcommand{\HdR}{H^1_{\mbox{\tiny \textup{dR}}}(\S,\bundle)}
\newcommand{\ZdR}{Z^1_{\mbox{\tiny \textup{dR}}}(\S,\bundle)}
\newcommand{\BdR}{B^1_{\mbox{\tiny \textup{dR}}}(\S,\bundle)}

\newcommand{\HdRof}[1]{H^1_{\mbox{\tiny \textup{dR}}}({#1},\bundlerest{#1})}
\newcommand{\ZdRof}[1]{Z^1_{\mbox{\tiny \textup{dR}}}({#1},\bundlerest{#1})}

\newcommand{\ddto}{\left. \frac{\dif}{\dif t}\right\vert _{t=0}}
\newcommand{\ddso}{\left. \frac{\dif}{\dif s}\right\vert _{s=0}}

\newcommand{\wwp}{\omega_{\mbox{\tiny \textup{WP}}}}

\newcommand{\wg}{\omega_{\mbox{\tiny \textup{G}}}}
\newcommand{\wdR}{\omega_{\mbox{\tiny \textup{dR}}}}

\newcommand{\dR}[1]{[ {#1} ]_{\mbox{\tiny \textup{dR}}}}

\DeclareMathOperator{\tr}{Tr}
\newcommand{\kil}{\kappa}

\newcommand{\alphat}{\tilde{\alpha}}

\DeclareMathOperator{\Ima}{Im}

\newcommand{\tH}{\mathfrak{t}}
\newcommand{\twinf}[1]{\mathfrak{t}_{#1}}

\makeatletter
\newcommand{\myitem}[1]{%
\item[#1]\protected@edef\@currentlabel{#1}%
}
\makeatother

\DeclareRobustCommand\epsind{\varepsilon}

\begin{document} 

\begin{abstract}

We present a new proof of Wolpert's Magic Formula, stating that any Fenchel--Nielsen coordinates on the Teichm\"uller space of a closed surface are Darboux coordinates for the Weil--Petersson symplectic form. 
Our approach relies on a de Rham cohomology model for the tangent space to the character variety model of Teichm\"uller space, in which, by the seminal work of Goldman, the Weil--Petersson form is expressed by a natural symplectic form, called the Goldman form. 

We extend the work of Fillastre and Seppi, who have managed, using that approach and Stokes's Theorem, to provide a new proof of Wolpert's sum of cosines formula. Using the unique isometric symmetry on any hyperbolic pair of pants, we also introduce a way, given a pair of pants decomposition of a closed surface, to construct an associated linear involution on every tangent space of its Teich\"uller space.

That allows us to derive a new proof of Wolpert's Magic Formula and deduce a self-contained proof of the closedness of the Goldman form on Teichm\"uller space.
\end{abstract}

\maketitle

{\hypersetup{linkcolor=black} \setcounter{tocdepth}{1} \tableofcontents}

\begin{spacing}{1.23}

\section*{Introduction}

There are various models for the \emph{Teichm\"uller space} $\teich(\S)$ of an oriented closed connected surface $\S$ of genus $g\geq 2$. It is classically defined as the \emph{moduli space of marked complex structures on $\S$}, that is the space of complex structures on $\S$, up to biholomorphism isotopic to the identity. By the Uniformisation Theorem, it is also naturally realised as the space of hyperbolic metrics on $\S$ up to isometries isotopic to the identity, which is sometimes called \emph{Fricke space}. Considering holonomy maps of hyperbolic surfaces, it can also be identified with a connected component of the character variety
\begin{equation*}
\chi(\pi_1\S)=\text{Hom}\bp{\pi_1\S,\Isom^+(\H^2)} \, \big/ \, \Isom^+(\H^2) \, ,
\end{equation*}
consisting of conjugacy classes of \emph{Fuchsian}, i.e. discrete and faithful, representations of the fundamental group of $\S$ into the Lie group $\Isom^+(\H^2)$ of orientation-preserving isometries of the hyperbolic plane $\H^2$. In fact, there are two such components, corresponding to the choice of orientation on $\S$. Those are isomorphic to each other, and smooth manifolds of dimension $6g-6$. We we shall denote by $\chifd(\pi_1\S)$ the component obtained as the image of the holonomy maps of hyperbolic structures on $\S$, with respect to the fixed orientation of $\S$.

The \emph{Weil--Petersson metric} on $\teich(\S)$ is a Kähler metric introduced by Weil, using the Petersson inner product on forms on a Riemann surface \cite{Weil:1958}. It comes from a symplectic form $\wwp$ on $\teich(\S)$, which is classically defined using the complex quadratic differential model for the cotangent bundle of the Teichm\"uller space \cite[Section~7.7]{Hubbard:2006}.

\subsection*{Goldman form on the character variety}

In his seminal work \cite{Goldman:1984}, Goldman has used the \emph{Killing form} $\kil$ of the Lie algebra $\isom(\H^2)$ in order to introduce a natural symplectic form on the character variety $\chi(\pi_1\S)$. 

Indeed, if $\rho \in \text{Hom}(\pi_1\S,\Isom^+(\H^2))$, the fundamental group $\pi_1\S$ acts on $\isom(\H^2)$ via $\rho$ and the adjoint action of $\Isom(\H^2)$, by
\begin{equation*}
 \namelessfunction{\pi_1\S \times \isom(\H^2)}{\isom(\H^2)}{(\gamma,x)}{\Ad \rho(\gamma) \cdot x} \, .
\end{equation*}
Then, a classical model for the Zariski tangent space of the character variety $\chi(\pi_1\S)$ at $[\rho]$ is
\begin{equation*}
 \T_{[\rho]} \chi(\pi_1\S) = \Hisomeq \, ,
\end{equation*}
the first cohomology group of the \emph{group cohomology of $\pi_1\S$ with values in the Lie algebra $\isom(\H^2)$} (‘‘twisted'' by the group action described above). The \emph{Goldman form} is the skew-symmetric $2$-form on $\chi(\pi_1\S)$ defined by: for all $[\tau],[\tau'] \in \T_{[\rho]} \chi(\pi_1\S)=\Hisom $,
\begin{equation*}
\wg ( [\tau],[\tau'])= \kil ([\tau] \smile [\tau'])(R) \, ,
\end{equation*}
where $R$ is the fundamental class of the group homology group $H_2(\pi_1S , \Z)$, and $\smile$ is the cup product in group cohomology, here paired with the Killing form $\kil$ on $\isom(\H^2)$ (see Section~\ref{Sec Group cohomology model}).

Goldman proved that the $2$-form on $\chifd(\pi_1\S)$ defined in that way is symplectic (in particular closed) and actually coincides with the Weil--Petersson form through the natural identification between $\chifd(\pi_1\S)$ and the Teichm\"uller space $\teich(\S)$ given by holonomy maps of hyperbolic surfaces:

\begin{theorem*}[{Goldman \cite{Goldman:1984}}]
\label{theo intro Goldman Theorem}
The holonomy map $\Hol: (\teich (\S) ,\wwp) \to (\chifd(\pi_1\S), \wg)$ is a symplectomorphism.
\end{theorem*}

\subsection*{Results presented in this article} 

Several important results concerning the Weil--Petersson symplectic form are due to Wolpert \cite{Wolpert:1981,Wolpert:1983,Wolpert:1985} who managed, among other things, to show that the \emph{Fenchel--Nielsen coordinates} on the Teichm\"uller space, given by a \emph{pair of pants decomposition} of $\S$, are Darboux coordinates for the Weil--Petersson form. Surprisingly, that fact holds for any choice of a pair of pants decomposition, and is thus  often called \emph{Wolpert's Magic Formula}.

In this article, we present new proofs of some of Wolpert's classical results (Theorems~\ref{theo intro First Wolpert formula}, \ref{theo intro Second Wolpert formula}, \ref{theo intro w(dl,dl)=0} and Corollary~\ref{theo intro WP in FN}), as well as a new contribution (Theorem~\ref{theo intro existence involution on tangent space}).

We shall only use the character variety model and prove all those results for the Goldman form on $\chifd(\pi_1\S)$ (which coincides with the Weil--Petersson form by Goldman's Theorem). That differs from Wolpert's approach, which considers the classical model of the Teichmüller space $\teich(\S)$ as the moduli space of marked complex structures on $\S$, and thus sees the tangent space $\T_{[X]}\teich(\S)$ as the space of Beltrami differentials on the Riemann surface $(\S,X)$. The core of our approach is the use of the de Rham cohomology model for the tangent space of $\chifd(\pi_1\S)$. Then, our proofs only rely on hyperbolic geometry facts and on Stokes's Theorem.

Let us recall that one can define \emph{infinitesimal twists} on Teichm\"uller space in the following way. Given an oriented simple
closed geodesic curve $c$ on a closed hyperbolic surface $(\S,h)$, one can obtain a new marked hyperbolic surface by cutting $(\S,h)$ along $c$ and re-gluing after a twist of length $t \in \R$. Then, the infinitesimal twist $\twinf{c}$ is the tangent vector on Teichm\"uller space given by the derivative of that process at $t=0$.

We start by presenting Fillastre and Seppi's proof \cite{Fillastre-Seppi:2020} of the following formula of Wolpert \cite{Wolpert:1981}.

\begin{maintheorem}
\label{theo intro First Wolpert formula}
Let $\rho$ be a Fuchsian representation of $\pi_1\S$, and let $c$ and $c'$ be two oriented simple closed geodesic curves on the hyperbolic surface $\H^2/\rho(\pi_1\S)$. Then, the Goldman product of their associated infinitesimal twists $\twinf{c}$ and $\twinf{c'}$, seen as elements in the tangent space $\T_{[\rho]}\chifd(\pi_1\S)$, satisfies 
\begin{equation*}
\wg(\twinf{c},\twinf{c'})= 2 \sum_{q\in c \cap c'}\cos \theta_q \, ,
\end{equation*}
where for all $q\in c \cap c'$, $\theta_q$ is the oriented angle between the oriented closed geodesic curves $c$ and $c'$ at $q$.
\end{maintheorem} 

Actually, the work of Fillastre and Seppi \cite{Fillastre-Seppi:2020} concerns a generalised version of Wolpert's formula for balanced geodesic graphs on closed hyperbolic surfaces. In this article, we present a simplified version of their proof which only concerns closed geodesic curves. 

Using the same approach, we then proceed to give an original proof of a second formula of Wolpert \cite{Wolpert:1983} concerning the differential map of the  \emph{geodesic length function} $l_c$, which can be defined in the following way. If $c$ is a simple closed curve on $\S$ and $[\rho] \in \chifd(\pi_1\S)$, then there exists a unique simple closed geodesic curve on the hyperbolic surface $(\S,h) = \H^2/\rho(\pi_1\S)$ freely homotopic to $c$ (see Propositions~\ref{prop Unique geodesic loop} and \ref{prop geod is simple}). Then, $l_c[\rho]$ is the length of that closed geodesic curve.

\begin{maintheorem}
\label{theo intro Second Wolpert formula}
Let $c$ be a simple closed curve on $\S$. The vector field $\twinf{c}$ on $\chifd(\pi_1\S)$ given by the infinitesimal twist along the simple closed geodesic curve freely homotopic to $c$ and the geodesic length function $l_c$ on $\chifd(\pi_1\S)$ satisfy
\begin{equation*}
\wg(\twinf{c}, \cdot)= 2 \, \dif l_{c} \, .
\end{equation*}
\end{maintheorem}

We also give a completely new proof of the following formula concerning infinitesimal length variation in Fenchel--Nielsen coordinates of the Teichm\"uller space.

\begin{maintheorem}
\label{theo intro w(dl,dl)=0}
Let $(\gamma_i)_{1\leq i\leq 3g-3}$ be simple closed curves on $\S$ giving a pair of pants decomposition of $\S$ and Fenchel--Nielsen coordinates $(l_{\gamma_i},t_{\gamma_i})_{1\leq i\leq 3g-3}$ on $\chifd(\pi_1\S)$. Then, for all $1\leq i,j\leq 3g-3$,
\begin{equation*}
\wg(\partial_{l_{\gamma_i}}, \partial_{l_{\gamma_j}})= 0 \, .
\end{equation*}
\end{maintheorem}

The classical proof of Theorem~\ref{theo intro w(dl,dl)=0} \cite{Wolpert:1985,Hubbard:2006}, uses the closedness of the Goldman form (or of the Weil--Petersson form, depending on the model of the Teichm\"uller space one is considering) in order to reduce to case of points $[\rho_0] \in \chifd(\pi_1\S)$ having every twist coordinate $t_{\gamma_i}$ equal to $0$. In that case, the whole hyperbolic surface $(\S,h_0)=\H^2/\rho_0(\pi_1\S)$ is endowed with an orientation-reversing isometric symmetry $\sigma$. Then, that global isometric symmetry is used to naturally define a linear involution on $\T_{[\rho_0]} \chifd(\pi_1\S)$ in order to conclude (for a more detailed description of the classic proof, see the paragraph following Theorem~\ref{theo w(dl,dl)=0}).

Our proof of Theorem~\ref{theo intro w(dl,dl)=0} does not rely on the closedness of the Goldman form. Hence, it requires a new tool: a way to naturally associate each pair of pants decomposition with a linear involution ${}^*$ on the tangent space $\T_{[\rho]} \chifd(\pi_1\S)$ at any point $[\rho]$ (i.e. not necessarily admitting a global orientation-reversing isometric symmetry). Using restrictions of de Rham forms to hyperbolic pairs of pants and the orientation-reversing isometric symmetry on those, we prove the following.

\begin{maintheorem}
\label{theo intro existence involution on tangent space}
Let $\Gamma=(\gamma_i)_{1\leq i\leq 3g-3}$ be a family of simple closed curves on $\S$ giving a pair of pants decomposition of $\S$. Then, there is a smooth $(1,1)$-tensor $\mathrm{Inv^{\Gamma}}$ on $\chifd(\pi_1\S)$ such that
\begin{equation*}
\bp{\mathrm{Inv}^{\Gamma}}^2 = \Id \, 
\end{equation*}
and such that at every point $[\rho] \in \chifd(\pi_1\S)$, the linear involution $\mathrm{Inv}^{\Gamma}_{[\rho]}$, is ‘‘naturally induced'' by the unique orientation-reversing isometric symmetry on each hyperbolic pair of pants $\pant \subset \H^2/\rho(\pi_1\S)$ given by the pair of pants decomposition $\Gamma$ (see Theorem~\ref{theo existence involution on tangent space} for a more rigorous statement).

Moreover, for all $[\rho] \in \chifd(\pi_1\S)$ and $[\tau],[\tau'] \in \T_{[\rho]}\chifd(\pi_1\S)$, one has
\begin{equation*}
\wg(\mathrm{Inv}^{\Gamma}[\tau],\mathrm{Inv}^{\Gamma}[\tau']) = - \wg([\tau],[\tau']) \, .
\end{equation*}

\end{maintheorem}

Together, Theorems~\ref{theo intro First Wolpert formula}, \ref{theo intro Second Wolpert formula} and \ref{theo intro w(dl,dl)=0} imply \emph{Wolpert's Magic Formula} \cite{Wolpert:1985}.

\begin{maincorollary}[{Wolpert's Magic Formula}]
\label{theo intro WP in FN}
In any Fenchel--Nielsen coordinates $(l_{\gamma_i},t_{\gamma_i})_{1\leq i\leq 3g-3}$ on $\chifd(\pi_1\S)$, the Goldman form is expressed as 
\begin{equation*}
\wg = 2\sum_{i=1}^{3g-3} \dif t_{\gamma_i}\wedge \dif l_{\gamma_i} \, .
\end{equation*}
\end{maincorollary}

As explained, our proofs of Theorems~\ref{theo intro First Wolpert formula}, \ref{theo intro Second Wolpert formula}, \ref{theo intro w(dl,dl)=0} and Corollary~\ref{theo intro WP in FN}, contrarily to the classical one \cite{Wolpert:1985,Hubbard:2006}, do not rely on the closedness of the Goldman form. Thus, closedness comes as a direct corollary.

\begin{maincorollary}
\label{theo intro WP closed}
The Goldman form is a symplectic form on $\chifd(\pi_1\S)$.
\end{maincorollary}

\subsection*{Organisation of the article}

We start by giving some background concerning hyperbolic geometry, Teichm\"uller space and Fenchel--Nielsen coordinates in Section~\ref{Sec Background}.

In Section~\ref{Sec Group cohomology model}, we present the  group cohomology model for the tangent space to Teichm\"uller space, coming from its realisation as a component of a character variety, and we introduce the Goldman form. Section~\ref{Sec De Rham model} describes an equivalent de Rham cohomology model, providing the use of Stokes's Theorem.

In Section~\ref{sec twists}, we consider infinitesimal twists on the Teichm\"uller space, and describe their realisations in both descriptions of its tangent space. 

Sections~\ref{Sec Wolpert part 1} and~\ref{Sec Wolpert part 2} are respectively devoted to the proofs of Theorems~\ref{theo intro First Wolpert formula} and \ref{theo intro Second Wolpert formula}, using Stokes's Theorem in the de Rham cohomology model of tangent spaces. The first proof is a particular case of the work of Fillastre and Seppi \cite{Fillastre-Seppi:2020}, while the second one, although based on a similar approach, is new. 

In Sections~\ref{sec Induced de Rham forms and symmetries of pairs of pants} and \ref{sec FN involution} we use symmetries on hyperbolic pair of pants in order to associate each pair of pants decomposition with a family of involutions on each tangent space of $\chifd(\pi_1\S)$. That construction provides a proof for most of Theorem~\ref{theo intro existence involution on tangent space}.

In Section~\ref{Sec Formula for infinitesimal lengths},  we use these involutions in order to compute the Goldman product of two infinitesimal Fenchel--Nielsen length variations and thus prove Theorem~\ref{theo intro w(dl,dl)=0} without using the closedness of the Goldman form. Along the way, we complete the proof of Theorem~\ref{theo intro existence involution on tangent space} by showing that the involutions introduced in Section~\ref{sec FN involution} define a smooth tensor.

Finally, Section~\ref{Sec WG closedness} acts as a conclusion. We use all the previous results to get a new proof of Wolpert's Magic Formula for the Goldman form (Corollary~\ref{theo intro WP in FN}). Then, closedness of the Goldman form on Teichm\"uller space becomes a corollary (Corollary~\ref{theo intro WP closed}). 

\subsection*{Acknowledgment} I am deeply grateful to Andrea Seppi for introducing me to this subject, as well as for his regular help, guidance and advice. I also thank Fran\c{c}ois Fillastre for multiple remarks and advice concerning the writing of this article.

\section{Backgrounds}
\label{Sec Background}

Throughout this whole article we let $\S$ be an oriented closed connected surface of genus $g \geq 2$.

\subsection{On the hyperboloid model of the hyperbolic space}
\label{subsec hyperboloid model}

The \emph{Minkowski space} $\Rmink$ is the vector space $\R^3$ endowed with the following bilinear form of signature $(2,1)$:
\begin{equation*}
\bilmink{x}{y} = x_1y_1+x_2y_2-x_3y_3 \, .
\end{equation*}
A model for the \emph{hyperbolic plane} is the oriented Riemannian manifold $\H^2 = (\mathcal{H}^2, \hypcan)$, where
\begin{equation*}
\mathcal{H}^2 = \left\lbrace x \in \Rmink \st \bilmink{x}{x} = -1, x_3>0 \right\rbrace \, ,
\end{equation*}
is the \emph{upper hyperboloid} in $\Rmink$, endowed with the hyperbolic metric $\hypcan$ induced by the ambient bilinear form $\bilmink{\cdot}{\cdot}$. In that model, the set $\Isom(\H^2)$ of \emph{isometries} of $\H^2$ can then be seen as the Lie group $\O_+(2,1)$, consisting of matrices preserving both the bilinear form $\bilmink{\cdot}{\cdot}$ and $\mathcal{H}^2$. The Lie group $\Isom^+(\H^2)$ of \emph{orientation-preserving isometries} of $\H^2$ can then be seen as the identity component $\SO_0(2,1)$, and its Lie algebra $\isom(\H^2)$ is identified with $\so(2,1)$.

The \emph{Killing form} $\kil$ on the Lie algebra $\isom(\H^2)$ is a non-degenerate bilinear form. It is $\Ad$-invariant: for all $\phi \in \Isom(\H^2)$ and $x,y \in \isom(\H^2)$,
\begin{equation}
\label{eq Ad-inv of Kil}
\kil(\Ad \phi  \cdot x, \Ad \phi \cdot y) = \kil(x,y) \, .
\end{equation}
In the $\O_+(2,1)$ model of $\Isom(\H^2)$, it is expressed in the following way. For all $A,B \in \so(2,1)$,
\begin{equation}
\label{eq kil in SO}
\kil (A,B) \coloneqq \tr (AB) \, ,
\end{equation}
where $\tr$ denotes the trace of a matrix.

\subsection{On Teichm\"uller space and holonomy}
\label{Models for the Teichmuller space of a surface}

A model for the \emph{Teichm\"uller space} of $\S$ is the quotient 
\begin{equation*}
\teich (\S) \coloneqq \{ \text{hyperbolic metrics on } \S \} \, \big/ \, \diffo(\S) \, ,
\end{equation*}
where $\diffo(\S)$, the group of diffeomorphisms of $\S$ isotopic to the identity, acts by pull-back. It can be endowed with the structure of a smooth $(6g -6)$-dimensional manifold \cite{Tromba:1992}.

Let $\pi: \tilde{\S} \to \S$ be a universal covering of $\S$ and consider a hyperbolic metric $ h $ on the surface $\S$. As the surface $\S$ is closed, the Riemannian metric $ h $ on $\S$ is complete, and so too is the pull-back metric $\tilde{h} \coloneqq \pi^* h $ on $\tilde{\S}$. The oriented Riemannian surface $(\tilde{\S}, \tilde{h})$ is simply connected, complete, and has constant sectional curvature equal to $-1$. Thus it is isometric to the hyperbolic plane $\H^2$ \cite[Theorem 8.6.2]{Ratcliffe_19}. An orientation-preserving  isometry $(\tilde{\S},\tilde{h}) \to \H^2$ is called a \emph{developing map} and is unique up to post composition by an element of $\Isom^+(\H^2)$. Indeed, if $\dev, \dev': (\tilde{\S},\tilde{h}) \to \H^2$ are two developing maps, then $\dev' = \phi  \circ \dev $ where $\phi =\dev' \circ \dev^{-1}: \H^2 \to \H^2$ is an orientation-preserving isometry of $\H^2$.

Let $\dev: \tilde{\S} \to \H^2$ be a developing map of the hyperbolic surface $(\S, h )$. By definition of the pull-back metric $\tilde{h} = \pi^* h $ on $\tilde{S}$, the fundamental group $\pi_1\S$ acts by orientation-preserving isometries on $(\tilde{\S},\tilde{h})$. Hence, with any element $\gamma \in \pi_1\S$, one can associate the isometry $\rho(\gamma) \in \Isom^+(\H^2)$ defined by
\begin{equation*}
\rho(\gamma) \coloneqq \dev \circ \gamma \circ \dev^{-1} \, .
\end{equation*}
That defines a group representation $\rho: \pi_1\S \to \Isom^+(\H^2)$ called the \emph{holonomy map}. By definition, it is the unique representation for which the developing map $\dev$ is \emph{equivariant}, meaning that for all $\gamma \in \pi_1\S$, 
\begin{equation*}
\dev \circ \gamma = \rho(\gamma) \circ \dev \, .
\end{equation*}
As the developing map associated with the hyperbolic surface $(\S, h )$ is unique up to post-composition by an element of $\Isom^+(\H^2)$, so too is its holonomy map up to conjugacy by an element of $\Isom^+(\H^2)$: if $\dev$ and $\dev' = \phi  \circ \dev$ are two developing maps with associated holonomy $\rho$ and $\rho'$, then for all $\gamma \in \pi_1\S$,
\begin{equation*}
\rho'(\gamma) = \dev' \circ \gamma \circ (\dev')^{-1} = \phi \circ \dev \circ \gamma \circ \dev^{-1} \circ \phi ^{-1} = \phi \circ \rho(\gamma)\circ \phi ^{-1} \, .
\end{equation*}
Holonomy thus provides a map 
\begin{equation*}
\Hol: \teich (\S)\longrightarrow\chi \bp{ \pi_1\S , \Isom^+(\H^2)} \, ,
\end{equation*}
where the \emph{character variety} $\chi ( \pi_1\S)$ is defined as the quotient
\begin{equation*}
\chi(\pi_1\S) \coloneqq \text{Hom}\bp{\pi_1\S,\Isom^+(\H^2)} \, \big/ \, \Isom^+(\H^2) \, , 
\end{equation*}
with $\Isom^+(\H^2)$ acting by conjugation. That map actually induces a diffeomorphism
\begin{equation*}
\label{Hol}
\Hol: \teich (\S)\longrightarrow \chifd (\pi_1\S) \, , 
\end{equation*}
where $ \chifd(\pi_1\S)$ is one of the two connected components of the character variety $\chi(\pi_1\S)$ consisting of conjugacy classes of faithful and discrete representations \cite{Goldman_thesis}, called \emph{Fuchsian representations} .

\subsection{On Fenchel--Nielsen coordinates on Teichm\"uller space}
\label{subsec Fenchel_Nielsen coordinates}

\subsubsection{Geodesic representatives of free homotopy classes}

Let us recall a classical algebraic topology fact \cite[Section 1.1 Exercise 6]{Hatcher_02}: as the surface $\S$ is path-connected, there is a one-to-one correspondence between free homotopy classes of closed curves on $\S$ and conjugacy classes in its fundamental group $\pi_1\S$. From now on, let us denote by $[\gamma]$ the conjugacy class of an element $\gamma \in \pi_1\S$.

Let $ h $ be a hyperbolic metric on the closed surface $\S$ and consider a pair $(\dev,\rho)$ of developing and holonomy maps associated with $(\S, h)$. As $\S$ is a closed surface, the holonomy $\rho(\gamma)$ of a non-trivial $\gamma \in \pi_1\S$ is always a hyperbolic translation of the hyperbolic plane $\H^2$ \cite[Theorem 9.6.3]{Ratcliffe_19} and hence preserves a unique geodesic line. Let $C_{\gamma}$ be the unique oriented geodesic line of $\H^2$ along which $\rho(\gamma)$ is a positive hyperbolic translation. Pulling back $C_{\gamma}$ by $\dev$, we get an oriented complete geodesic curve in $(\tilde{\S},\tilde{h})$ preserved by the action of $\gamma$ and projecting onto an oriented closed geodesic curve $c^h_{\gamma}$ of $(\S, h )$. Note that if $\gamma,\gamma' \in \pi_1\S$ are conjugated, i.e. if there exist $\eta \in \pi_1\S$ such that $\gamma = \eta\gamma'\eta^{-1}$, then $C_{\gamma'} = \rho(\eta)C_\gamma$ is another lift of $c^h_\gamma$, so that $c^h_\gamma =c^h_{\gamma'}$. Thus, we have the following well-known proposition.

\begin{proposition}
\label{prop Unique geodesic loop}
Let $ h $ be a hyperbolic metric on the closed surface $\S$. Then, for each non-trivial $\gamma \in \pi_1\S$, there exists a unique closed $h$-geodesic curve in $(\S,h)$ having $[\gamma]$, the conjugacy class of $\gamma$ in $\pi_1\S$, as its free homotopy class. We shall denote it by $c^h_{\gamma}$.
\end{proposition}

A closed curve on $\S$ is called \emph{simple} if it is non-trivial and does not intersect itself, i.e. if, as a set, it can be seen as the image of an injective continuous map $c:\R/\Z \to \S$. One can prove the two following facts, concerning Proposition~\ref{prop Unique geodesic loop} and simple closed curves.

\begin{proposition}[{\cite[Proposition 3.3.9 point 1]{Hubbard:2006}}]
\label{prop geod is simple}
Let $\gamma \in \pi_1\S$ and let $ h $ be a hyperbolic metric on the closed surface $\S$. If there exist a simple closed curve on $\S$ having free homotopy class $[\gamma]$, then $c^h_\gamma$ is also simple.
\end{proposition}

\begin{proposition}[{\cite[Proposition 3.3.9 point 2]{Hubbard:2006}}]
\label{prop geod are disjoint}
Let $\gamma, \eta\in \pi_1\S$ and let $ h $ be a hyperbolic metric on the closed surface $\S$. If $[\gamma]$ and $[\eta]$ are two distinct free homotopy classes given by two disjoint simple closed curves on $\S$, then $c^h_\gamma$ and $c^h_\eta$ do not intersect.
\end{proposition}

Proposition~\ref{prop geod is simple} hence invites us to introduce the following notation. 

\begin{definition}[Simple element of the fundamental group]
\label{def simple}
An element $\gamma \in \pi_1\S$ is said to be \emph{simple} if there exists a simple closed curve on $\S$ having free homotopy class $[\gamma]$.
\end{definition}

Note that the closed surface $\S$ of genus $g \geq 2$ can be decomposed into $2g-2$ topological pairs of pants, by cutting it through $3g-3$ disjoint simple closed curves (see Figure~\ref{Pair of pants decomposition}). 

\begin{figure}[htbp]
\centering
\includegraphics{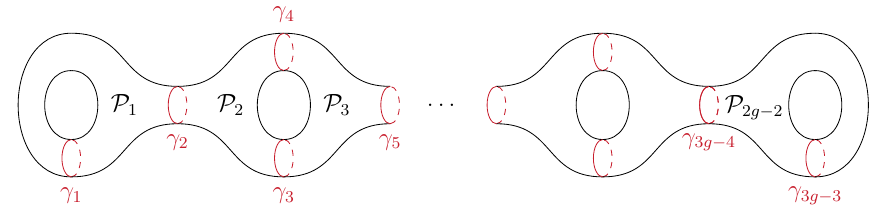}
\caption{A pair of pants decomposition of the genus $g$ surface.}
\label{Pair of pants decomposition}
\end{figure}

\begin{definition}[Pair of pants decomposition]
We shall say that $([\gamma_i])_{1\leq i \leq 3g-3}$ is a \emph{pair of pants decomposition of the surface $\S$} if it is a tuple of $3g-3$ pairwise distinct conjugacy classes of elements of $\pi_1\S$ given by the free homotopy classes of disjoint simple closed curves decomposing $\S$ into $2g-2$ topological pairs of pants. We shall also abusively say that “$\gamma \in \pi_1\S$ is in the pair of pants decomposition $([\gamma_i])_{1\leq i \leq 3g-3}$'' if $[\gamma] \in \{[\gamma_i]\}_{1\leq i \leq 3g-3}$.
\end{definition}

Let $([\gamma_i])_{1\leq i \leq 3g-3}$ be a pair of pants decomposition of $\S$. There are some natural coordinates, on the Teichm\"uller space of $\S$ associated with that pair of pants decomposition, called \emph{Fenchel--Nielsen coordinates}. They consist of $3g-3$ \emph{length parameters}, denoted by 
\begin{equation*}
(l_{\gamma_1},\dots ,l_{\gamma_{3g-3}}) : \teich(\S) \longrightarrow (\R^*_+)^{3g-3} \, 
\end{equation*}
and $3g-3$ \emph{twist parameters}, denoted by 
\begin{equation*}
(t_{\gamma_1},\dots ,t_{\gamma_{3g-3}}) : \teich(\S) \longrightarrow \R^{3g-3} \, ,
\end{equation*}
defined in the following way.

\subsubsection{Length parameters}

The length parameters are naturally defined in the following way. If $h$ is a hyperbolic metric on $\S$, for all $1\leq i \leq 3g-3$, the parameter $l_{\gamma_i}[h]$ is the length, in $\R_+^*$, of the closed geodesic curve $c^h_{\gamma_i}$ on $(\S,h)$ given by Proposition~\ref{prop Unique geodesic loop} (which only depends on the conjugacy class $[\gamma_i]$).

\subsubsection{Twist parameters}

The twist parameters must be defined more precociously using the additional datum of a multicurve $\Gamma'$ on $\S$. For a complete description of those parameters we refer to \cite[Section 7.6]{Hubbard:2006}. Still, let us quickly describe how them.

Let $h$ be a hyperbolic metric on $\S$. By Propositions~\ref{prop Unique geodesic loop}, \ref{prop geod is simple} and \ref{prop geod are disjoint}, one can decompose the hyperbolic surface $(\S, h )$ by cutting it along the disjoint simple closed $ h $-geodesic curves $(c^h_{\gamma_i})_{1 \leq i \leq 3g-3}$. Each connected component of that decomposition is a pair of pants endowed with a hyperbolic metric and having geodesic boundary. We shall refer to those as \emph{hyperbolic pairs of pants}.

When all the twists parameters $(t_{\gamma_i})_{1\leq i \leq 3g-3}$ of $[h] \in \teich(\S)$ are $0$, the geodesic arcs orthogonal to two boundary components of a pair of pants all align on the closed geodesic curves bounding pair of pants (see Figure~\ref{Twist figure}), producing a multicurve on $\S$ freely homotopic to $\Gamma'$. Otherwise, the parameter $t_{\gamma_i}$ in $(-l_{\gamma_i},l_{\gamma_i})$ corresponds to how much two hyperbolic pairs of pants inside $(\S,h)$ are twisted before being glued along $c^h_{\gamma_i}$, and are then naturally extended to $\R$ by considering that doing one full turn is equivalent to applying a Dehn twist along the curve $c^h_{\gamma_i}$.

The Fenchel--Nielsen parameters of a hyperbolic metric $h$ on $\S$ completely determine the point $[h]$ in $\teich(\S)$ because of the following classical hyperbolic geometry facts. 

Cutting a hyperbolic pair of pants along the three geodesic arcs orthogonal to two of the three boundary components (which are also the minimal length path between these components) divides it into two isometric \emph{right-angled hyperbolic hexagons}. Because of the well-known uniqueness, up to isometry, of a right-angled hyperbolic hexagon with fixed lengths of three alternated edges \cite[Theorem 3.5.13]{Ratcliffe_19}, every hyperbolic pair of pants in $(\S, h )$ given by the decomposition $([\gamma_i])_{1\leq i \leq 3g-3}$ is thus determined by the lengths parameters of its three geodesic boundaries. Then, the twist parameters indicate how one has to glue those hyperbolic pairs of pants in order to recover the marked hyperbolic surface $(\S, [h] )$.

\begin{figure}[htbp]
\centering
\includegraphics{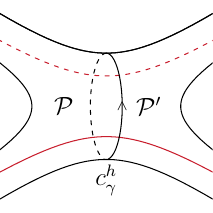} \qquad \qquad \qquad\includegraphics{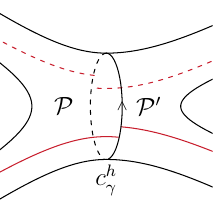} 
\caption{The gluing of two hyperbolic pairs of pants $\pant$ and $\mathcal{P'}$ along the curve $c^h_{\gamma}$ with respective twists $0$ and $ t_{\gamma} \in (0, l_{\gamma}) $; the red lines are the geodesic arcs orthogonal to boundary components.}
\label{Twist figure}
\end{figure}

The Fenchel--Nielsen parameters given by the pair of pants decomposition $([\gamma_i])_{1\leq i \leq 3g-3}$ (and the multicurve $\Gamma'$) are actually coordinates on the Teichm\"uller space $\teich(\S)$ parametrising it as $(\R_+^*)^{3g-3} \times \R^{3g-3}$ \cite[Theorem 7.6.3]{Hubbard:2006}.

From now on, using the holonomy diffeomorphism to identify $\teich(\S)$ and $\chifd(\pi_1\S)$, we shall also denote the Fenchel--Nielsen coordinates on $\chifd(\pi_1\S)$ by 
\begin{equation*}
(l_{\gamma_1},\dots ,l_{\gamma_{3g-3}},t_{\gamma_1},\dots ,t_{\gamma_{3g-3}}) :\chifd(\pi_1\S) \longrightarrow (\R_+^*)^{3g-3} \times \R^{3g-3} \, .
\end{equation*}

\begin{remark}
\label{remark change of coordinates multicurve}
The only non-canonical aspect of the twist coordinates is the arbitrary choice of the elements of $\teich(\S)$ having zero twist coordinates. That is why the multicurve $\Gamma'$ is needed. The influence of that choice is the following. Let $\bar\Gamma'$ be another multicurve on $\S$ defining, with the pair of pants decomposition $([\gamma_i])_{1\leq i \leq 3g-3}$, some Fenchel--Nielsen coordinates
\begin{equation*}
(\bar l_{\gamma_1},\dots ,\bar l_{\gamma_{3g-3}},\bar t_{\gamma_1},\dots ,\bar t_{\gamma_{3g-3}}) \, .   
\end{equation*}
By definition, for all $1\leq i \leq3g-3$, we have $l'_{\gamma_i}=l_{\gamma_i}$. Moreover, an hyperbolic surface $(S,h)$ having all twists parameters $t'_{\gamma_i}$ equal to $0$ has all the geodesic arcs orthogonal to two boundary components of a hyperbolic pair of pants aligning. Thus, by continuity of the length and twist coordinates \cite[Theorem 7.6.3]{Hubbard:2006}, for all $1\leq i \leq3g-3$, there exist $k_i \in \Z$ such that 
\begin{equation*}
    \begin{cases}
        \bar l_{\gamma_i}=l_{\gamma_i} \\
        \bar t_{\gamma_i}=t_{\gamma_i} + \frac{k_i}{2} l_{\gamma_i} \, .
    \end{cases}
\end{equation*}
In particular, we have that for all $1\leq i \leq3g-3$,
\begin{equation}
\label{eq change of FN coordinates}
    \begin{cases}
        \dif \bar l_{\gamma_i}= \dif l_{\gamma_i} \\
        \dif \bar t_{\gamma_i}=\dif t_{\gamma_i} + \frac{k_i}{2} \dif l_{\gamma_i} ,
    \end{cases}
    \quad \text{and} \quad 
    \begin{cases}
        \partial_{\bar l_{\gamma_i}}= \partial_{l_{\gamma_i}} - \frac{k_i}{2}\partial _{t_{\gamma_i}} \\
        \partial_{\bar t_{\gamma_i}}=\partial_{t_{\gamma_i}} \, .
    \end{cases}
\end{equation}
\end{remark}

\section{Tangent space of Teichm\"uller space via group cohomology}
\label{Sec Group cohomology model}

Using the diffeomorphism $ \Hol : \teich(\S) \to \chifd(\pi_1\S)$ introduced in Section~\ref{Models for the Teichmuller space of a surface}, one gets the following description for tangent spaces of the  Teichm\"uller space of $\S$. At a point $[\rho] \in \chi (\pi_1\S)$, the Zariski tangent space of the character variety can be seen as 
\begin{equation*}
\T_{[\rho]} \chi (\pi_1\S) \simeq \Hisomeq \, ,
\end{equation*}
the first cohomology group of the \emph{twisted group cohomology of $\pi_1\S$ with values in the Lie algebra $\isom(\H^2)$} \cite{Goldman:1984}. It is defined as the quotient 
\begin{equation*}
\Hisomeq \coloneqq \Zisomeq \, \big/ \, \Bisomeq \, ,
\end{equation*}
where:
\begin{itemize}
 \item $\Zisom$ is the space of \emph{cocycles} with respect to the adjoint action of $\rho$, that is the real vector space of maps $\tau: \pi_1\S \to \isom(\H^2)$ satisfying that for all $\gamma,\eta \in \pi_1\S$,
\begin{equation}
\label{Zeq}
\tau(\gamma\eta)=\Ad \rho (\gamma) \cdot \tau(\eta)+ \tau(\gamma) \, ,
\end{equation}
\item $\Bisom$ is the real vector space of \emph{coboundaries}, i.e. of cocycles of the form 
\begin{equation}
\label{Beq}
\tau(\gamma) = \Ad \rho (\gamma) \cdot x - x \, ,
\end{equation}
where $x \in \isom(\H^2)$.
\end{itemize}

We shall denote by $[\tau]$ the class in $\Hisom$ of an element $\tau \in \Zisom$. If $(\rho_s)_{s \in I}$ is a smooth family of Fuchsian representations of $\pi_1\S$ such that $\rho_0 = \rho$, we shall write 
\begin{equation*}
 [\tau] = \ddso [\rho_s] \in \T_{[\rho]}\chifd(\S) = \Hisomeq
\end{equation*}
when we consider $[\tau]$, the tangent vector at $[\rho_0]$ of the smooth path $[\rho_s]_{s\in I}$ in $\chifd(\S)$, seen as an element of $\Hisom$. In that case, a representative of $[\tau]$ is given by the cocycle
\begin{equation*}
 \tau = \ddso \rho_s \rho_0^{-1} \in \Zisomeq \, .
\end{equation*}

We can check that $\tau$ is indeed a cocycle, i.e. satisfies condition \eqref{Zeq}. As $(\rho_s)_s$ is a path of representations, for all $\gamma,\eta\in\pi_1\S$,
\begin{equation*}
 \ddso \rho_s(\gamma\eta) = \ddso\bp{\rho_s(\gamma)\rho_s(\eta)} = \rho_0(\gamma)\ddso\rho_s(\eta)+\ddso\rho_s(\gamma)\rho_0(\eta) \, ,
\end{equation*}
and thus 
\begin{align*}
\tau(\gamma\eta) &= \ddso \rho_s(\gamma\eta) \rho_0(\gamma\eta)^{-1} =\ddso \rho_s(\gamma\eta) \rho_0(\eta)^{-1}\rho_0(\gamma)^{-1} \\
& = \rho_0(\gamma)\ddso\rho_s(\eta)\rho_0(\eta)^{-1}\rho_0(\gamma)^{-1} + \ddso\rho_s(\gamma)\rho_0(\eta) \rho_0(\eta)^{-1}\rho_0(\gamma)^{-1}\\
&= \Ad \rho (\gamma) \cdot \tau(\eta)+ \tau(\gamma) \, .
\end{align*}

Moreover, if $(\rho_s)_{s\in I}$ and $(\mu_s)_{s\in I}$ are two smooth families of Fuchsian representations of $\pi_1\S$ conjugated to each other at every $s \in I$ and such that $\rho=\rho_0=\mu_0$, we can check that their associated cocycles differ by a coboundary, i.e. a cocycle satisfying \eqref{Beq}. Indeed, in that case, for all $s \in I$ and $\gamma\in \pi_1 \S$,
\begin{equation*}
\rho_s(\gamma) = \phi_s^{-1} \circ \mu_s(\gamma) \circ \phi_s \, ,
\end{equation*}
where $(\phi_s)_{s\in I}$ is a smooth path in $\Isom^+(\H^2)$ such that $\phi_0=\Id$. Setting
\begin{equation*}
 x \coloneqq \ddso \phi_s \in \isom(\H^2) 
\end{equation*}
and differentiating, we get that for all $\gamma\in \pi_1 \S$,
\begin{equation*}
\ddso \rho_s(\gamma) =\ddso \mu_s(\gamma) + \mu_0(\gamma)\cdot x -x \cdot \mu_0(\gamma) \, ,
\end{equation*}
Thus, as $\rho=\rho_0=\mu_0$, for all $\gamma\in \pi_1 \S$,
\begin{equation*}
\tau(\gamma) = \ddso \rho_s(\gamma) \rho_0(\gamma)^{-1} = \ddso \mu_s(\gamma) \mu_0(\gamma)^{-1} + \Ad \rho (\gamma) \cdot x - x \, .
\end{equation*}

Finally, let us note that a direct computation using the cocycle relation \eqref{Zeq} satisfied by every element of $\Zisom$, implies the following fact.

\begin{proposition}
\label{prop cocy conj}
Let $\tau \in \Zisom$. Then, for all $\eta, \zeta \in \pi_1\S$,
\begin{equation*}
 \tau(\zeta\eta\zeta^{-1}) = \tau(\zeta)- \Ad \rho(\zeta\eta\zeta^{-1}) \cdot \tau(\zeta) + \Ad \rho(\zeta) \cdot \tau(\eta) \, .
\end{equation*}
\end{proposition}

\subsection{The Goldman symplectic form}

In \cite{Goldman:1984}, Goldman has introduced the following symplectic form on the character variety $\chi (\pi_1\S)$.

\begin{definition}[Goldman form] 
The \emph{Goldman form} is the skew-symmetric $2$-form on $\chi (\pi_1\S)$ defined by: for all $[\tau],[\tau'] \in \T_{[\rho]} \chi (\pi_1\S) = \Hisom$,
\begin{equation*}
\wg ( [\tau],[\tau'])= \kil ([\tau] \smile [\tau'])(R) \, ,
\end{equation*}
where $R$ is the fundamental class of the group homology group $H_2(\pi_1\S , \Z) \simeq \Z$ which can be represented by an element of $\Z[\pi_1\S] \otimes \Z[\pi_1\S]$, and where
\begin{equation*}
 \smile \, : \Hisomeq \times \Hisomeq \longrightarrow H^2_{\Ad \rho}\bp{\pi_1\S,\isom(\H^2)}
\end{equation*} 
is the cup product in group cohomology, here paired with the Killing form $\kil$ on $\isom(\H^2)$, i.e. for all $\gamma,\eta \in \pi_1\S$, 
\begin{equation*}
\kil([\tau] \smile [\tau'])(\gamma \otimes \eta) \coloneqq \kil\bp{\tau(\gamma), \tau'(\eta)}-\kil\bp{\tau(\eta), \tau'(\gamma)} \, .
\end{equation*}
\end{definition}

The closedness of the Goldman symplectic form is not obvious. In \cite{Goldman:1984}, Goldman proves it, following the work of Atiyah and Bott \cite{Atiyah_Bot_83}, by using various tools and arguments from algebraic topology and gauge theory. We shall prove that it is closed on the component $\chifd(\pi_1\S)$ in Section~\ref{Sec WG closedness}. Hence, throughout the rest of the article, let us simply write that $\wg$ is the Goldman form on $\chi(\pi_1\S)$, rather than the Goldman symplectic form.

\section{Tangent space of Teichm\"uller space via de Rham cohomology}
\label{Sec De Rham model}

Let us now introduce another equivalent model for tangent spaces of $\chifd(\pi_1\S)$, using \emph{de Rham cohomology}. It will be very useful to us for two reasons. First, it gives access to \emph{Stokes's Theorem}. Moreover, we shall see that it allows to get representatives of tangent vectors in $\T_{[\rho]}\chifd(\pi_1\S)$ which are more intuitively and geometrically related to the variation of the associated hyperbolic surface $\H^2/\rho(\pi_1\S)$.

\subsection{A flat bundle and its de Rham cohomology}
\label{subsec A flat bundle}

Let $\rho$ be Fuchsian representation of $\pi_1\S$. Consider the hyperbolic surface $(\S,h) = \H^2/\rho(\pi_1\S)$ with universal covering $\pi:\tilde{\S}  \to \S$. 

The trivial vector bundle 
\begin{equation*}
   E \coloneqq \tilde{\S} \times \isom(\H^2)
\end{equation*}
is a flat vector bundle over $\tilde{\S}$. The composition of $\rho$ with the adjoint representation of $\Isom(\H^2)$ gives the following linear representation
\begin{equation*}
   \Ad \rho :\pi_1\S \longrightarrow \mathrm{Aut}\bp{\isom(\H^2)} \, .
\end{equation*}
Thus, we can consider the flat vector bundle $\bundle$ over $\S$ defined by
\begin{equation*}
\bundle = \bp{\tilde{\S} \times \isom(\H^2) } \, \big/ \, \pi_1\S \, ,
\end{equation*}
where $\pi_1\S$ acts diagonally on the product $\tilde{\S} \times \isom(\H^2)$ by the following action
\begin{equation*}
\gamma \cdot (p,x) \coloneqq \bp{\gamma \cdot p, \Ad \rho(\gamma)\cdot x} \, .
\end{equation*}

\begin{remark}
\label{remark identification uS and H2}
As $\pi : \H^2 \to \H^2/\rho(\pi_1\S)$ is a universal covering of $(\S,h)$, the universal cover $\tilde{\S}$ and $\H^2$ are identified (that is the same as using the developing map associated with the holonomy map $\rho$, see subsection~\ref{Models for the Teichmuller space of a surface}). Thus, the action of any element $\gamma \in \pi_1\S$ on $\tilde{S}$ is given by the action of $\rho(\gamma)$ on $\H^2$. Still, we choose to continue to denote the action of $\pi_1\S$ on the universal cover $\tilde{S}$ by $\gamma \cdot p$, rather than using $\rho$ and $\H^2$, in order to keep track of the identification $\tilde{\S} = \H^2$, and to quickly understand which action of $\pi_1\S$ is concerned in a formula.
\end{remark}

As $\bundle$ is a flat vector bundle over $\S$, the exterior covariant derivative $\dif$ on $\bundle$-valued forms induced by the flat connection on $\bundle$ satisfies
\begin{equation*}
    \dif \circ \dif =0 \, , 
\end{equation*}
and one can then consider the \emph{de Rham cohomology of $\S$ with values in $\bundle$} \cite[2.21]{Deligne_70}. The first de Rham cohomology group $\HdR$ is the quotient
\begin{equation*}
\HdR \coloneqq \ZdR \, \big/ \, \BdR \, ,
\end{equation*}
where $\ZdR$ is the vector space of closed $\bundle$-valued $1$-forms, and $\BdR$ is the subspace of exact $\bundle$-valued $1$-forms. We shall denote by $\dR{\alpha}$ the class in $\HdR$ of an element $\alpha \in \ZdR$.

By definition of $\bundle$, each $\alpha \in \ZdR$ is induced by a closed $\Ad \rho $-equivariant $\isom(\H^2)$-valued $1$-form $\alphat$ on $\tilde{\S}$, i.e. satisfying: for all $\gamma \in \pi_1\S$,
\begin{equation}
\label{eq ad equiv 1-form}
\alphat \circ \dif \gamma = \Ad \rho(\gamma)\cdot \alphat \, .
\end{equation}
An exact form $\alpha = \dif f \in \BdR$ is then induced by an exact $\isom(\H^2)$-valued $1$-form $\alphat =\dif \tilde{f}$ on $\tilde{\S}$, where $\tilde{f}$ is an $\Ad \rho $-equivariant $\isom(\H^2)$-valued map, i.e. satisfying: for all $\gamma \in \pi_1\S$.
\begin{equation}
\label{eq ad equiv map}
\tilde{f} \circ \gamma = \Ad\rho(\gamma) \cdot \tilde{f} \, ,
\end{equation}
In that case, $f$ is the section of $\bundle$ induced by $\tilde{f}$.

\subsection{Equivalence with group cohomology}
\label{subsec de Rham map}

The first cohomology groups $\Hisom$ and $\HdR$ are isomorphic because of classical cohomology results. Indeed, by de Rham's Theorem for local systems, $\HdR$ is isomorphic to $H^1(\S,\bundle)$, the cohomology of $\S$ with local coefficients in $\bundle$ \cite[Theorem 2.23]{Deligne_70}. Moreover, by the work of Eilenberg \cite[Theorem 28.1]{Eilenberg:1947}, as $\S$ is an Eilenberg--MacLane space of type $K(\pi_1\S,1)$, the first cohomology group $H^1(\S,\bundle)$ is isomorphic to the first group cohomology group $\Hisom$. 

Let us introduce an explicit natural isomorphism between $\HdR$ and $\Hisom$: the \emph{de Rham map}. Let $p_0 \in \tilde{\S}$ be an arbitrary base point. Given any $\alpha \in \ZdR$, let $\tau_{\alpha} : \pi_1\S \to \isom(\H^2)$ be the map defined by 
\begin{equation*}
\tau_\alpha (\gamma) \coloneqq \int_{p_0}^{\gamma \cdot p_0}\alphat \, ,
\end{equation*}
where $\alphat$ is the closed $\Ad \rho$-equivariant $\isom(\H^2)$-valued 1-form on $\tilde{\S}$ inducing $\alpha$, and the integration is along any smooth path from $p_0$ to $p_1$ in $\tilde{\S}$ (the value of that integral does not depend on the choice of such a path because the form $\alphat$ is closed and $\tilde{\S}$ is simply connected).

\begin{lemma}
For all $\alpha \in \ZdR$, the map $\tau_\alpha : \pi_1\S \to \isom(\H^2)$ satisfies the group cocycle condition~\eqref{Zeq}.
\end{lemma}
\begin{proof}
For all $\gamma,\eta\in \pi_1\S$, using the $\Ad \rho $-equivariance~\eqref{eq ad equiv 1-form} of $\alphat$, we get
\begin{equation*}
\tau_\alpha(\gamma \eta) = \int_{p_0}^{\gamma\eta \cdot p_0}\alphat =\int_{p_0}^{\gamma \cdot p_0}\alphat + \int_{\gamma \cdot p_0}^{\gamma\eta \cdot p_0}\alphat = \tau_\alpha(\gamma) + \int_{p_0}^{\eta \cdot p_0} (\alphat \circ \dif \gamma) = \tau_\alpha(\gamma) + \Ad \rho(\gamma)\cdot \tau_\alpha(\eta)\, . \qedhere
\end{equation*}
\end{proof}

\begin{lemma}
\label{lem change of base point}
For all $\alpha \in \ZdR$, the choice of the base point $p_0$ in the definition of $\tau_\alpha$ only changes $\tau_\alpha$ by a coboundary. 
\end{lemma}
\begin{proof}
Let $p_0' \in \tilde{\S}$. Using the $\Ad \rho $-equivariance~\eqref{eq ad equiv 1-form} of $\alphat$, we get
\begin{align*}
\tau'_\alpha (\gamma)  \coloneqq \int_{p_0'}^{\gamma \cdot p_0'}\alphat & = \int_{p'_0}^{p_0} \alphat +\int_{p_0}^{\gamma \cdot p_0} \alphat + \int_{\gamma \cdot p_0}^{\gamma \cdot p_0'} \alphat = - \int _{p_0}^{p_0'} \alphat + \int_{p_0}^{ p_0'} (\alphat \circ \dif \gamma)  + \int_{p_0}^{\gamma \cdot p_0} \alphat\\
& = \tau_\alpha (\gamma) + \Ad \rho (\gamma) \cdot x -x \, ,
\end{align*}
where 
\begin{equation*}
    x= \int _{p_0}^{p_0'} \alphat \in \isom(\H^2)\, . \qedhere
\end{equation*}
\end{proof}

\begin{lemma}
If $\alpha \in \BdR$, then  $\tau_\alpha \in \Bisom$.
\end{lemma}
\begin{proof}
Let $\alpha \in \BdR$. That is, $\alpha$ is exact and there is a section $f$ of $\bundle$ such that $\alpha = \dif f$. Then, $\alphat = \dif \tilde{f}$, where $\tilde{f} : \tilde{\S} \to \isom(\H^2)$ satisfies that for all $\gamma \in \pi_1\S$,
\begin{equation*}
\tilde{f} \circ \gamma= \Ad \rho(\gamma)\cdot\tilde{f} \, .
\end{equation*}
Hence, by the Fundamental Theorem of Calculus, for all $\gamma\in \pi_1\S$,
\begin{equation*}
\tau_\alpha (\gamma) = \int_{p_0}^{\gamma \cdot p_0}\dif \tilde{f} = \tilde{f}(\gamma \cdot p_0) - \tilde{f}(p_0) = \Ad \rho (\gamma) \cdot \tilde{f}(p_0) -\tilde{f}(p_0) \, ,
\end{equation*}
and $\tau_\alpha$ is a coboundary.
\end{proof}

Because of the three lemmas we have just stated, the following linear map is well-defined and independent of the choice of a point $p_0 \in \tilde{\S}$:
\begin{equation}
\label{eq Psi}
\function{\Psi_\rho}{\HdR}{\Hisomeq}{\dR{\alpha}}{[\tau_\alpha]} \, .
\end{equation}

The map $\Psi_\rho$ is actually the natural vector space isomorphism, called the  \emph{de Rham map}, given by the classical cohomology results discussed above (de Rham's Theorem and Eilenberg's Theorem), so that we have the following.

\begin{theorem}
The map $\Psi_\rho : \HdR \to \T_{[\rho]}\chifd(\pi_1\S)$ defined by \eqref{eq Psi} is a vector space isomorphism.
\end{theorem}

Let us only prove the injectivity of $\Psi_\rho$, as we shall later re-use the arguments from that proof. First, here is a small useful fact.

\begin{lemma}
\label{lem good value dR form}
Let $\alpha \in \ZdR$ and $p_0 \in \tilde \S$. If $\tau \in \Zisom$ is such that $\Psi_\rho\dR{\alpha}=[\tau]$, then there exists a representative $\alpha'$ of $\dR{\alpha}$ such that for all $\gamma \in \pi_1\S$,
\begin{equation*}
 \int_{p_0}^{\gamma \cdot p_0} \alphat' = \tau(\gamma) \, .
\end{equation*}
\end{lemma}
\begin{proof}
As $\Psi_\rho\dR{\alpha}=[\tau]$, there exists $x \in \isom(\H^2)$ such that for all $\gamma\in \pi_1\S$,
\begin{equation*}
\int_{p_0}^{\gamma \cdot p_0} \alphat = \tau(\gamma)+ \Ad \rho (\gamma) \cdot x - x \, .
\end{equation*}
Let $f$ be any global section of $\bundle$ such that its lift $\tilde{f}: \tilde{\S} \to \isom(\H^2)$ takes the value $\tilde{f}(p_0)=x$. Then, setting $\alpha'= \alpha -\dif f$, we get that for all $\gamma \in \pi_1\S$,
\begin{equation*}
 \int_{p_0}^{\gamma \cdot p_0} \alphat' =\int_{p_0}^{\gamma \cdot p_0} (\alphat - \dif \tilde{f})= \int_{p_0}^{\gamma \cdot p_0}  \alphat - \tilde f(\gamma\cdot p_0) + \tilde f(p_0) = \int_{p_0}^{\gamma \cdot p_0} \alphat - \Ad \rho(\gamma) \cdot x +x = \tau(\gamma)\, . \qedhere
\end{equation*}
\end{proof}

\begin{proposition}
\label{prop Psi is injective}
The map $\Psi_\rho : \HdR \to \T_{[\rho]}\chifd(\pi_1\S)$ defined by \eqref{eq Psi} is injective.
\end{proposition}
\begin{proof}
Let $\dR{\alpha} \in \HdR$ such that $\Psi_\rho\dR{\alpha}=0$. Let us show that $\dR{\alpha} =0$.

Let $p_0 \in \tilde{\S}$. By Lemma~\ref{lem good value dR form}, up to a good choice of representative $\alpha \in \ZdR$, we have that for all $\gamma \in \pi_1\S$, 
\begin{equation*}
\int_{p_0}^{\gamma\cdot p_0} \alphat = 0 \, .
\end{equation*}
Then, the smooth map $F: \tilde{\S} \to \isom(\H^2)$ defined by
\begin{equation*}
F(p) = \int_{p_0}^{p} \alphat \, 
\end{equation*}
clearly satisfies $\alphat = \dif F$. Moreover, it is $\Ad \rho$-equivariant: for all $p\in \S$ and $\gamma \in \pi_1\S$,
\begin{equation*}
F( \gamma \cdot p) = \int_{p_0}^{\gamma \cdot p} \alphat = \int_{p_0}^{\gamma \cdot p_0} \alphat + \int_{\gamma \cdot p_0}^{\gamma \cdot p} \alphat = \int_{\gamma \cdot p_0}^{\gamma \cdot p} \alphat = \Ad \rho (\gamma) \int_{p_0}^{p} \alphat=\Ad \rho(\gamma)\cdot F(p) \, .
\end{equation*}
Hence, $F$ induces a smooth section $f$ of $\bundle$ such that $\alpha = \dif f$. That is $\dR{\alpha} = 0$.
\end{proof}

\subsection{The Goldman form in de Rham cohomology}
\label{subsec wdR}

Note that the Killing form $\kil$ on $\isom(\H^2)$ provides a non-degenerate bilinear form on each fibre of $\bundle$, that we shall abusively still denote by $\kil$. Indeed, two elements $X$ and $Y$ of the fibre of $\bundle$ at a point $q \in \S$ are classes $X=[p,x]$ and $Y =[p,y]$, where $p\in \pi^{-1}(q) \subset \tilde{\S}$ and $x,y \in \isom(\H^2)$. Then, because of $\Ad$-invariance property of $\kil$, see~\eqref{eq Ad-inv of Kil}, the quantity
\begin{equation*}
\kil(X,Y) = \kil([p,x],[p,y]) \coloneqq \kil(x,y)
\end{equation*}
is well-defined. Hence, there is a natural skew-symmetric bilinear map on the vector space $\HdR$ defined by 
\begin{equation*}
\wdR(\dR{\alpha},\dR{\alpha'}) = \int_\S \kil(\alpha \wedge \alpha') \, ,
\end{equation*}
where $\kil(\alpha \wedge \alpha')$ is the wedge product of $\alpha$ and $\alpha'$ paired with the bilinear form $\kil$ on $\bundle$ introduced in subsection~\ref{subsec A flat bundle}, i.e. for every smooth vector fields $X,Y$ on $\S$,
\begin{equation*}
\kil(\alpha \wedge \alpha')(X,Y) = \kil\bp{\alpha(X), \alpha'(Y)}-\kil\bp{\alpha(Y), \alpha'(X)} \, .
\end{equation*}

\begin{remark}
The bilinear map $\wdR$ is well-defined on the cohomology group $\HdR$ because $\S$ is closed. Indeed, as $\S$ is compact the integral above is finite. Moreover, it does not depend on the choice of representatives of the classes $\dR{\alpha}$ and $\dR{\alpha'}$. Indeed, let $\alpha, \alpha' \in \ZdR$ and let $f$ be a smooth section of $\bundle$. As $\alpha$ is closed (i.e. $\dif \alpha=0$), one easily checks that 
\begin{equation*}
\kil(\alpha \wedge \dif f) = \dif \bp{\kil(\alpha,f) } \, . 
\end{equation*}
Hence, using Stokes's theorem and that $\partial\S = \emptyset$, we get
\begin{equation*}
\int_\S \kil\bp{\alpha \wedge (\alpha' + \dif f)} = \int_\S \kil(\alpha \wedge \alpha') + \int_{\partial\S}\kil(\alpha,f) = \int_\S \kil(\alpha \wedge \alpha') \,.
\end{equation*}
\end{remark}

\begin{proposition}
\label{prop wg=wdR}
The isomorphism $\Psi_\rho: \HdR,\to \Hisom$ is such that $\Psi_\rho^*\wdR = \wg$.
\end{proposition}

\begin{proof}
In de Rham cohomology, the fundamental class of the closed oriented manifold $\S$ represents integration over $\S$. Moreover, the de Rham isomorphism sends wedge products onto cup products \cite[Theorem 2.23]{Deligne_70}. Hence, we have
\begin{equation*}
\wg \bp{\Psi_\rho\dR{\alpha} , \Psi_\rho\dR{\alpha'}} = \kil\bp{\Psi_\rho\dR{\alpha} \smile \Psi_\rho\dR{\alpha'}}(R)  = \int_\S \kil(\alpha \wedge \alpha')  = \wdR(\dR{\alpha}, \dR{\alpha'}) \, . \qedhere
\end{equation*}
\end{proof}

\section{Twists on Teichm\"uller space}

\label{sec twists}

In this section, we shall describe some particular tangent vectors of $\chifd(\pi_1\S)$: \emph{infinitesimal twists along closed geodesic curves}. 

\begin{definition}[Infinitesimal twist]
\label{def infinitesimal twist}
Let $\rho$ be a Fuchsian representation of $\pi_1\S$ and consider the hyperbolic surface $(\S,h)= \H^2/\rho(\pi_1\S)$. Consider a simple element $\gamma \in \pi_1\S$ (Definition~\ref{def simple}) and the associated oriented simple closed geodesic curve $c^h_{\gamma}$ on $(\S,h)$ given by Proposition~\ref{prop geod is simple}.

The \emph{infinitesimal twist along $\gamma$ at $[\rho]$} is the element $\twinf{\gamma} \in \T_{[\rho]}\chifd(\pi_1\S)$ representing the infinitesimal variation at $t=0$ of the holonomy maps $[\rho_t]_{t\in \R}$ of the family $(\S,h_t)_{t\in\R}$ of hyperbolic surfaces obtained by cutting $(\S,h)$ along the simple closed oriented geodesic curve $c^h_{\gamma}$ and gluing it back with a twist $t\in \R$ on the right side of the oriented geodesic curve $c^h_{\gamma}$.
\end{definition}

Let us now describe infinitesimal twists in both the group and de Rham cohomology models of the tangent space of $\chifd(\pi_1\S)$.

\subsection{Infinitesimal twists in group cohomology}

\label{subsec inf translations and twists}

First, let us introduce the following.

\begin{definition}[Infinitesimal translation along a geodesic line]
\label{def infinitesimal translation}
Let $C$ be an oriented geodesic line of $\H^2$. We shall denote by $T_t(C) \in \Isom^+(\H^2)$ the hyperbolic translation by $t\in \R$ along $C$. The \emph{infinitesimal translation along $C$} is the generator $\tH(C) \in \isom(\H^2)$ of the $1$-parameter subgroup $\{ T_t(C) \ \vert \ t \in \R \} < \Isom^+(\H^2)$. It can be expressed as 
\begin{equation*}
\tH(C) \coloneqq \ddto T_t(C) \, .
\end{equation*}
\end{definition}

Let $\rho$ be a Fuchsian representation of $\pi_1\S$, consider the hyperbolic surface $(\S,h) =\H^2/\rho(\pi_1\S)$ and a simple element $\gamma \in \pi_1\S$. The transformation of the holonomy group $\rho(\pi_1\S)$ by a twist along the oriented simple closed geodesic curve $c^h_\gamma$ and its associated cocycle can be algebraically described in the following way \cite{Johnson_Millson_87}. Let $C_\gamma$ be the oriented geodesic line of $\H^2$ such that $\rho(\gamma)$ is a positive hyperbolic translation $T_{l}(C_\gamma) \in \Isom^+(\H^2)$ with $l>0$. 

\vspace{\baselineskip}

\underline{If $c^h_{\gamma}$ separates $(\S,h) = \H^2/\rho(\pi_1\S)$ into two connected component $\S_0$ and $\S_1$}, respectively on the left and right side of the oriented geodesic curve $c^h_{\gamma}$, then $\pi_1\S$ is the amalgamated product 
 \begin{equation*}
 \pi_1\S = \pi_1 \S_0 *_{\langle\gamma\rangle} \pi_1 \S_1 \, .
 \end{equation*}
A holonomy map $\rho_t$ of the hyperbolic surface obtained from twisting $(\S,h)$ by $t\in\R$ along $c^h_\gamma$ can be expressed as 
\begin{equation*}
 \begin{cases}
 \rho_t \vert_{\pi_1 \S_0 } = \rho \\
 \rho_t \vert_{\pi_1 \S_1 } = T_t(C_\gamma) \cdot \rho \cdot T_{-t}(C_\gamma) \, . 
 \end{cases}
\end{equation*}
\begin{remark}
That is indeed well-defined, as for all $t\in\R$,
\begin{equation*}
 T_t(C) \cdot \rho(\gamma) \cdot T_{-t}(C_\gamma) = T_t(C_\gamma) \cdot T_{l}(C_\gamma) \cdot T_{-t}(C_\gamma) = T_{l}(C_\gamma) = \rho(\gamma)
\end{equation*}
so that those two representations agree on $\langle \gamma\rangle$, the group on which the product defining $\pi_1\S$ is amalgamated. 
\end{remark} 

A representative of 
\begin{equation*}
 \twinf{\gamma}=[\tau] = \ddto [\rho_t] \in \T_{[\rho]}\chifd(\S) = \Hisom
\end{equation*}
is given by the cocycle
\begin{equation*}
 \tau = \ddto \rho_t \rho^{-1} \in \Zisomeq \, ,
\end{equation*}
which is defined by
\begin{equation} 
\label{eq inf twist separating}
 \begin{cases}
\tau(\eta)=0 & \text{if } \eta \in \pi_1 \S_0 \, , \\
\tau(\eta)=\tH(C_\gamma) - \Ad \rho(\eta) \cdot \tH(C_\gamma) & \text{if } \eta \in \pi_1 \S_1 \, .
 \end{cases}
\end{equation}

\vspace{\baselineskip}

\underline{If $c^h_{\gamma}$ does not separate $(\S,h) = \H^2/\rho(\pi_1\S)$}, setting $\S' = \S \setminus c^h_{\gamma}$ and letting $ \langle G \, \vert \, R \rangle$ be a finite presentation of $\pi_1 \S'$, the group $\pi_1\S$ is the HNN extension
 \begin{equation*}
 \pi_1\S = \pi_1 \S' *_{j} = \left\langle G, \zeta \st R, \zeta^{-1}j_0(\gamma)\zeta = j_1(\gamma)\right\rangle \, ,
 \end{equation*}
where the inclusions $j_0(\gamma),j_1(\gamma) \in \pi_1 \S'$ come from the two different inclusions of the closed curve $c^h_\gamma$ as a boundary of surface $\S'$. We may assume that moreover, $\rho(j_0(\gamma))= \rho(\gamma)$ and that these inclusions fit with the orientation of $\gamma$ in the following sense: 
the surface $\S'$ is on the left of $j_{0}(\gamma)$ and on the right of $j_1(\gamma)$.

A holonomy map $\rho$ of the hyperbolic surface obtained from twisting $(\S,h)$ by $t\in\R$ along $c^h_\gamma$ can be expressed as 
\begin{equation*}
 \begin{cases}
 {\rho_t}\vert_{\pi_1 \S'} = \rho \\
 \rho_t(\zeta) = T_t(C_\gamma) \cdot \rho(\zeta) \, .
 \end{cases}
 \end{equation*}

\begin{remark}
That is indeed well-defined, as for all $t\in \R$, 
\begin{equation*}
 \rho_t(\zeta)^{-1} \rho\bp{j_0(\gamma)} \rho_t(\zeta) = \rho(\zeta)^{-1} T_{-t}(C_\gamma) T_{l} (C_\gamma)T_{t}(C_\gamma) \rho(\zeta) =\rho(\zeta)^{-1} T_{l}(C_\gamma) \rho(\zeta) =\rho\bp{j_1(\gamma)} \, .
\end{equation*}
\end{remark}

A representative of 
\begin{equation*}
 \twinf{\gamma}=[\tau] = \ddto [\rho_t] \in \T_{[\rho]}\chifd(\S) = \Hisom
\end{equation*}
is thus given by the cocycle
\begin{equation*}
 \tau = \ddto \rho_t \rho^{-1} \in \Zisomeq \, ,
\end{equation*}
which is defined by
\begin{equation}
\label{eq inf twist non-separating}
 \begin{cases}
 \tau(\eta) = 0 & \text{if} \ \eta \in \pi_1 \S' \, , \\
 \tau(\zeta) = \tH(C_\gamma) \, . 
 \end{cases}
\end{equation}

Let us notice that the group cocycle representing infinitesimal twist variation along $\gamma$ satisfies the following.

\begin{proposition}
\label{prop restriction of inf twist is coboundary}
Let the setting be as above, let $\tau\in \Zisom$ be a group cocycle such that $\twinf{\gamma} =[\tau]$ and consider a connected subsurface $\Sigma$ of the hyperbolic surface $(\S,h)$ such that the  the subgroup inclusion $\pi_1\Sigma \to \pi_1\S$ is injective. If $c^h_\gamma \cap \Sigma = \emptyset$, then the restriction of $\tau$ to $\pi_1\Sigma$ is the restriction of a coboundary, i.e. there exists $x \in \isom(\H^2)$ such that 
\begin{equation*}
    \forall \eta \in \pi_1\Sigma \, , \qquad \tau(\eta) = \Ad \rho(\eta) \cdot x-x   \, .
\end{equation*}
\end{proposition}
\begin{proof}
If $c^h_{\gamma}$ separates $(\S,h)$ into two connected component $\S_0$ and $\S_1$, then either $\Sigma\subseteq\S_0$ or $\Sigma\subseteq\S_1$. Then, it is direct consequence of the fact that, up to a coboundary, $\tau$ is expressed by \eqref{eq inf twist separating}.

If $c^h_{\gamma}$ does not separates $(\S,h)$, then $\Sigma\subseteq\S' = \S\setminus c^h_{\gamma}$. Then, it is direct consequence of the fact that, up to a coboundary, $\tau$ is expressed by \eqref{eq inf twist non-separating}.
\end{proof}

\subsection{Infinitesimal twists in de Rham cohomology}
\label{subsec Infinitesimal twists in de Rham cohomology}

Let $\rho$ be a Fuchsian representation of $\pi_1\S$ and consider the hyperbolic surface $(\S,h)= \H^2/\rho(\pi_1\S)$ with universal covering $\pi:\tilde{\S} \to \S$. Given a simple element $\gamma \in \pi_1\S$, let us construct an explicit $1$-form $\alpha_\gamma \in \ZdR$ such that $\Psi_{\rho}(\dR{\alpha_\gamma})$ is $\twinf{\gamma}$, the infinitesimal twist along $\gamma$ at $[\rho]$.

Let us define an $\isom(\H^2)$-valued $1$-form $\alphat_\gamma$ on $\tilde \S  $ by, for all $(p,v) \in \T\tilde{S}$
\begin{equation*}
\alphat_\gamma (p,v) =
\begin{cases}
 \tH(C_{\gamma,p})\, \dif(f\circ \tilde{\delta}) & \text{if } p\in \pi^{-1}(N^\gamma_\varepsilon) \, ,  \\
0 & \text{otherwise,}
\end{cases}
\end{equation*}
where:
\begin{itemize}
\item $\varepsilon>0$ is any small enough real number so that $N^\gamma_\varepsilon\subseteq \S$, the $\varepsilon$-neighbourhood of $c^h_{\gamma}$ in $(\S, h )$, is topologically an annulus,
\item $\delta: N^\gamma_\varepsilon \to (-\varepsilon,\varepsilon)$ is the signed $ h $-distance to the oriented closed geodesic curve $c^h_{\gamma}$ function (negative on its left and positive on its right), lifting to $\tilde{\delta}: \pi^{-1}(N^\gamma_\varepsilon) \to (-\varepsilon,\varepsilon)$,
\item $f:(-\varepsilon,\varepsilon) \to \R$ is a smooth function such that $f(\pm\varepsilon) = \pm1/2$ and its derivative $f'$ is compactly supported in $(-\varepsilon,\varepsilon)$ (i.e. $f'$ is a bump function with $\int f'=1$),
\item for all $p\in \pi^{-1}(N^\gamma_\varepsilon)\subset \tilde{\S}$, $C_{\gamma,p}$ is the image by $\dev$ of the oriented lift of $c_\gamma$ at distance less than $\varepsilon$ to $p$ (it is unique by definition of $\varepsilon)$.
\end{itemize}

Let $\eta \in \pi_1\S$ and $C$ be an oriented geodesic line of $\H^2$. Using the identification $\tilde{S} \simeq \H^2$ (see Remark~\ref{remark identification uS and H2}) the action of $\dif \eta $ on $\T\tilde{\S}$ is given by the one of $\rho(\eta) \in \Isom^+(\H^2)$ on $\T\H^2$. Moreover, one has $\tH(\rho(\eta) \cdot C) = \Ad \rho(\eta)\cdot \tH(C)$. Hence, the form $\alphat_\gamma$ satisfies that for all $\eta \in \pi_1\S$,
\begin{equation*}
\alphat_\gamma \circ \dif \eta \coloneqq \Ad \rho(\eta)\circ\alphat_\gamma \, ,
\end{equation*}
and thus induces an element $\alpha_\gamma \in \ZdR$, defined by: for all $(p,v) \in \T\tilde{\S}$ 
\begin{equation}
\label{eq twist in dR}
\alpha_\gamma (\pi(p),v) =
\begin{cases}
[p, \tH (C_{\gamma,p})]\, \dif(f\circ {\tilde \delta})(v) & \text{if } p\in \pi^{-1}(N^\gamma_\varepsilon) \, ,  \\
0 & \text{otherwise.}
\end{cases}
\end{equation}
The form $\alpha_\gamma$ is thus supported on the hyperbolic annulus $N^\gamma_\varepsilon$ (see Figure~\ref{fig alphagamma}). On that annulus it satisfies 
\begin{equation}
\label{eq twist in dR ver2}
\alpha_\gamma = \nu_\gamma \, \dif (f \circ \delta)  = \nu_\gamma (f' \circ \dif \delta) \, , 
\end{equation}
where $\nu_{\gamma}$ is the constant $\bundle$-valued function on $N^\gamma_\varepsilon$ defined by: for all $p\in \pi^{-1}( N^\gamma_\varepsilon)$,
\begin{equation}
\label{eq def nu}
 \nu_\gamma\bp{\pi(p)} \coloneqq \big[p,\tH(C_{\gamma,p})\big] \in \bundle \, .
\end{equation}

\begin{figure}[htbp]
\centering
\includegraphics{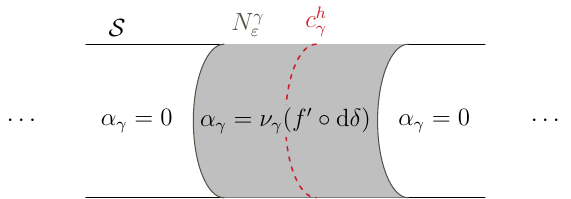} 
\caption{The form $\alpha_\gamma \in \ZdR$ representing the infinitesimal twist along $\gamma \in \pi_1\S$.}
\label{fig alphagamma}
\end{figure}

The following proposition states that the form $\alpha_\gamma$ defined in that way does indeed represent the infinitesimal twist $\twinf{\gamma}$ described in subsection~\ref{subsec inf translations and twists}.

\begin{proposition}
Let $\gamma \in \pi_1 \S$ be a simple element. Consider $\alpha_\gamma \in \ZdR$ an associated de Rham form  given by \eqref{eq twist in dR}, using a real number $\varepsilon>0$ and a smooth function $f:(-\varepsilon,\varepsilon) \to \R$. Then, 
\begin{equation*}
    \Psi_{\rho}(\dR{\alpha_\gamma}) = \twinf{\gamma} \, .
\end{equation*}
\end{proposition}

\begin{proof}
Let $p_0 \in \tilde{\S}$ be an arbitrary base point such that $\pi(p_0) \notin N_\varepsilon^\gamma$. Given $\eta \in \pi_1\S$, let $\sigma$ be any smooth path from $p_0$ to $\eta \cdot p_0$ transverse to every lift of $c_\gamma^h$ (for instance the geodesic arc between those two points). Then, using the expression of $\alphat_\gamma$, we have 
\begin{equation}
\label{eq inf twist in dR}
\int_{p_0}^{\eta \cdot p_0} \alphat_\gamma = \int_\sigma \alphat_\gamma =\sum_{p \in \Ima \sigma \cap \pi^{-1}(c_\gamma^h)}\epsilon_p(\sigma)\, \tH(C_{\gamma,p}) \, ,
\end{equation}
where the coefficients are given by 
\begin{equation*}
\epsilon_p(\sigma) = \begin{cases}
1 & \text{if $\sigma$ positively crosses } C_{\gamma,p} \text{ at } p \, , \\
-1 & \text{if $\sigma$ negatively crosses } C_{\gamma,p} \text{ at } p \, ,
\end{cases}
\end{equation*}
(positively crossing an oriented geodesic line being crossing it from left to right).

\vspace{\baselineskip}

\underline{If $c^h_{\gamma}$ separates $\S$ into two connected component $\S_0$ and $\S_1$}, then we can assume that $p_0 \in \pi^{-1}(\S_0)$ and set $p_1 \in \pi^{-1}(\S_1)$, such that the only lift of $c_\gamma^h$ separating $p_0$ and $p_1$ is $C_\gamma$ the oriented geodesic line along which $\rho(\gamma)$ is a hyperbolic translation. Then, using the expression of $\alphat_\gamma$, we have 
\begin{equation}
\label{eq p0 to p1}
\int_{p_0}^{p_1} \alphat_\gamma = \tH(C_\gamma) \, .
\end{equation}
If $\eta \in \pi_1\S_0$, then, by lifting a path in $\S_0$ we get a path from $p_0$ to $\eta \cdot p_0$ which does not intersect any lift of $c_\gamma^h$. Thus, using \eqref{eq inf twist in dR}, we have that for all $\eta \in \pi_1\S_0$
\begin{equation}
\label{eq inf twist dR0}
\int_{p_0}^{\eta \cdot p_0} \alphat_\gamma = 0 \, .
\end{equation}
Similarly, for all $\eta \in \pi_1\S_1$, we have
\begin{equation*}
\int_{p_1}^{\eta \cdot p_1} \alphat_\gamma = 0 \, ,
\end{equation*}
so that, using \eqref{eq p0 to p1} and remembering what was done in subsection~\ref{subsec de Rham map} concerning the change of base point for the de Rham map, 
\begin{equation}
\label{eq inf twist dR1}
\int_{p_0}^{\eta \cdot p_0} \alphat_\gamma = \tH(C_\gamma)- \Ad \rho(\eta) \cdot \tH(C_\gamma) \, ,
\end{equation}
Hence, it is clear from equations~\eqref{eq inf twist dR0} and \eqref{eq inf twist dR1} that $\Psi_{\rho}(\dR{\alpha_\gamma}) = \twinf{\gamma}$ where $\twinf{\gamma}$ is described in subsection~\ref{subsec inf translations and twists}.

\vspace{\baselineskip}

\underline{If $c_\gamma^h$ does not separate $\S$}, it is clear from expression~\eqref{eq inf twist in dR} that $\Psi_{\rho}(\dR{\alpha_\gamma}) = \twinf{\gamma}$ where $\twinf{\gamma}$ is described in subsection~\ref{subsec inf translations and twists}.
\end{proof}

\begin{remark}
\label{rem advantage dR}
Here, one clearly sees the advantage of the de Rham cohomology model for the tangent space of Teichm\"uller space evoked at the beginning of Section~\ref{Sec De Rham model}. The deformation of a hyperbolic surface by a twist along a simple closed geodesic curve $c$ is ‘‘localised'' around that curve. As a consequence a de Rham form associated with the infinitesimal twist $\twinf c$ along $c$ can be taken with support in any neighbourhood of the simple closed geodesic curve $c$.
\end{remark}

\section{A first formula of Wolpert}
\label{Sec Wolpert part 1}

In \cite{Wolpert:1981,Wolpert:1983}, Wolpert has found the explicit formula linking the Weil--Petersson product of infinitesimal twists along two closed geodesic curves and their intersection angles. Here is the equivalent statement for the Goldman form, which coincides with the Weil--Petersson form by Goldman's Theorem \cite{Goldman:1984}.

\begin{theorem}[Theorem~\ref{theo intro First Wolpert formula}]
\label{First Wolpert formula}
Let $\rho$ be a Fuchsian representation of $\pi_1\S$ and consider the associated hyperbolic surface $(\S,h)= \H^2/\rho(\pi_1\S)$. Let $\gamma,\eta \in \pi_1\S$ be two simple elements not conjugated with each other. Then, the infinitesimal twists $\twinf{\gamma},\twinf{\eta} \in \T_{[\rho]}\chifd(\pi_1\S)$ satisfy 
\begin{equation*}
\wg(\twinf{\gamma},\twinf{\eta})= 2 \sum_{q\in c^h_{\gamma} \cap c^h_{\eta}}\cos\theta_q \, ,
\end{equation*}
where $\theta_q \in \R/2\pi\Z$ is the oriented angle between the simple closed geodesic curves $c^h_{\gamma}$ and $c^h_{\eta}$ (given by Proposition~\ref{prop Unique geodesic loop}) at $q \in (\S, h )$.
\end{theorem}

\begin{corollary}
\label{w(dt,dt)=0}
If $\gamma,\eta \in \pi_1\S$ are two distinct elements of a pair of pants decomposition, in the associated Fenchel--Nielsen coordinates on $\chifd(\pi_1\S)$, one has
\begin{equation*}
\wg(\partial_{t_{\gamma}}, \partial_{t_{\eta}})= 0\, .
\end{equation*}
\end{corollary}

\begin{proof}
Apply Theorem~\ref{First Wolpert formula}, while noticing that the curves $c^h_{\gamma}$ and $c^h_{\eta}$ do not intersect. 
\end{proof}

In this section, we present a proof of Theorem~\ref{First Wolpert formula} due to Fillastre and Seppi \cite{Fillastre-Seppi:2020}. The idea is to prove a simple fact concerning the Killing product of infinitesimal hyperbolic translations, and then use the de Rham representatives of infinitesimal twists given in subsection~\ref{subsec Infinitesimal twists in de Rham cohomology} and Stokes's Theorem.

\subsection{Killing product of infinitesimal translations}

Here is the formula for the Killing product of infinitesimal hyperbolic translations (Definition~\ref{def infinitesimal translation}).

\begin{lemma}
\label{lem cos and killing}
Let $C$ and $C'$ be two oriented geodesic lines of $\H^2$ intersecting each other with oriented angle $\theta \in \R/2\pi\Z$. Then, their infinitesimal translations satisfy
\begin{equation*}
\kil\bp{\tH(C),\tH(C')}= 2\cos \theta\, .
\end{equation*}
\end{lemma}
\begin{proof}
Let us work in the hyperboloid model $\H^2 \subset \R^{2,1}$ described in subsection~\ref{subsec hyperboloid model}. By transitivity of the action of $\Isom^+(\H^2)$ on $2$-frames of $\H^2 \subset \R^{2,1}$ with oriented angle $\theta$, we can assume that $C = (\H^2 \cap \{ x_1=0\})$ and is oriented so that $x_2$ increases along $C$, and that $C' = R_\theta \cdot C$, where $R_\theta$ is the hyperbolic rotation
\begin{equation*}
R_\theta \coloneqq \left( \begin{matrix}
\cos \theta & \sin \theta & 0\\
-\sin \theta & \cos \theta & 0 \\
0 & 0 & 1
\end{matrix} \right) \in \SO_0(2,1)\, .
\end{equation*}
Then, we have 
\begin{equation*}
\tH(C) =\ddto T_t(C) = \ddto \left( \begin{matrix}
1 & 0 & 0\\
0 & \cosh t & \sinh t \\
0 & \sinh t & \cosh t
\end{matrix} \right)
= \left( \begin{matrix}
0 & 0 & 0\\
0 & 0 & 1 \\
0 & 1 & 0
\end{matrix} \right)  \in \so(2,1)\, ,
\end{equation*}
and 
\begin{equation*}
\tH(C') =\ddto T_t(C') = \ddto R_\theta \cdot T_t(C) \cdot R_{-\theta}= \Ad R_\theta  \cdot \tH(C) = \left( \begin{matrix}
0 & 0 & \sin \theta\\
0 & 0 & \cos \theta \\
\sin \theta & \cos \theta & 0
\end{matrix} \right)  \in \so(2,1) \, .
\end{equation*}
Hence, using the expression~\eqref{eq kil in SO} of the Killing form on $\so(2,1)$, we get
\begin{equation*}
\kil\bp{\tH(C),\tH(C')}= \tr\bp{\tH(C)\,\tH(C')} = 2\cos \theta\, . \qedhere
\end{equation*}
\end{proof}

\subsection{Proof of the formula} We can now prove Theorem~\ref{First Wolpert formula}.

\begin{proof}[Proof of Theorem~\ref{First Wolpert formula}]
Let $\rho$ be Fuchsian representation of $\pi_1\S$ and consider the hyperbolic surface $(\S,h) = \H^2/\rho(\pi_1\S)$ with universal covering $\pi:\tilde{\S} = \H^2 \to \S$. Let $\gamma,\eta \in \pi_1\S$ be two simple elements not conjugated with each other (then the simple closed geodesic curves $c^h_\gamma$ and $c^h_\eta$ are distinct). Let $\alpha_\gamma,\alpha_\eta \in \ZdR$ be de Rham representatives of respectively $\Psi_\rho^{-1}(\twinf{\gamma})$ and $\Psi_\rho^{-1}(\twinf{\eta})$ given by \eqref{eq twist in dR}, using a real number $\varepsilon>0$ and a smooth function $f:(-\varepsilon,\varepsilon) \to \R$. We recall that: 
\begin{itemize}
 \item $\varepsilon$ is such that the respective $\varepsilon$-neighbourhoods $N^\gamma_\varepsilon$ and $N^\eta_\varepsilon$ of $c_\gamma^h$ and $c_\eta^h$ (in which $\alpha_\gamma$ and $\alpha_\eta$ are respectively supported) are topologically annuli,
 \item $\varepsilon$ can be taken as small as wanted, 
 \item $f(\pm\varepsilon) = \pm1/2$.
\end{itemize}

For all $q \in c^h_{\gamma} \cap c^h_{\eta}$, let $Q_q$ be the connected component of $N^\gamma_\varepsilon \cap N^\eta_\varepsilon$ containing $q$. Let us assume that we have chosen $\varepsilon$ in the expression of $\alpha_\gamma$ and $\alpha_\eta$ small enough so that points in the finite set $c^h_{\gamma} \cap c^h_{\eta}$ are all separated by a distance at least $4\varepsilon$. Then all the $Q_q$'s are disjoint and, by Proposition~\ref{prop wg=wdR}, we have
\begin{equation*}
\wg(\twinf{\gamma},\twinf{\eta}) = \wdR(\dR{\alpha_\gamma},\dR{\alpha_\eta}) = \int_\S \kil(\alpha_\gamma \wedge \alpha_\eta) = \int_{N^\gamma_\varepsilon \cap N^\eta_\varepsilon} \kil(\alpha_\gamma \wedge \alpha_\eta)= \sum_{q \in c^h_{\gamma} \cap c^h_{\eta}}\int_{Q_q} \kil(\alpha_\gamma \wedge \alpha_\eta)\, .
\end{equation*}

Let $q \in c^h_{\gamma} \cap c^h_{\eta}$ and $p$ be any point in $\pi^{-1}(q) \subset \tilde{\S}=\H^2$. Consider $C_{\gamma,p}$ and $C_{\eta,p}$ the lifts in $\H^2$ of respectively $c^h_{\gamma}$ and $c^h_{\eta}$ going through $p$. Using the description of $\alpha_\gamma$ and $\alpha_\eta$ given by expression \eqref{eq twist in dR ver2}, and Lemma~\ref{lem cos and killing}, we have 
\begin{equation*}
\int_{Q_q} \kil(\alpha_\gamma \wedge \alpha_\eta) = \int_{Q_q} \kil \bp{\tH(C_{\gamma,p}),\tH(C_{\eta,p})} \, \dif(f\circ \delta) \wedge \dif(f\circ \delta') = 2 \cos (\theta_{p}) \int_{Q_q} \dif(f\circ \delta) \wedge \dif(f\circ \delta') \, ,
\end{equation*}
where $\theta_{p} \in \R/2\pi\Z$ is the oriented angle between $C_{\gamma,p}$ and $C_{\eta,p}$ at $q \in \H^2$, and $\delta,\delta': N^\gamma_\varepsilon \cap N^\eta_\varepsilon \to \R$ are the respective signed distance function to $c_\gamma^h$ and $c_\eta^h$. 

Note that $\theta_p$ and $\theta_q$, the oriented angle between $c^h_{\gamma}$ and $c^h_{\eta}$ at $q$ in $(\S,h)= \H^2/\rho(\pi_1\S)$, are equal. Hence, we have
\begin{equation*}
\wg(\twinf{\gamma},\twinf{\eta}) = \sum_{q \in c^h_{\gamma} \cap c^h_{\eta} }\int_{Q_q} \kil(\alpha_\gamma \wedge \alpha_\eta) = 2 \sum_{q \in c^h_{\gamma} \cap c^h_{\eta} } \cos \theta_q \int_{Q_{q}} \dif(f\circ \delta) \wedge \dif(f\circ \delta') \, .
\end{equation*}
We claim that for all $q \in c^h_{\gamma} \cap c^h_{\eta}$, 
\begin{equation*}
    \int_{Q_{q}} \dif(f\circ \delta) \wedge \dif(f\circ \delta') =1 \, , 
\end{equation*}
concluding the proof.

Let us now prove our claim. Let $q \in c^h_{\gamma} \cap c^h_{\eta}$. First, notice that Stokes's Theorem implies that
\begin{equation*}
 \int_{Q_q} \dif(f\circ \delta) \wedge \dif(f\circ \delta') = \int_{Q_q} \dif\bp{(f\circ \delta) \, \dif(f\circ \delta')} = \int_{\partial Q_q}(f\circ \delta) \, \dif(f\circ \delta') \, .
\end{equation*}
The region $Q_q$ has four smooth edges as boundary components (see Figure~\ref{fig Qq}):
\begin{equation*}
\mathcal{E}_\pm \coloneqq \overline{Q_q}\cap \left\lbrace f \circ \delta = \pm \frac{1}{2} \right\rbrace
\quad \text{and} \quad
\mathcal{E}'_\pm \coloneqq \overline{Q_q}\cap \left\lbrace f \circ \delta' = \pm \frac{1}{2} \right\rbrace \, \, \, .
\end{equation*}

\begin{figure}[htbp]
\centering
\includegraphics{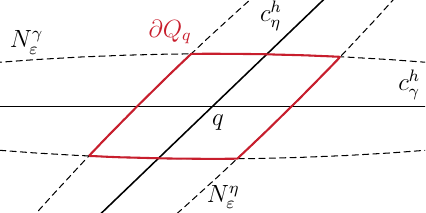} 
\caption{The region $Q_q$ on which we apply Stokes's Theorem.}
\label{fig Qq}
\end{figure}

Since $f\circ \delta'$ is constant on both $\mathcal{E}'_\pm$, we have
\begin{equation*}
\int_{\mathcal{E}'_\pm}(f\circ \delta) \, \dif(f\circ \delta')=0 \, \, .
\end{equation*}
Setting $q_1$ and $q_2$ the two endpoints of $\mathcal{E}_+$, and taking orientation into account, we get
\begin{equation*}
\int_{\mathcal{E}_+}(f\circ \delta) \, \dif(f\circ \delta') =\int_{\mathcal{E}_+} \frac{1}{2} \, \dif(f\circ \delta') = \frac{1}{2}\bp{(f\circ \delta')(q_2)-(f\circ \delta')(q_1)} = \frac{1}{2} \, .
\end{equation*}
As $f\circ\delta= -1/2$ on $\mathcal{E}_-$ which has orientation the opposite of the one of $\mathcal{E}_+$, we similarly get
\begin{equation*}
\int_{\mathcal{E}_-}(f\circ \delta) \, \dif(f\circ \delta') = \frac{1}{2} \, ,
\end{equation*}
and finally 
\begin{equation*}
 \int_{Q_q} \dif(f\circ \delta) \wedge \dif(f\circ \delta') = \int_{\partial Q_q}(f\circ \delta) \, \dif(f\circ \delta')= \int_{\mathcal{E}_+}(f\circ \delta) \, \dif(f\circ \delta')+\int_{\mathcal{E}_-}(f\circ \delta) \, \dif(f\circ \delta') = 1 \, . \qedhere
\end{equation*}
\end{proof}

\section{A second formula of Wolpert}
\label{Sec Wolpert part 2}

Let us begin this section by recalling the definition of the \emph{geodesic length function} on Teichm\"uller space.

\begin{definition}[Geodesic length function]
For all $\gamma \in \pi_1\S$, the \emph{geodesic length function} \cite{Fricke-Klein} 
\begin{equation*}
    l_\gamma: \chifd(\pi_1\S) \longrightarrow \R
\end{equation*}
is defined in the following way. For all $[\rho] \in \chifd(\pi_1\S)$, $l_{\gamma}[\rho]$ is the length of the unique closed geodesic on the hyperbolic surface $(\S, h ) = \H^2/\rho(\pi_1\S)$ with free homotopy class $[\gamma]$ (see Proposition~\ref{prop Unique geodesic loop}). 
\end{definition}

\begin{remark}
When $\gamma \in \pi_1\S$ is part of a pair of pants decomposition, the geodesic length function $l_\gamma$ coincides with the length parameter having same notation, introduced in subsection~\ref{subsec Fenchel_Nielsen coordinates}.
\end{remark}

Note that $l_{\gamma}[\rho]$ is the translation length of $\rho(\gamma)$. Thus, using the $\SO_0(2,1)$ description of $\Isom^+(\H^2)$, the geodesic length function can be expressed as the smooth map
\begin{equation}
\label{eq length in SO(2,1)}
l_{\gamma}[\rho] = \cosh^{-1}\vp{\frac{\tr \rho(\gamma)-1}{2}} \in \R^*_+ \, ,
\end{equation}
which is well-defined, as two representatives of a same $[\rho] \in \chifd(\pi_1\S)$ are conjugated and thus have the same trace.

In \cite{Wolpert:1983}, Wolpert has explicitly described the relation between the infinitesimal twist along a closed geodesic curve $c^h_\gamma$, the geodesic length function $l_{\gamma}$, and the Weil--Petersson form. Here is the equivalent statement for the Goldman form, which coincides with the Weil--Petersson form by Goldman's Theorem \cite{Goldman:1984}.

\begin{theorem}[Theorem~\ref{theo intro Second Wolpert formula}]
\label{Second Wolpert formula}
Let $[\rho] \in \chifd(\pi_1\S)$ and consider a simple element $\gamma \in \pi_1\S$. Then, the infinitesimal twist $\twinf{\gamma} \in \T_{[\rho]}\chifd(\pi_1\S)$ and the geodesic length function $l_{\gamma}:\chifd(\pi_1\S)\to\R^*_+$ satisfy
\begin{equation*}
\wg(\twinf{\gamma}, \cdot)= 2 \, \dif_{[\rho]}l_{\gamma} \, .
\end{equation*}
In particular, if $\gamma \in \pi_1\S$ is part of a pair of pants decomposition of $\S$, in the associated Fenchel--Nielsen coordinates on $\chifd(\pi_1\S)$,
\begin{equation*}
\wg(\partial_{t_{\gamma}}, \cdot)= 2 \, \dif l_{\gamma} \, .
\end{equation*}
\end{theorem}

In this section we provide a new proof of Theorem~\ref{Second Wolpert formula}. We shall first compute the differential of the geodesic length function, and then use de Rham representatives of infinitesimal twists given in subsection~\ref{subsec Infinitesimal twists in de Rham cohomology} and Stokes's Theorem to conclude.

\subsection{Differentiation of the length function}

Here is the formula for the differential map 
\begin{equation*}
\dif_{[\rho]}l_\gamma: \T_{[\rho]}\chifd(\pi_1\S) \to \R
\end{equation*}
of the geodesic length function, using the $\SO_0(2,1)$ model of $\Isom^+(\H^2)$.

\begin{proposition}
\label{dl}
Let $\rho$ be a Fuchsian representation of $\pi_1\S$, and $\gamma \in \pi_1\S$. Then, for all $[\tau] \in \T_{[\rho]}\chifd(\pi_1\S) = \Hisom$, in the $\SO_0(2,1)$ model of $\Isom^+(\H^2)$, one has
\begin{equation*}
 \dif_{[\rho]}l_{\gamma}[\tau] = \frac{\tr\bp{\tau(\gamma)\rho(\gamma)}}{ 2\sinh l_{\gamma}[\rho] } \, .
\end{equation*}
\end{proposition}

\begin{proof}
Let us first notice that the wanted expression is well-defined: one easily checks that adding a coboundary to a cocycle $\tau \in \Zisom$ does not change the quantity $\tr(\tau(\gamma)\rho(\gamma))$.

Let $\tau \in \Zisom$ and set $(\rho_s)_{s \in \R}$to  be a smooth path of Fuchsian representations of $\pi_1\S$ such that $\rho_0 = \rho$ and 
\begin{equation*}
 \tau= \ddso \rho_s \rho^{-1} \, . 
\end{equation*} 
Using the $\SO_0(2,1)$ model of $\Isom^+(\H^2)$ and differentiating expression~\eqref{eq length in SO(2,1)}, we get 
\begin{equation*}
 \ddso l_{\gamma}[\rho_s] = \frac{1}{2}\tr\vp{\ddso \rho_s(\gamma)}(\cosh^{-1})'\vp{\frac{\tr \rho(\gamma)-1}{2}} = \frac{\tr\bp{\tau(\gamma)\rho(\gamma)}}{2 \sinh l_{\gamma}[\rho]} \, . \qedhere
\end{equation*}
\end{proof}

\subsection{Proof of the formula}

Before proving Theorem~\ref{Second Wolpert formula}, let us notice the following fact concerning infinitesimal hyperbolic translations (Definition~\ref{def infinitesimal translation}).

\begin{lemma}
\label{lem killing trace and length}
Let $C$ be an oriented geodesic line of $\H^2$. Then, using the $\SO_0(2,1)$ description of $\Isom^+(\H^2)$, we have that for all $M \in \so(2,1)$ and $t \in \R^*$,
\begin{equation*}
 \kil \bp{\tH(C), M} = \frac{\tr\bp{T_t(C)M}}{\sinh t } \, .
\end{equation*}
\end{lemma}
\begin{proof}
Up to conjugation by an element of $\SO_0(2,1)$, we can assume that
\begin{equation*}
T_t(C) = \left( \begin{matrix}
1 & 0 & 0\\
0 & \cosh t & \sinh t \\
0 & \sinh t & \cosh t
\end{matrix} \right)
\quad \text{and} \quad
\tH(C) = \left( \begin{matrix}
0 & 0 & 0\\
0 & 0 & 1 \\
0 & 1 & 0
\end{matrix} \right) \, .
\end{equation*}
Let $a,b,c\in \R$ such that $M \in \so(2,1)$ is expressed by
\begin{equation*}
M= \left( \begin{matrix}
0 & a & b\\
-a & 0 & c \\
b & c & 0
\end{matrix} \right) \, .
\end{equation*}
Using the expression~\eqref{eq kil in SO} of the Killing form on $\so(2,1)$, we have 
\begin{equation*}
 \tr\bp{T_t(C)M}= 2 c \sinh t = \tr\bp{\tH(C) M} \sinh t = \kil\bp{\tH(C) ,M} \sinh t \, . \qedhere
\end{equation*}
\end{proof}

\begin{proof}[Proof of Theorem~\ref{Second Wolpert formula}]
Let $\rho$ be a Fuchsian representation of $\pi_1\S$ and consider the associated hyperbolic surface $(\S,h) = \H^2/\rho(\pi_1\S)$ with universal covering $\pi:\tilde{\S} = \H^2 \to \S$. Let $\gamma$ be a simple element of $\pi_1\S$ and let $\alpha_\gamma \in \ZdR$ be a de Rham representative of $\Psi_\rho^{-1}(\twinf{\gamma})$ given by \eqref{eq twist in dR}, using a real number $\varepsilon>0$ and a smooth function $f:(-\varepsilon,\varepsilon) \to \R$. We recall that: 
\begin{itemize}
 \item $\varepsilon$ is such that the $\varepsilon$-neighbourhood $N^\gamma_\varepsilon$ of $c_\gamma^h$ (in which $\alpha_\gamma$ is supported) is topologically an annulus,
 \item $\varepsilon$ can be taken as small as wanted, 
 \item $f(\pm\varepsilon) = \pm1/2$.
\end{itemize}

Let $[\tau] \in \Hisom$ and let $\beta \in \ZdR$ be a de Rham representative of $\Psi_\rho^{-1}[\tau]$. By Proposition~\ref{prop wg=wdR}, and because $\alpha_\gamma$ is supported on $N_\varepsilon^\gamma$, we have 
\begin{equation*}
\wg(\twinf{\gamma}, [\tau])= \wdR(\dR{\alpha_\gamma},\dR{\beta}) = \int_{N^\gamma_\varepsilon} \kil(\alpha_\gamma\wedge \beta) \, .
\end{equation*}
Using the expression \eqref{eq twist in dR ver2}, on $N^\gamma_\varepsilon$ we have 
\begin{equation*}
\kil(\alpha_\gamma \wedge \beta) = \dif(f\circ \delta) \wedge \kil(\nu_{\gamma}, \beta) \, ,
\end{equation*}
where $\nu_\gamma$ is defined by~\eqref{eq def nu}. Set $g$ to be the $1$-form on $N^\gamma_\varepsilon$ defined by $g\coloneqq\kil(\nu_{\gamma}, \beta)$.
As $\beta$ is closed and $\nu_{\gamma}$ is constant on $N^\gamma_\varepsilon$, we have that $\dif g = \kil(\nu_\gamma,\dif\beta) =0$ and
\begin{align*}
\dif\bp{(f\circ \delta) g} & = \dif (f\circ \delta) \wedge g + (f\circ \delta) \, \dif g = \dif (f\circ \delta) \wedge g = \dif (f\circ \delta) \wedge  \kil(\nu_{\gamma}, \beta) = \kil(\alpha_\gamma \wedge \beta) \, .
\end{align*}
Then, by Stokes's Theorem,
\begin{equation*}
\int_{ N^\gamma_\varepsilon} \kil(\alpha_\gamma\wedge \beta) = \int_{ N^\gamma_\varepsilon}\dif\bp{(f\circ \delta)g} = \int_{\partial N^\gamma_\varepsilon}(f\circ \delta)g \, .
\end{equation*}

\begin{figure}[htbp]
\centering
\includegraphics{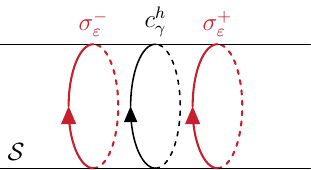} 
\caption{The annulus region $N^\gamma_\varepsilon$ on which we apply Stokes's Theorem.}
\label{Annulus}
\end{figure}

Let $\sigma^{\pm}_{\varepsilon}$ be the two loop boundaries of $N^\gamma_\varepsilon$ such that $f\circ\delta = \pm 1/2$ on $\sigma^{\pm}_{\varepsilon}$ (see Figure~\ref{Annulus}). Taking orientation into account, we get 
\begin{equation*}
\wg(\twinf{\gamma}, [\tau]) = \int_{\partial N^\gamma_\varepsilon}(f\circ \delta)g = \int_{\sigma^{+}_{\varepsilon}}(f\circ \delta)g - \int_{\sigma^{-}_{\varepsilon}}(f\circ \delta)g = \frac{1}{2} \vp{\int_{\sigma^{+}_{\varepsilon}}g + \int_{\sigma^{-}_{\varepsilon}}g} \, .
\end{equation*}
Now, let $\varepsilon$ tend to $0$. Note that, as $\varepsilon$ tends to 0, both $\sigma^{-}_{\varepsilon}$ and $\sigma^{+}_{\varepsilon}$ tend to $c^h_{\gamma}$. Let $p_0$ be a point on $C_\gamma$, the oriented geodesic line of $\H^2$ along which $\rho(\gamma)$ is a positive translation. We thus have 
\begin{equation*}
\wg(\twinf{\gamma}, [\tau]) = \int_{c^h_{\gamma}}g = \int_{p_0}^{\gamma \cdot p_0}\tilde{g} = \int_{p_0}^{\gamma \cdot p_0}\kil\bp{\tH(C_{\gamma}),\tilde{\beta}} \, ,
\end{equation*}
where we integrate along the geodesic path from $p_0$ to $\gamma \cdot p_0$ . Using the bilinearity of $\kil$, that becomes
\begin{equation}
\label{eq proof dl}
\wg(\twinf{\gamma}, [\tau]) =\kil\vp{\tH(C_{\gamma}), \int_{p_0}^{\gamma \cdot p_0} \tilde{\beta} } \, .
\end{equation}
By subsection~\ref{subsec wdR}, we can set $\tau_\beta \in \Hisom$ to be the cocycle defined by: for all $\eta \in \pi_1\S$,
\begin{equation*}
 \tau_\beta(\eta)= \int_{p_0}^{\eta \cdot p_0} \tilde{\beta},
\end{equation*} and we have $[\tau_\beta]=\Psi_\rho\dR{\beta} = [\tau]$. Then equation~\eqref{eq proof dl} becomes
\begin{equation*}
\wg(\twinf{\gamma}, [\tau]) = \kil\vp{\tH(C_{\gamma}), \tau_\beta(\gamma) } \, .
\end{equation*}
As $\rho(\gamma)$ is the hyperbolic translation $T_{l_\gamma[\rho]}(C_\gamma)$, using Proposition~\ref{dl} and Lemma~\ref{lem killing trace and length}, that becomes
\begin{equation*}
\wg(\twinf{\gamma}, [\tau]) = 2 \, \dif_{[\rho]}l_{\gamma}[\tau_\beta] = 2 \, \dif_{[\rho]}l_{\gamma}[\tau] \, . \qedhere
\end{equation*}
\end{proof}

\section{Induced de Rham forms and symmetries of hyperbolic pairs of pants}

\label{sec Induced de Rham forms and symmetries of pairs of pants}

In order to construct the involutions from Theorem~\ref{theo intro existence involution on tangent space}, we need to consider restrictions of $\bundle$-valued de Rham forms to hyperbolic pairs of pants and use the orientation-reversing isometric symmetry on those hyperbolic pairs of pants. In this section, let us properly define and study those.

\subsection{Induced de Rham forms}

\label{subsec Induced de Rham forms}

Let $\rho$ be a Fuchsian representation of $\pi_1\S$ and $U$ be a connected open subset of $\S$ such that the subgroup inclusion $\pi_1U \to \pi_1\S$ is injective. Consider a universal covering ${\pi}: \tilde{\S} \to \S$ and let $\tilde{U}$ be a connected component of $\pi^{-1}(U)\subseteq\tilde{\S}$. Then, ${\pi}\vert_{\tilde{U}}: \tilde{U} \to U$ is a universal covering of $U$. Hence, for all $\alpha \in \ZdR$, the restriction ${\alpha}_{\vert U}$ can be seen as an element, that we shall denote by $\alpha_{U}$, of $\ZdRof{U}$, where
\begin{equation*}
\bundlerest{U} \coloneqq \bp{\tilde{U} \times \isom(\H^2)}/\pi_1U \, .
\end{equation*}
Then, $\alpha_U \in \ZdRof{U}$ is induced by $\alphat_U \coloneqq \alphat_{\vert \tilde{U}}$ which is a $\rho(\pi_1U)$-equivariant closed $\isom(\H^2)$-valued 1-form on $\tilde{U}$.

\subsection{Symmetries on pairs of pants and involutions on induced de Rham forms}
\label{subsec symmetry on pants}

Let $\rho$ be a Fuchsian representation of $\pi_1\S$ and consider the hyperbolic surface $(\S,h)= \H^2/\rho(\pi_1\S)$. Fix a pair of pants decomposition $([\gamma_i])_{1\leq i\leq 3g-3}$ of $\S$ and a universal covering $\pi:\tilde{\S} \to \S$. 

Let $\pant \subset (\S,h)$ be a hyperbolic pair of pants given by that decomposition (seen as an open subset). Note that $\pi_1\pant = \langle \gamma,\eta\rangle$ is the free group generated by two elements $\gamma,\eta \in \pi_1\S$,  such that $[\gamma]$ and $[\eta]$ are free homotopy classes of boundary loops of $(\pant,h) \subset \S$ (see Figure~\ref{fig pi1 pant}). Then, $c^h_\gamma$ and $c^h_\eta$ are two distinct oriented geodesic boundary components of the hyperbolic pair of pants $(\pant,h)$. 

\begin{figure}[htbp]
\centering
\includegraphics{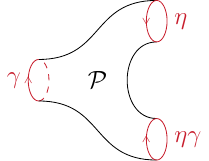} 
\caption{Homotopy classes of boundary loops of a pair of pants generating its fundamental group.}
\label{fig pi1 pant}
\end{figure}

Let $C_\gamma$ and $C_\eta$ be the two respective oriented geodesic lines of $\H^2$ along which $\rho(\gamma)$ and $\rho(\eta)$ are positive hyperbolic translations. Their pull-backs by the developing map $\dev$ associated with $\rho$ are geodesic lines of $(\tilde{\S},\tilde{h)}$. Let $\upant$ be the connected component of $\pi^{-1}(\pant) \subset \tilde{\S}$ having the pull-backs by the developing map of $C_\gamma$ and $C_\eta$ as boundaries. Then, $\pi : \upant \to \pant$ is a universal covering (see Figure~\ref{fig universal pants}).

The hyperbolic pair of pants $\pant$ is naturally endowed with an orientation-reversing isometric symmetry $\sigma$ exchanging the two right-angled hyperbolic hexagons composing it. Let $c_0^h \subset \pant \subset \H^2/\rho(\pi_1\S)$ be the $ h $-geodesic arc orthogonal to $c^h_\gamma$ and $c^h_\eta$. It is one of the three geodesic arcs cutting $\pant \subset (\S,h)$ into two isometric right-angled hyperbolic hexagons.

Let $C$ be the geodesic line of $\H^2$ orthogonal to $C_\gamma$ and $C_\eta$ endowed with an arbitrary orientation (the orientation of $C$ will not matter). The geodesic line $C$ contains a segment which is a lift of $c^h_0$ in $\upant$, the connected component of $\pi^{-1}(\pant)$ having $C_\gamma$ and $C_\eta$ as boundaries. A lift $\tilde{\sigma}$ of the symmetry $\sigma$ to $\upant\subset \H^2$, is the restriction of $S \in \Isom^-(\H^2)$, the hyperbolic reflection along $C$ on $\H^2$ (see Figure~\ref{fig universal pants}).

\begin{remark}
\label{remark identification uS and H2 2}
As explained in Remark~\ref{remark identification uS and H2}, using the universal covering $\pi : \tilde{S} = \H^2 \to \H^2/\rho(\pi_1\S)$ (or equivalently the identification $\tilde{S}\simeq\H^2$ given by a developing map), the action of any element $\xi \in \pi_1\S$ on $\tilde{\S}$ is given by the action of $\rho(\xi)$ on $\H^2$. Similarly, the action of $\tilde{\sigma}$ is given by the restriction of the hyperbolic reflection $S$ to $\upant$, seen as a subset of $\H^2$. Still, for the same reasons as mentioned in Remark~\ref{remark identification uS and H2}, we shall denote distinctively the actions on universal covers and on subsets of $\H^2$.
\end{remark}

\begin{figure}[htbp]
\centering
\includegraphics{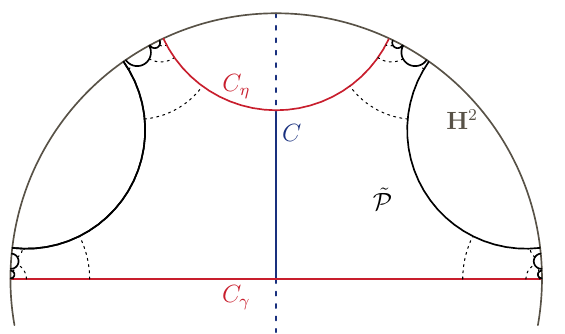}
\caption{The universal cover of the hyperbolic pair of pants $\pant$ in the Poincar\'e disk model of $\H^2$; the right-angled hyperbolic hexagons are indicated by the dashed geodesic segments.}
\label{fig universal pants}
\end{figure}

\begin{proposition}
\label{prop hol(l)S}
Let the setting be as above, let $S \in \Isom^-(\H^2)$ be the hyperbolic reflection along $C$, and let $\xi \in \{\gamma^{\pm1}, \eta^{\pm1}\}$. Then, we have
\begin{equation*}
\rho(\xi)\circ S =S \circ \rho(\xi)^{-1} \, .
\end{equation*}
Using the identification described in Remark~\ref{remark identification uS and H2 2}, that means that on $\tilde \pant$
\begin{equation*}
\xi \circ \tilde{\sigma} =\tilde{\sigma} \circ \xi^{-1} \, .
\end{equation*}
Consequently, we also have 
\begin{equation*}
\Ad\rho(\xi)\circ \Ad S =\Ad S \circ \Ad\rho(\xi)^{-1} \, .
\end{equation*}
\end{proposition}
\begin{proof}
As $S \in \Isom^-(\H^2)$ is the hyperbolic reflection along $C$, that is a direct consequence of the fact that $\rho(\xi)$ is a hyperbolic translation along a geodesic line of $\H^2$ orthogonal to $C$.
\end{proof}

\begin{proposition}
\label{prop alpha* is equiv}
Let $\alpha_{\pant} \in \ZdRof{\pant}$. Then, the $\isom(\H^2)$-valued $1$-form on $\upant$ defined by
\begin{equation*}
\alphat^*_{\pant} \coloneqq \Ad S \cdot (\tilde{\sigma}^*\alphat_\pant) = \Ad S \cdot (\alphat_\pant\circ\dif\tilde{\sigma}) \, .
\end{equation*}
is $\Ad \rho(\pi_1\pant)$-equivariant.
\end{proposition}
\begin{proof}
By Proposition~\ref{prop hol(l)S}, for all $\xi \in \{\gamma^{\pm1}, \eta^{\pm1}\}$, 
\begin{align*}
\alphat^*_{\pant}\circ \dif \xi &= \Ad S \cdot (\alphat_{\pant}\circ\dif\tilde{\sigma} \circ \dif \xi ) =\Ad S \cdot (\alphat_{\pant}\circ\dif \xi^{-1}\circ \dif\tilde{\sigma} ) \\
& = \Ad S \Ad\rho(\xi)^{-1}\cdot (\alphat_{\pant}\circ \dif\tilde{\sigma}) = \Ad\rho(\xi)\Ad S\cdot (\alphat_{\pant}\circ \dif\tilde{\sigma}) \\
& = \Ad\rho(\xi)\cdot (\alphat_{\pant}\circ \dif\tilde{\sigma})= \Ad\rho(\xi)\cdot \alphat^*_{\pant} \, .
\end{align*}
Hence, as $\pi_1\pant=\left\langle \gamma,\eta\right\rangle$, the identity holds for all $\xi \in \pi_1\pant$.
\end{proof}

The operator ${}^*$ on $\ZdRof{\pant}$ defined by Proposition~\ref{prop alpha* is equiv} is clearly linear. It seems to depend on the choice of the geodesic arc $c^h_0$ among the three $h$-geodesic arcs cutting $\pant \subset (\S,h)$ into two isometric right-angled hyperbolic hexagons, as well as on the choice of its lift $C$ in $\tilde{\pant}$. It actually does not.

\begin{proposition}
The operator ${}^*$ on $\ZdRof{\pant}$ does not depend on the choice of $c^h_0 \subset \pant \subset (\S,h)$ among the three geodesic arcs cutting $\pant \subset (\S,h)$ into two isometric right-angled hyperbolic hexagons, nor of its lift in $\tilde{\pant}$.
\end{proposition}
\begin{proof}
\textbullet\ First, let us prove that if $c_0^h$ is chosen and fixed, the operator ${}^*$ does not depend on the choice of its lift in $\tilde{\pant}$. Let $C$ and $C'$ be two lifts of $c_0^h$ in $\upant$, and let $S,S' \in \Isom^-(\H^2)$ be the hyperbolic reflections respectively along $C$ and $C'$, and $\tilde{\sigma},\tilde{\sigma}'$ their restriction to $\upant$. Let us define 
\begin{equation*}
\alphat^*_{\pant} \coloneqq \Ad S \cdot (\alphat_\pant\circ\dif\tilde{\sigma}) \quad \text{and} \quad\alphat^\star_{\pant} \coloneqq \Ad S'  \cdot (\alphat_\pant\circ\dif\tilde{\sigma}') \, .
\end{equation*}
We have to prove that $\alphat^*_{\pant} = \alphat^\star_{\pant}$.

As $C$ and $C'$ are two lifts of $c^h_0$ in $\upant$, there exists an element $\xi \in \pi_1\pant$ such that $C' = \rho(\xi) C$. We then have that $S' = \rho(\xi) S \rho(\xi)^{-1} $ and $\tilde{\sigma}' = \xi \circ \tilde{\sigma} \circ \xi^{-1}$. Using, the $\Ad \rho $-equivariance~\eqref{eq ad equiv 1-form} of $\alphat_\pant$, we have
\begin{align*}
\alphat^\star_{\pant} & = \Ad S' \cdot (\alphat_\pant\circ \dif\tilde{\sigma}') = \Ad\rho(\xi)\Ad S \Ad\rho(\xi)^{-1} \cdot (\alphat_\pant\circ \dif \xi \circ \dif\tilde{\sigma} \circ \dif \xi ^{-1}) \\
& = \Ad\rho(\xi)\Ad S \cdot (\alphat_\pant\circ \dif\tilde{\sigma} \circ \dif \xi ^{-1}) = \Ad \rho(\xi) \cdot (\alphat^*_\pant \circ \dif \xi^{-1}) \, .
\end{align*}
Using, the $\Ad \rho $-equivariance of $\alphat^*_\pant$ (Proposition~\ref{prop alpha* is equiv}), that becomes 
\begin{equation*}
 \alphat^\star_{\pant} = \alphat^*_{\pant}\, .
\end{equation*}

\textbullet\ Now, let us prove that the choice among the three geodesic arcs cutting $\pant \subset (\S,h)$ into two isometric right-angled hyperbolic hexagons does not change the operator ${}^*$. If $C$ is a lift of $c^h_0$, one of those three minmal geodesic arcs, then the geodesics $C_1\coloneqq \rho(\gamma)^{1/2} C$ and $C_2\coloneqq \rho(\eta)^{1/2} C$ are lifts of the two other arcs (see Figure~\ref{fig universal pants s}). 

Let us treat the case where the chosen lift is $C'=C_1$, the other case can be treated identically. In that case, the hyperbolic reflection $S' \in \Isom^-(\H^2)$ along $C'$ is $S' = \rho(\gamma)^{1/2} S \rho(\gamma)^{-1/2} $. Let us define
\begin{equation*}
\alphat^*_{\pant} \coloneqq  \Ad S \cdot (\alphat_\pant\circ\dif\tilde{\sigma}) \quad \text{and} \quad\alphat^\star_{\pant} \coloneqq \Ad S'  \cdot (\alphat_\pant\circ\dif\tilde{\sigma}') \, .
\end{equation*}
Using that $\rho(\gamma)$ is a hyperbolic translation along a geodesic orthogonal to $C$, we have 
\begin{equation*}
\Ad S'  = \Ad\rho(\gamma)^{1/2}\Ad S\Ad\rho(\gamma)^{-1/2} = \Ad S\Ad\rho(\gamma)^{-1} = \Ad\rho(\gamma)\Ad S \, ,
\end{equation*}
Using the identification described in Remark~\ref{remark identification uS and H2 2}, that last equality translates to 
\begin{equation*}
 \tilde{\sigma}' = \tilde{\sigma} \circ \gamma^{-1} = \gamma \circ \tilde{\sigma} \, .
\end{equation*}
Thus, by the $\Ad \rho $-equivariance~\eqref{eq ad equiv 1-form} of $\alphat$, we have 
\begin{equation*}
\alphat^\star_{\pant} =\Ad S'  \cdot (\alphat_\pant\circ\dif\tilde{\sigma}') = \Ad S \Ad \rho(\gamma)^{-1} \cdot (\alphat_\pant\circ \dif \gamma \circ \dif\tilde{\sigma}) = \Ad S \cdot (\alphat_\pant\circ\dif\tilde{\sigma}) =\alphat^*_{\pant} \, ,
\end{equation*}
concluding the proof.
\end{proof}

Note that the operator ${}^*$ can also be defined on the vector space $Z^0_{\mbox{\tiny \textup{dR}}}(\pant,\bundlerest{\pant})$ of smooth sections of $\bundlerest{\pant}$ by: for all $f_\pant \in Z^0_{\mbox{\tiny \textup{dR}}}(\pant,\bundlerest{\pant})$,
\begin{equation*}
\tilde{f}_\pant^* \coloneqq \Ad S \cdot (\tilde \sigma^* \tilde f_\pant)= \Ad S \cdot(\tilde{f}_\pant \circ \tilde{\sigma}).
\end{equation*}

\begin{lemma}
\label{lem * and d commute}
The exterior derivative operator $\dif$ and the operator ${}^*$ commute: for all $f_\pant \in Z^0_{\mbox{\tiny \textup{dR}}}(\pant,\bundlerest{\pant})$,
\begin{equation*}
 \dif (f_\pant^*)= (\dif f_\pant)^* \, .
\end{equation*}
\end{lemma}
\begin{proof}
Let $f_\pant \in Z^0_{\mbox{\tiny \textup{dR}}}(\pant,\bundlerest{\pant})$. The action of $\Ad S$ on $\isom(\H^2)$ is linear. Hence, 
\begin{equation*}
 \dif (\tilde{f}_\pant^*) = \dif\bp{\Ad S \cdot (\tilde{f}_\pant \circ\sigma)} = \Ad S \cdot (\dif \tilde{f}_\pant \circ \dif \tilde{\sigma}) = (\dif f_\pant)^* \, . \qedhere
\end{equation*}
\end{proof}

By Lemma~\ref{lem * and d commute}, the linear operator ${}^* : \ZdRof{\pant} \to \ZdRof{\pant}$ sends coboundaries onto coboundaries. Hence, we can define the linear endomorphism ${}^*$ on $\HdRof{\pant}$ by: for all $\dR{\alpha_\pant} \in \HdRof{\pant}$, 
\begin{equation*}
 \dR{\alpha_\pant}^* \coloneqq \dR{\alpha_\pant^*} \in \HdRof{\pant} \, .
\end{equation*}

As $\tilde{\sigma}^2=\Id$, for all $\dR{\alpha_\pant} \in \HdRof{\pant}$, we have $(\dR{\alpha_\pant}^*)^*=\dR{\alpha_\pant}$. Geometrically, the ${}^*$ transformation is the linear involution on $\HdRof{\pant}$ induced by the orientation-reversing isometric symmetry $\sigma$ on the hyperbolic pair of pants $\pant$ exchanging the two right-angled hyperbolic hexagons composing it.

\subsection{A technical lemma}
\label{subsec central lemma}

The goal of this subsection is to prove a technical result (Lemma~\ref{lem central result}) which we shall use in the proof of Theorem~\ref{theo intro w(dl,dl)=0}. We work in the same setting as in the previous subsection.

Let $\rho$ be a Fuchsian representation of $\pi_1\S$ and consider the hyperbolic surface $(\S,h)= \H^2/\rho(\pi_1\S)$. Fix a pair of pants decomposition $([\gamma_i])_{1\leq i\leq 3g-3}$ of $\S$. Let $\pant \subset (\S,h)$ be a hyperbolic pair of pants given by that decomposition, with fundamental group $\pi_1\pant = \langle \gamma,\eta\rangle$ generated by two elements $\gamma,\eta \in \pi_1\S$ which are elements of the pair of pants decomposition.

Consider a smooth family of Fuchsian representations $(\rho_s)_{s\in I}$ such that $\rho_0=\rho$ and satisfying:
\begin{enumerate}[label=(\Alph*)]
 \item for all $s \in I$, $\rho_s(\gamma)$ is a positive hyperbolic translation along the same oriented geodesic $C_\gamma \subset \H^2 $ independent of $s$ (see Figure~\ref{fig universal pants s}), \label{cdt A}
 \item for all $s \in I$, $\rho_s(\eta)$ is a positive hyperbolic translation along an oriented geodesic $C_\eta^s \subset \H^2$ orthogonal to a geodesic $C$, which independent of $s$ and orthogonal to $C_\gamma$ (see Figure~\ref{fig universal pants s}). \label{cdt B}
\end{enumerate}

Given such a path of representations, we shall use the following notations. For all $s \in I$, $\pant_s$ is the hyperbolic pair of pants in $(\S,h_s) \coloneqq \H^2/\rho_s(\pi_1\S)$ given by the same boundary loop classes as $\pant$, that is $\gamma$, $\eta$ and $\eta\gamma$ . For all $s\in I$, $\upant_s\subset \H^2$ is the connected component of $\pi^{-1}(\pant_s)$ having $C_\gamma$ and $C^s_\eta$ as boundaries. 

\begin{figure}[htbp]
\centering
\includegraphics{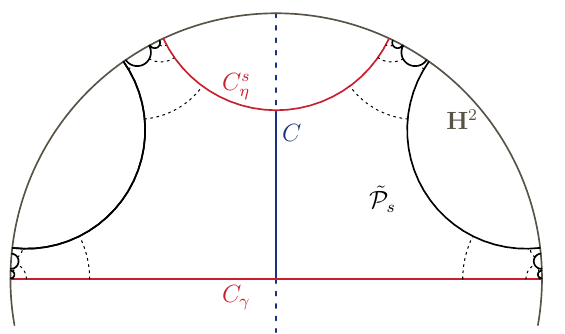}
\caption{The universal cover of the hyperbolic pair of pants $\pant_s$ in the Poincar\'e disk model of $\H^2$.}
\label{fig universal pants s}
\end{figure}

We also denote by $\sigma_s$ the orientation-reversing isometric symmetry on $\pant_s$ exchanging the two right-angled hyperbolic hexagons composing it. A lift $\tilde{\sigma}_s$ of $\sigma_s$ to $\upant_s$ is the restriction of $S \in \Isom^-(\H^2)$ the hyperbolic reflection along $C$ on $\H^2$ (see Figure~\ref{fig universal pants s}).

\begin{lemma}
\label{lem central result}
Let $(\rho_s)_{s\in I}$ be a smooth path of Fuchsian representations of $\pi_1\S$ satisfying conditions \ref{cdt A} and \ref{cdt B} and let the notations be as above. Let $p_0 \in C \cap \tilde{\pant}_0$ and let $\alpha \in Z^1_{\mbox{\tiny \textup{dR}}}(\S,E_{\rho_0})$. Assume that 
\begin{equation*}
 \Psi_{\rho_0}\dR{\alpha} = \ddso[\rho_s], 
\end{equation*}
i.e. that there exist $x \in \isom(\H^2)$ such that for all $\xi \in \pi_1\S$
\begin{equation}
\label{eq alpha ddso rho}
 \int_{p_0}^{\xi \cdot p_0}\alphat = \ddso \rho_s(\xi)\rho_0(\xi)^{-1} + \Ad \rho_0(\xi) \cdot x- x\, .
\end{equation}
Then, for all $\xi \in \{\gamma^{\pm1}, \eta^{\pm1}\}$, 
\begin{equation*}
 \int_{p_0}^{\xi \cdot p_0}(\tilde \alpha_{\pant_0}^* - \alphat_{\pant_0}) = \Ad \rho_0(\xi) \cdot y- y\, ,
\end{equation*}
where : \begin{itemize}
 \item $\alpha_{\pant_0}$ is the restriction of $\alpha$ to $\pant_0$ (see subsection~\ref{subsec Induced de Rham forms}),
 \item $\alphat_{\pant_0}^*$ is defined by Proposition~\ref{prop alpha* is equiv},
 \item $y \in \isom(\H^2)$ is defined by 
 \begin{equation}
 \label{eq y}
 y \coloneqq \Ad S\cdot x -x \, .
 \end{equation}
\end{itemize} 
\end{lemma}

\begin{remark}
Lemma~\ref{lem central result} actually tells us that for every $\alpha \in \ZdR$, $\alpha_\pant$ and $\alpha_\pant^*$ share the same cohomology class in $\HdRof{\pant}$.
Indeed, using the transitivity of the action of $\Isom^+(\H^2)$ on orthogonal $2$-frames of $\H^2$, one can show that every smooth path $[\rho_s]_{s\in I}$ in $\chifd(\pi_1\S)$ such that $[\rho_0]=[\rho]$ can be represented by a smooth family of Fuchsian representations $(\rho_s)_{s\in I}$ (equal to $\rho_0$ at $s=0$) satisfying conditions \ref{cdt A} and \ref{cdt B}. Then, using Lemma~\ref{lem central result} and the same argument as in the proof of the injectivity of $\Psi_\rho$ (Proposition~\ref{prop Psi is injective}), one gets that for every $\alpha \in \ZdR$, there exists a smooth section $f_\pant$ of $\bundlerest{\pant}$ such that 
\begin{equation*}
\alpha^*_{\pant} - \alpha_{ \pant} = \dif f_\pant \, .
\end{equation*} 
\end{remark}

\begin{proof}[Proof of Lemma~\ref{lem central result}]
Let $\xi \in \{\gamma^{\pm 1},{\eta}^{\pm 1}\}$. First note that for all $s\in I$, using Proposition~\ref{prop hol(l)S} with $\rho_s$, we get
\begin{equation}
\label{eq S xi s-dependent}
 \rho_s(\xi) \circ S = S \circ \rho_s (\xi)^{-1} \, .
\end{equation}
As $p_0 \in C \subset \H^2 = \tilde{\S}$, we have 
\begin{equation}
\label{eq sigma(p)=p}
\tilde \sigma_0(p_0)=p_0 \, .
\end{equation}
Using Proposition~\ref{prop hol(l)S} with $\rho = \rho_0$ and \eqref{eq sigma(p)=p}, we have 
\begin{equation*}
\tilde\sigma_0(\xi \cdot p_0) = \xi ^{-1}\cdot \tilde \sigma_0(p_0) = \xi^{-1} \cdot p_0 \, .
\end{equation*}
Hence, we have
\begin{equation*}
 \int_{p_0}^{\xi \cdot p_0} \alphat^* = \int_{p_0}^{\xi \cdot p_0} \Ad S\circ \alphat \circ \dif \tilde\sigma_0 = \Ad S \cdot \int_{\tilde\sigma_0 (p_0)}^{\tilde\sigma_0(\xi \cdot p_0)} \alphat = \Ad S \cdot \int_{p_0}^{\xi^{-1} \cdot p_0} \alphat \, . 
\end{equation*}
Using equation~\eqref{eq alpha ddso rho}, that becomes
\begin{equation}
\label{eq alpha* and AdS}
 \int_{p_0}^{\xi \cdot p_0} \alphat^* = \Ad S \cdot \ddso \rho_s(\xi)^{-1}\rho_0(\xi)+ \Ad S\Ad \rho_0(\xi)^{-1}\cdot x -\Ad S\cdot x \, . 
\end{equation}
The identity~\eqref{eq S xi s-dependent} and the fact that $S^2 = \Id$ imply that
\begin{equation*}
 \Ad S \cdot \ddso \rho_s(\xi)^{-1}\rho_0(\xi)= \ddso S \rho_s(\xi) ^{-1}\rho_0(\xi)S = \ddso \rho_s(\xi) S^2\rho_0(\xi)^{-1} =\ddso \rho_s(\xi)\rho_0(\xi)^{-1} \, ,
\end{equation*}
and that
\begin{equation*}
\Ad S\Ad \rho_0(\xi)^{-1}\cdot x -\Ad S\cdot x =\Ad \rho_0(\xi) \Ad S\cdot x -\Ad S\cdot x \, . 
\end{equation*}
Equation~\eqref{eq alpha* and AdS} thus becomes
\begin{equation*}
 \int_{p_0}^{\xi \cdot p_0} \alphat^* = \ddso \rho_s(\xi)\rho_0(\xi)^{-1}+ \Ad \rho_0(\xi) \Ad S \cdot x  -\Ad S \cdot x \, . 
\end{equation*}
Hence, remembering equation~\eqref{eq alpha ddso rho}, we have just shown that for all $\xi \in \{\gamma^{\pm 1},{\eta}^{\pm 1}\}$,
\begin{equation*}
 \int_{p_0}^{\xi \cdot p_0}\alphat^* - \alphat = \Ad \rho_0(\xi)\cdot y -y \, ,
\end{equation*}
where $y = \Ad S \cdot x -x \in \isom(\H^2)$.
\end{proof}

\section{Involutions on tangent spaces of Teichmüller space}
\label{sec FN involution}

In this section, we shall prove the following.

\begin{theorem}
\label{theo existence involution on tangent space}
Let $([\gamma_i])_{1\leq i\leq 3g-3}$ be a pair of pants decomposition of $\S$. Then, for every $[\rho] \in \chifd(\pi\S)$, there exists a linear involution 
\begin{equation*}
{}^*: \T_{[\rho]}\chifd(\pi\S) \longrightarrow
 \T_{[\rho]}\chifd(\pi\S) \, ,
\end{equation*}
such that, in the de Rham cohomology model for $\T_{[\rho]}\chifd(\pi_1\S)$, for each hyperbolic pair of pants $\pant \subset \H^2/\rho(\pi_1\S)$ given by the decomposition, the restriction of the linear involution ${}^*$ to $\HdRof{\pant}$ is the one induced by the orientation-reversing isometric symmetry on $\pant$ (defined in Section~\ref{sec Induced de Rham forms and symmetries of pairs of pants})

Moreover, for all $[\tau],[\tau'] \in \T_{[\rho]}\chifd(\pi_1\S)$, one has
\begin{equation}
\label{eq wg and *}
 \wg([\tau]^*,[\tau']^*) = - \wg([\tau],[\tau']) \, .
\end{equation}
\end{theorem}

\begin{remark}
Notice that, given a pair of pants decomposition $([\gamma_i])_{1\leq i\leq 3g-3}$, for a generic $[\rho] \in \chifd(\pi_1\S)$, the orientation-reversing isometric symmetry on hyperbolic pants of $\H^2/\rho(\pi_1\S)$ given by decomposition do not glue together into a global isometric symmetry (it only happens if, in associated Fenchel--Nielsen coordinates, one has $t_i[\rho] \in (l_i[\rho]/2)  \Z$ for all $1\leq i\leq 3g-3$). Nevertheless, the involution from Theorem~\ref{theo existence involution on tangent space} is defined on the tangent space at any point of the Teichm\"uller space.
\end{remark}

\begin{remark}
The involution ${}^*$ from Theorem~\ref{theo existence involution on tangent space} depends on the pair of pants decomposition $\Gamma \coloneqq ([\gamma_i])_{1\leq i\leq 3g-3}$ and on the point $[\rho] \in \chifd(\pi_1\S)$. Setting $\mathrm{Inv}^{\Gamma}_{[\rho]} \coloneqq  {}^*$, most of Theorem~\ref{theo intro existence involution on tangent space} is proved with Theorem~\ref{theo existence involution on tangent space}. The remaining part, i.e. the fact that $\mathrm{Inv}^{\Gamma}$ is a smooth tensor on $\chifd(\pi_1\S)$, will be proved in Section~\ref{Sec Formula for infinitesimal lengths} (see Corollary~\ref{coro inv is tensor}).
\end{remark}

We will construct the involution in subsection~\ref{subsec global symmetry}. By construction, it will restrict to the involutions from Section~\ref{sec Induced de Rham forms and symmetries of pairs of pants} on hyperbolic pair of pants. Then, in subsection~\ref{subsec relation with goldman form}, we will show, with Proposition~\ref{prop wg(*,*)}, that it satisfies \eqref{eq wg and *}, concluding the proof of Theorem~\ref{theo existence involution on tangent space}.

For the rest of this section, we fix a pair of pants decomposition $([\gamma_i])_{1\leq i\leq 3g-3}$ of $\S$ and consider the Fenchel--Nielsen coordinates on $\chifd(\pi_1\S)$ given by that pair of pants decomposition and a multicurve $\Gamma'$. We also let $\rho$ be a Fuchsian representation of $\pi_1\S$ and consider the hyperbolic surface $(\S,h) = \H^2/\rho(\pi_1\S)$ with universal cover $\pi: \H^2 \to \H^2/\rho(\pi_1\S)$. 

\subsection{\texorpdfstring{Parametrisation of hyperbolic annuli and property $(\mathrm{P}_{\epsind})$}{Parametrisation of hyperbolic annuli and property (P epsilon)}}
\label{subsec param annulus}

The first problem one faces when trying to define a global involution on $\HdR$ from all the symmetries on hyperbolic pair of pants is the following: given $\alpha \in \ZdR$, all the $\alpha_\pant^*$ on the pair of pants (defined in Section~\ref{sec Induced de Rham forms and symmetries of pairs of pants}) have no reason to glue together into a global $\bundle$-valued $1$-form on $\S$. In this subsection, we shall introduce a property ensuring that this gluing is possible.

\subsubsection{Parametrisation of hyperbolic annuli}

First let us start by describing a nice parametrisation of hyperbolic annuli. Let $\gamma$ be a simple element of $\pi_1\S$ and set $\varepsilon>0$ small enough so that $N^{\gamma}_\varepsilon$, the $\varepsilon$-neighbourhood of the simple closed geodesic curve $c^h_\gamma$ in $(\S,h) =\H^2/\rho(\pi_1\S)$, is an annulus. Then:
\begin{itemize}
 \item the geodesic parametrisation of the oriented simple closed curve $c^h_\gamma$ gives a smooth diffeomorphism $\varphi_\gamma : c^h_\gamma \to \R/l_\gamma[\rho]\Z$ (unique up to translation in $\R/l_\gamma[\rho]\Z$), where $l_\gamma[\rho]$ is the length of $c^h_\gamma$,
 \item the orthogonal projection on $c^h_\gamma$ is a smooth map $\Pi_\gamma : N^{\gamma}_\varepsilon \to c^h_\gamma$,
 \item the signed distance to $c^h_\gamma$ is a smooth map $\delta_\gamma : N^{\gamma}_\varepsilon \to (-\varepsilon,\varepsilon)$.
\end{itemize}
Introducing $\theta_\gamma \coloneqq \varphi_\gamma \circ \Pi_\gamma$, we get the following smooth diffeomorphic parametrisation of the annulus $N^{\gamma}_\varepsilon$:
\begin{equation*}
(\theta_\gamma, \delta_\gamma) : N^{\gamma}_\varepsilon \longrightarrow \R/l_\gamma[\rho]\Z \, \times (-\varepsilon,\varepsilon) \, .
\end{equation*}

\subsubsection{\texorpdfstring{Property $(\mathrm{P}_{\epsind})$}{Property (P epsilon)}}

We recall that we have fixed a pair of pants decomposition $([\gamma_i])_{1\leq i\leq 3g-3}$ and that we denote by $\pi$ the universal covering $\pi: \tilde{\S} = \H^2 \to \H^2/\rho(\pi_1\S)$. Given $\varepsilon>0$ such that the $\varepsilon$-neighbourhoods $(N^{\gamma_i}_\varepsilon)_{1\leq i\leq 3g-3}$ of the geodesic curves $(c^h_{\gamma_i})_{1\leq i\leq 3g-3}$ in $(\S,h) =\H^2/\rho(\pi_1\S)$ are pairwise disjoint annuli, we shall say that a de Rham form $\alpha \in \ZdR$ satisfies property \ref{property P epsilon} if: 
\begin{itemize}
 \myitem{$(\mathrm{P}_{\epsind})$}\label{property P epsilon} For all $1 \leq i \leq 3g-3$, on $N^{\gamma_i}_\varepsilon$, $\alpha$ is of the form: for all $p\in \pi^{-1}(N^{\gamma_i}_\varepsilon)$ and $v \in \T_{\pi(p)}\S$
\begin{equation*}
\alpha (\pi(p),v) = \mu_{\gamma_i} [p, \tH (C_{\gamma_i,p})] \, \dif_{\pi(p)} \theta_{\gamma_i}(v) \, ,
\end{equation*}
where $\mu_{\gamma_i} \in \R$, $C_{\gamma_i,p}$ is the oriented lift of $c^h_{\gamma_i}$ at distance less than $\varepsilon$ to $p \in \tilde{\S} = \H^2$, and $\theta_{\gamma_i}$ is given by the annulus parametrisation of $N^{\gamma_i}_\varepsilon$ described above.
\end{itemize}

\begin{proposition}
\label{prop peps imply gluing}
Let $\alpha \in \ZdR$ satisfying property~\ref{property P epsilon} for some $\varepsilon>0$.
Then, for every hyperbolic pair of pants $\pant \subset (\S,h)$ given by the pair of pants decomposition, the de Rham form $\alpha_\pant^*-\alpha_\pant \in \bundlerest{\pant}$ vanishes on the $\varepsilon$-neighbourhood of every boundary component of $\pant$. 
\end{proposition}

Before proving Proposition~\ref{prop peps imply gluing}, let us note its consequences. If a de Rham form $\alpha \in \ZdR$ satisfies property \ref{property P epsilon}, then all the $\alpha_\pant^*$ glue together into a global de Rham form on $\S$, that we shall denote $\alpha^* \in \ZdR$. Moreover, for every $\gamma \in \pi_1\S$ in the fixed pair of pants decomposition, we have 
\begin{equation*}
 \alpha_{\vert N_{\gamma}^\varepsilon} = \alpha^*_{\vert N_{\gamma}^\varepsilon} \, .
\end{equation*}

\begin{proof}[Proof of Proposition~\ref{prop peps imply gluing}]
Let $\pant \subset (\S,h)$ be a hyperbolic pant given by the pair of pants decomposition and $\gamma \in \pi_1\S$ be an element of the decomposition so that $c^h_\gamma \subset \partial \pant$.

As $\alpha \in \ZdR$ satisfies property~\ref{property P epsilon}, on $N_{\gamma}^\varepsilon$, $\alpha$ is of the form
\begin{equation*}
\alpha (\pi(p),v) = \mu_\gamma [p, \tH (C_{\gamma,p})] \, \dif_{\pi(p)} \theta_\gamma(v) \, ,
\end{equation*}
where $\mu_{\gamma} \in \R$, $C_{\gamma,p}$ is the oriented lift of $c^h_\gamma$ at distance less than $\varepsilon$ to $p \in \pi^{-1}(N_{\gamma}^\varepsilon) \subset \H^2$, and $\theta_\gamma$ is given by the annulus parametrisation of $N_{\gamma}^\varepsilon$. Note the two following facts:
\begin{itemize}
 \item As $c^h_\gamma$ bounds the hyperbolic pair of pants $\pant$, in annulus coordinates of $N_\varepsilon^\gamma \cap \pant$, the isometric symmetry $\sigma$ of $\pant$ acts by 
\begin{equation*}
 \sigma(\theta_\gamma, \delta_\gamma) = \sigma(\vartheta-\theta_\gamma, \delta_\gamma) \, 
\end{equation*}
where $\vartheta \in \R/l_\gamma[\rho] \Z$ is a constant. Hence, 
\begin{equation*}
 \dif \theta_\gamma \circ \dif \sigma = - \dif \theta_\gamma \, .
\end{equation*}
\item If $S \in \Isom^-(\H^2)$ is the hyperbolic reflection along the geodesic $C$ orthogonal to $C_\gamma$, then, using the $\O_+(2,1)$ model of $\Isom(\H^2)$, one easily checks that 
\begin{equation}
\label{eq Ad(S) t= -t}
\Ad S \cdot \tH(C_\gamma) = - \tH(C_\gamma)\, .
\end{equation}
\end{itemize}
Thus, by definition of $\alpha_\pant^*$ (Proposition~\ref{prop alpha* is equiv}), on $N_\varepsilon^\gamma \cap \pant$ we have that 
\begin{align*}
\alpha^*_\pant (\pi(p),v) & = \mu_\gamma [p, \Ad S_p \cdot \tH (C_{\gamma,p})] \, \dif_{\sigma \circ \pi(p)} \theta_\gamma \circ \dif_{\pi(p)}\sigma(v) \\
& = \mu_\gamma [p,\tH (C_{\gamma,p})] \, \dif_{\pi(p)}\theta_\gamma (v) \\
& = \alpha_\pant (\pi(p),v) \, ,
\end{align*}
where $S_p \in \Isom(\H^2)$ is a hyperbolic reflection along a geodesic orthogonal to $C_{\gamma,p}$.
\end{proof}

\begin{proposition}
\label{prop explicit value nice rep}
If $\alpha \in \ZdR$ satisfies property~\ref{property P epsilon}, then each coefficient $\mu_{\gamma_i}$ in the definition of property~\ref{property P epsilon} is 
\begin{equation*}
 \mu_{\gamma_i} = \frac{1}{l_{\gamma_i}[\rho]} \dif_{[\rho]} l_{\gamma_i} \bp{\Psi_\rho \dR{\alpha}} \, .
\end{equation*}
\end{proposition}
\begin{proof}
Let $\gamma = \gamma_i$ be an element of the pair of pants decomposition. Let $\alpha_\gamma$ be a de Rham representative of $\Psi_\rho^{-1}(\partial_{t_\gamma})$ given by the expression \eqref{eq twist in dR} and supported in $N_\varepsilon^\gamma>0$. 
On one side, by Proposition~\ref{prop wg=wdR} and Theorem~\ref{Second Wolpert formula}, we have
\begin{equation}
\label{eq 1 proof mu}
 2 \, \dif_{[\rho]} l_\gamma \bp{\Psi_\rho \dR{\alpha}} = \wdR(\dR{\alpha_\gamma},\dR{\alpha}) \, . 
\end{equation}
On the other side, using that $\alpha$ satisfies property~\ref{property P epsilon}, and the equation~\eqref{eq twist in dR} defining $\alpha_\gamma$, we have
\begin{equation}
\label{eq 2 proof mu intermediate}
 \wdR(\dR{\alpha_\gamma},\dR{\alpha}) = \mu_\gamma \int_{N_\varepsilon^\gamma} \kil\bp{\tH(C_{\gamma}), \tH(C_{\gamma})} \, \dif (f\circ \delta_\gamma) \wedge \dif \theta_\gamma \, ,
\end{equation} 
where $C_\gamma$ is the oriented geodesics along which $\rho(\gamma)$ is a positive translations, $(\theta_\gamma, \delta_\gamma)$ is the parametrisation of the annulus $N_\varepsilon^\gamma$ described subsection~\ref{subsec param annulus}, and $f:(-\varepsilon,\varepsilon) \to \R$ is a smooth function such that $\int f'=1$. By Lemma~\ref{lem cos and killing}, we also have 
\begin{equation*}
    \kil\bp{\tH(C_{\gamma}), \tH(C_{\gamma})}=2 \, .
\end{equation*}
Thus, equation~\eqref{eq 2 proof mu intermediate} becomes 
\begin{equation}
\label{eq 2 proof mu}
 \wdR(\dR{\alpha_\gamma},\dR{\alpha}) = \mu_\gamma\kil\bp{\tH(C_{\gamma}), \tH(C_{\gamma})}\int_{(-\varepsilon,\varepsilon)}f'\int_{\R/l_\gamma[\rho]\Z}\dif l= 2 \mu_\gamma l_\gamma[\rho ] \, .
\end{equation} 
Identities~\eqref{eq 1 proof mu} and \eqref{eq 2 proof mu} thus conclude the proof.
\end{proof}

\begin{corollary}
\label{cor coincide on neigh}
Let $\alpha_0,\alpha_1 \in \ZdR$ be representatives of the same de Rham class $\dR{\alpha} \in \HdR$. If there exists $\varepsilon>0$ such that both $\alpha_0$ and $\alpha_1$ satisfy property~\ref{property P epsilon}, then for every $\gamma$ in the fixed pair of pants decomposition, $\alpha_0$ and $\alpha_1$ are equal on $N_\varepsilon^\gamma$, the $\varepsilon$-neighbourhood of $c_\gamma^h$.
\end{corollary}

\subsection{\texorpdfstring{Existence of a representative satisfying property $(\mathrm{P}_{\epsind})$}{Existence of a representative satisfying property (P epsilon}}

The goal of this subsection is to prove the following fact.

\begin{proposition}
\label{prop exist rep satisfying P eps}
Every de Rham class $\dR{\alpha} \in \HdR$ admits a representative $\alpha \in \ZdR$ satisfying property~\ref{property P epsilon} for some $\varepsilon>0$.
\end{proposition}

First let us prove some intermediate results.

\begin{lemma}
\label{lem decomp isom via Ad}
Let $M=T_t(C) \in \Isom(\H^2)$ be the hyperbolic translation by $t \in \R$ along an oriented geodesic line $C$. Then, $\Ad M \in \mathrm{GL}(\isom(\H^2))$ is diagonalisable and the $3$-dimensional vector space $\isom(\H^2)$ decomposes as the following direct sum: 
\begin{equation*}
 \isom(\H^2) = \ker\bp{\Ad  M  - \Id} \newoplus \ker\bp{\Ad  M  - e^{t}\, \Id} \newoplus \ker\bp{\Ad  M  - e^{-t}\, \Id} \, ,
\end{equation*}
with $\ker(\Ad M  - \Id) = \mathrm{span}(\tH(C))$.
\end{lemma}
\begin{proof}
Let us use the $\SO_0(2,1)$ model of $\Isom^+(\H^2)$. Up to conjugation by an element of $\SO_0(2,1)$, we may assume that 
\begin{equation*}
M= T_t(C) = \left( \begin{matrix}
1 & 0 & 0\\
0 & \cosh t & \sinh t \\
0 & \sinh t & \cosh t
\end{matrix} \right)
\quad \text{and} \quad
\tH(C) = \left( \begin{matrix}
0 & 0 & 0\\
0 & 0 & 1 \\
0 & 1 & 0
\end{matrix} \right) \, .
\end{equation*}
Then, setting
\begin{equation*}
v = \left( \begin{matrix}
0 & -1 & 1\\
1 & 0 & 0 \\
1 & 0 & 0
\end{matrix} \right)
\quad \text{and} \quad
w = \left( \begin{matrix}
0 & 1 & 1\\
-1 & 0 & 0 \\
1 & 0 & 0
\end{matrix} \right) \, ,
\end{equation*}
one notices that $(\tH(C),v,w)$ is a basis of $\so(2,1)$ such that
\begin{align*}
& \Ad M  \cdot \tH(C) = \tH(C)\\
& \Ad M  \cdot v = e^t v \\
& \Ad M  \cdot w = e^{-t} w \, . \qedhere
\end{align*}
\end{proof}

As the kernel and image of a diagonalisable linear endomorphism are in direct sum, Lemma~\ref{lem decomp isom via Ad} implies the following.

\begin{corollary}
\label{cor th along gamma}
Let $M=T_t(C) \in \Isom(\H^2)$ be the hyperbolic translation by $t \in \R$ along an oriented geodesic line $C$. Then, we have
\begin{equation*}
\isom(\H^2) = \ker(\Ad M - \Id) \oplus \Ima (\Ad M - \Id) \, ,
\end{equation*}
and thus, for all $x \in \isom(\H^2)$, there exists $\lambda \in \R$ and $y \in \isom(\H^2)$ such that
\begin{equation*}
    x= \lambda \tH(C) + \Ad M \cdot y - y \, .
\end{equation*}
\end{corollary}

We now have the tools needed to prove Proposition~\ref{prop exist rep satisfying P eps}.

\begin{proof}[Proof of Proposition~\ref{prop exist rep satisfying P eps}]
We recall that we have fixed a pair of pants decomposition $([\gamma_i])_{1\leq i \leq 3g-3}$. Let $\varepsilon>0$ small enough so that the $2\varepsilon$-neighbourhoods $(N^{\gamma_i}_{2\varepsilon} )_{1\leq i \leq 3g-3}$ are pairwise disjoint hyperbolic annuli.

Let $1\leq i \leq 3g-3$ and $p_i$ be an arbitrary base point on $C_{\gamma_i} \subset \H^2$, the oriented geodesic line along which $\rho(\gamma_i)$ is a positive translation. By Corollary~\ref{cor th along gamma}, there exist $\mu_{\gamma_i} \in \R$ and $x_i \in \Isom^+(\H^2)$ such that 
\begin{equation}
\label{eq t(gamma_i)}
 \int_{p_i}^{\gamma_i\cdot p_i} \alphat = \mu_{\gamma_i} \tH(C_{\gamma_i})+ \Ad \rho(\gamma_i)\cdot x_i-x_i \, .
\end{equation}
Using the universal cover of $N^{\gamma_i}_\varepsilon$ given by $\tilde{N}^{\gamma_i}_\varepsilon$, the $\varepsilon$-neighbourhood of $C_{\gamma_i}$ in $\tilde{\S} = \H^2$, the form $\alpha$ induces $\alpha_{i} \in \ZdRof{N^{\gamma_i}_\varepsilon}$ (see subsection~\ref{subsec Induced de Rham forms}). Let $\alpha'_i \in \ZdRof{N^{\gamma_i}_\varepsilon}$ be the form defined by for all $(p,v) \in \T\tilde{N}^{\gamma_i}_\varepsilon$
\begin{equation*}
\alpha_i' (\pi(p),v) = \mu_{\gamma_i} [p, \tH (C_{\gamma_i,p})] \dif_{\pi(p)} \theta_{\gamma_i}(v) \, .
\end{equation*}
We have 
\begin{equation}
\label{eq t'(gamma_i)}
 \int_{p_i}^{\gamma_i\cdot p_i} \alphat'_i = \mu_{\gamma_i}  \tH(C_{\gamma_i}) \, .
\end{equation}
Using that $\pi_1 (N_i^\varepsilon) = \langle \gamma_i \rangle$, the equalities \eqref{eq t(gamma_i)} and \eqref{eq t'(gamma_i)}, and the same argument as in the proof of the injectivity of $\Psi_\rho$ (Proposition~\ref{prop Psi is injective}), we get a section $f_i$ of $\bundlerest{N_{i}^\varepsilon}$ such that
\begin{equation*}
 \alpha_i' = \alpha_i + \dif f_i \, .
\end{equation*}
One can easily extend the section $f_i$ of $\bundlerest{N_{i}^\varepsilon}$ to a smooth section $\bar{f}_i$ of $\bundle$ supported in $N_{i}^{2\varepsilon}$. Then, the $1$-form $\alpha' \in \ZdR$ defined by
\begin{equation*}
 \alpha' = \alpha + \sum_{i=1}^{3g-3} \dif \bar{f}_i 
\end{equation*}
satisfies property~\ref{property P epsilon}. 
\end{proof}

\subsection{A global involution on cocycles}
\label{subsec global symmetry}

We can now define a linear involution on the whole $\HdR$ in the following way. Let $\dR{\alpha} \in \HdR$ with representative $\alpha$ satisfying property~\ref{property P epsilon} for some $\varepsilon>0$. By Proposition~\ref{prop exist rep satisfying P eps}, such a representative exists for every class in $\HdR$. Then, by Proposition~\ref{prop peps imply gluing}, all the $\alpha_\pant^*$ glue together to give an element of $\ZdR$ that we denote by $\alpha^*$, so that we set 
\begin{equation*}
 \dR{\alpha}^* \coloneqq \dR{\alpha^*} \, .
\end{equation*}

\begin{proposition}
The process described above is well-defined and gives a linear involution
\begin{equation*}
{}^*: \HdR \longrightarrow
 \HdR \, .
\end{equation*}
\end{proposition}
\begin{proof}
First, note that if ${}^*: \HdR \to \HdR$ is well-defined, it is clearly linear and an involution.

In order to prove that it is well-defined, we have to show that $\dR{\alpha}^* = \dR{\alpha^*}$ does not depend on the choice of the representative $\alpha$ satisfying property~\ref{property P epsilon}.

Assume that $\alpha_0,\alpha_1 \in \ZdR$ are such that $\dR{\alpha_0} = \dR{\alpha_1}$ and that they respectively satisfy property~{\renewcommand\epsind{\varepsilon_0}\ref{property P epsilon}} and {\renewcommand\epsind{\varepsilon_1}\ref{property P epsilon}} with $\varepsilon_0,\varepsilon_1>0$. Setting $\varepsilon = \min(\varepsilon_0, \varepsilon_1)$, we have that $\alpha_0$ and $\alpha_1$ both satisfy property~\ref{property P epsilon}.

As $\dR{\alpha_0} = \dR{\alpha_1}$, there exists a smooth section $f$ of $\bundle$ such that $\alpha_0-\alpha_1=\dif f$. By Lemma~\ref{lem * and d commute}, on each hyperbolic pair of pants $\pant \subset (\S,h)$ given by the pair of pants decomposition, we have
\begin{equation*}
 (\alpha_0)_\pant^*-(\alpha_1)_\pant^* = (\dif f_\pant)^* = \dif (f_\pant^*) \, .
\end{equation*}
Thus, if  we show that all the $f_\pant^*$ glue together to a section of $\bundle$, we will get $\dR{\alpha_0^*}=\dR{\alpha_1^*}$, concluding the proof.

Let $\gamma$ in the fixed pair of pants decomposition and $\pant \subset(\S,h)$ be a hyperbolic pair of pants given by the decomposition having $c_\gamma^h$ as one of its boundary components. As $\alpha_0$ and $\alpha_1$ both satisfy property~\ref{property P epsilon} and $\dR{\alpha_0}=\dR{\alpha_1}$, by Corollary~\ref{cor coincide on neigh}, $\dif f = \alpha_0-\alpha_1$ vanishes on $N_\varepsilon^\gamma$, and thus, $f$ is constant on $N_\varepsilon^\gamma$. 

Using the definition of the bundle $\bundle$ (or equivalently the $\Ad \rho $-equivariance~\eqref{eq ad equiv map} of the $\isom(\H^2)$-valued map $\tilde f$ on $\tilde\S = \H^2$ inducing $f$) and Corollary~\ref{cor th along gamma}, one notices that there must exist a constant $\lambda_{\gamma} \in \R$ such that for all $p\in \pi^{-1}(N^\gamma_\varepsilon)$,
\begin{equation*}
f\bp{\pi(p)} = \lambda_\gamma\big[p,\tH(C_{\gamma,p})\big] \in \bundle \, , 
\end{equation*}
where $C_{\gamma,p} \subset \H^2$ is the oriented lift of $c_\gamma$ at distance less than $\varepsilon$ to $p$.

If $S \in \Isom^-(\H^2)$ is the hyperbolic reflection along a geodesic $C$ orthogonal to the oriented geodesic $C_{\gamma,p}$, by \eqref{eq Ad(S) t= -t}, we have
\begin{equation*}
\Ad S \cdot \tH(C_{\gamma,p}) = - \tH(C_{\gamma,p})\, .
\end{equation*}
Thus, by definition of the involution on forms on a hyperbolic pair of pants $\pant$ of the decomposition (Proposition~\ref{prop alpha* is equiv}), all the $(f_\pant)^*$ glue together to give a global section $f^*$ of $\bundle$ such that for all $1\leq i \leq 3g-3$ and $p\in \pi^{-1}( N^{\gamma_i}_\varepsilon)$,
\begin{equation*}
 f^*\bp{\pi(p)} = -\lambda_{\gamma_i}\big[p,\tH(C_{\gamma_i,p})\big] \, .
\end{equation*}
That concludes the proof.
\end{proof}

\begin{remark}
\label{rem * depend on pop dec}
Note that the linear involution ${}^*$ on $\HdR$ depends on the fixed pair of pants decomposition $([\gamma_i])_{1\leq i \leq 3g-3}$, but not on the multicurve $\Gamma'$ used for defining Fenchel--Nielsen coordinates.
\end{remark}

Let us define the linear involution
\begin{equation*}
 {}^*: \Hisomeq \longrightarrow \Hisomeq 
\end{equation*}
by push-forward via the isomorphism $\Psi_\rho$. That is, for all $[\tau] \in \Hisom$,
\begin{equation*}
 [\tau]^*\coloneqq \Psi_\rho(\dR{\alpha}^*) \, 
\end{equation*}
where $\dR{\alpha} = \Psi_\rho^{-1}[\tau] \in \ZdR$.

\subsection{Relation with the Goldman form}
\label{subsec relation with goldman form}

\begin{proposition}
\label{prop wdR(*,*)}
Let $\dR{\alpha},\dR{\alpha'} \in \ZdR$. Then,
\begin{equation*}
 \wdR(\dR{\alpha}^*,\dR{\alpha'}^*) = - \wdR(\dR{\alpha},\dR{\alpha'})
\end{equation*}
\end{proposition}
\begin{proof}
By Proposition~\ref{prop exist rep satisfying P eps}, we can assume that we consider representatives $\alpha$ and $\alpha'$ satisfying property~\ref{property P epsilon} for some $\varepsilon>0$. Then, by definition, $\dR{\alpha}^* = \dR{\alpha^*}$ and $\dR{\alpha'}^* = \dR{\alpha'^*}$.
Let $P$ be the set of hyperbolic pair of pants given by the fixed pair of pants decomposition. We have
\begin{equation*}
 \wdR(\dR{\alpha}^*,\dR{\alpha'}^*) = \wdR(\dR{\alpha^*},\dR{\alpha'^*})= \int_S \kil(\alpha^* \wedge \alpha'^*) = \sum_{\pant \in P} \int_\pant \kil(\alpha^*_\pant \wedge \alpha'^*_\pant) \, .
\end{equation*}
Let $\pant \in P$. The Killing form $\kil$ is invariant under the adjoint action of $\Isom(\H^2)$, so for all $S \in \Isom(\H^2)$ and $x,y \in \isom(\H^2)$,
\begin{equation*}
 \kil(\Ad S \cdot x, \Ad S \cdot y) = \kil(x,y) \, .
\end{equation*}
Hence, by definition of $\alpha_\pant^*$ and $\alpha_{\pant}'^*$ (Proposition~\ref{prop alpha* is equiv}), we have 
\begin{equation*}
 \int_\pant \kil(\alpha^*_\pant \wedge \alpha'^*_\pant) = \int_\pant \kil\bp{(\alpha_\pant \circ \dif\sigma) \wedge (\alpha'_\pant \circ \dif\sigma)}  = \int_\pant \sigma^*\bp{\kil(\alpha_\pant \wedge \alpha'_\pant )}\, ,
\end{equation*}
where $\sigma$ is the orientation-reversing isometric symmetry on $\pant$ exchanging the two right-angled hyperbolic hexagons composing it. As $\sigma$ is orientation-reversing, using the change of variables formula for differential forms that becomes
\begin{equation*}
 \int_\pant \kil(\alpha^*_\pant \wedge \alpha'^*_\pant) = - \int_\pant \kil(\alpha_\pant \wedge \alpha'_\pant ) \, .
\end{equation*}
Thus, 
\begin{equation*}
 \wdR(\dR{\alpha}^*,\dR{\alpha'}^*) = \sum_{\pant \in P} \int_\pant \kil(\alpha^*_\pant \wedge \alpha'^*_\pant) = -\sum_{\pant \in P} \int_\pant \kil(\alpha_\pant \wedge \alpha'_\pant ) = - \wdR(\dR{\alpha},\dR{\alpha'}) \, . \qedhere
\end{equation*}
\end{proof}

As the involution ${}^*$ on $\Hisom$ is defined as the push-forward by $\Psi_\rho$ of the one on $\HdR$, Propositions~\ref{prop wg=wdR} and \ref{prop wdR(*,*)} imply the following.

\begin{proposition}
\label{prop wg(*,*)}
Let $[\tau],[\tau'] \in \T_{[\rho]}\chifd(\pi_1\S) = \Hisom$. Then,
\begin{equation*}
 \wg([\tau]^*,[\tau']^*) = - \wg([\tau],[\tau']) \, .
\end{equation*}
\end{proposition}

\begin{corollary}
\label{cor wg(*,*)=0}
Let $[\tau],[\tau'] \in \T_{[\rho]}\chifd(\pi_1\S) = \Hisom$. If $[\tau]=[\tau]^*$ and $[\tau']=[\tau']^*$, then
\begin{equation*}
 \wg([\tau],[\tau']) = 0 \, .
\end{equation*}
\end{corollary}

Let us end this section by noticing the two following facts.

\begin{proposition}
\label{prop dt*=-dt}
Let $\gamma \in \pi_1\S$ be an element of the fixed pair of pants decomposition. Then, the infinitesimal twist $\partial_{t_\gamma} \in \T_{[\rho]}\chifd(\pi_1\S) = \Hisom$ satisfies 
\begin{equation*}
 \partial_{t_\gamma}^* = - \partial_{t_\gamma} \, .
\end{equation*}
\end{proposition}
\begin{proof}
Let $\varepsilon>0$ be small enough so that $N^\gamma_\varepsilon\subseteq \S$, the $\varepsilon$-neighbourhood of $c^h_{\gamma}$ in $(\S, h )= \H^2/\rho(\pi_1\S)$, is an annulus. Let $\alpha_\gamma$ be a de Rham representative of $\Psi_\rho^{-1}(\partial_{t_\gamma})$ given by the expression \eqref{eq twist in dR} where $f:(-\varepsilon,\varepsilon) \to \R$ is a smooth function such that $f(\pm\varepsilon) = \pm1/2$, and $f'$ is compactly supported in $(-\varepsilon,-\varepsilon/2)$. As $f'$ vanishes on $(-\varepsilon/2,\varepsilon/2)$, $\alpha_\gamma$ vanishes on $N_{\varepsilon/2}^\gamma$. Thus, $\alpha_\gamma$ satisfies property~{\renewcommand\epsind{\varepsilon/2}\ref{property P epsilon}} so that $\dR{\alpha_\gamma}^*= \dR{\alpha_\gamma^*}$. 

Let $\pant$ be the hyperbolic pair of pants on which $\alpha_\gamma$ is supported. In annulus coordinates of $N_\varepsilon^\gamma \cap \pant$, the isometric symmetry $\sigma$ of $\pant$ and the distance $\delta$ to $c_\gamma^h$ satisfy
\begin{equation*}
 \delta \circ \sigma = \delta \quad \text{and} \quad  \dif \delta \circ \dif \sigma = \dif \delta \, .
\end{equation*}
Then, using \eqref{eq Ad(S) t= -t}, the expression \eqref{eq twist in dR} of $\alpha_\gamma$, and the definition of $\alpha_\gamma^*$, on $N_\varepsilon^\gamma \cap \pant$ we have
\begin{align*}
\alpha^*_\pant (\pi(p),v) & = \mu_\gamma [p, \Ad S_p \cdot \tH (C_{\gamma,p})] \, \dif_{f \circ \sigma \circ \pi(p)} f \circ \dif_{\sigma \circ \pi(p)} \delta \circ \dif_{\pi(p)}\sigma(v) \\
& = -\mu_\gamma [p,\tH (C_{\gamma,p})] \, \dif_{f \circ \pi(p)} f \circ\dif_{\pi(p)}\delta (v) \\
& = -\alpha_\pant (\pi(p),v) \, ,
\end{align*}
where $S_p \in \Isom(\H^2)$ is a hyperbolic reflection along a geodesic orthogonal to $C_{\gamma,p}$.
\end{proof}

\begin{proposition}
\label{prop tau* - tau}
Let $[\tau] \in \T_{[\rho]}\chifd(\pi_1\S) = \Hisom$. Then,
\begin{equation*}
[\tau]^* - [\tau] \in \mathrm{span}(\partial_{t_{\gamma_i}})_{1\leq i \leq 3g-3} \, .
\end{equation*}
\end{proposition}
\begin{proof}
By Proposition~\ref{prop exist rep satisfying P eps}, we can set $\alpha \in \ZdR$ to be a representative of $\Psi_\rho^{-1}[\tau] \in \HdR$ satisfying property~\ref{property P epsilon} for some $\varepsilon>0$. Let $\gamma \in \pi_1\pant$ be any element of the fixed pair of pants decomposition, and let $\alpha_\gamma \in \ZdR$ be a representative of $\Psi_\rho^{-1}(\partial_{t_\gamma})$ given by the expression \eqref{eq twist in dR} and with support in $N^\gamma_\varepsilon$, the $\varepsilon$-neighbourhood of $c_\gamma^h$. As $\alpha - \alpha^*$ vanishes on $N^\gamma_\varepsilon$, by Proposition~\ref{prop wg=wdR}, we have
\begin{equation*}
 \wg([\tau]^*-[\tau],\partial_{t_\gamma}) = \wdR(\dR{\alpha^*}-\dR{\alpha}, \dR{\alpha_\gamma}) = \int_\S \kil\bp{(\alpha^* - \alpha)\wedge\alpha_\gamma}= 0 \, .
\end{equation*}
Hence, by Theorem~\ref{Second Wolpert formula}, we have just proved that for all $\gamma \in (\gamma_1, \dots, \gamma_{3g-3})$, 
\begin{equation*}
 \dif_{[\rho]}l_\gamma ([\tau]^*-[\tau]) =0 \, .
\end{equation*}
That is $[\tau]^* - [\tau] \in \mathrm{span}(\partial_{t_{\gamma_i}})_{1\leq i \leq 3g-3}$.
\end{proof}

\section{A formula for infinitesimal lengths variations}
\label{Sec Formula for infinitesimal lengths}

The purpose of this section is to prove the following fact.

\begin{theorem}[Theorem~\ref{theo intro w(dl,dl)=0}]
\label{theo w(dl,dl)=0}
If $\gamma,\eta \in \pi_1\S$ are two elements of a pair of pants decomposition of $\S$, in associated Fenchel--Nielsen coordinates on $\chifd(\pi_1\S)$, one has
\begin{equation}
\label{eq w(dl,dl)=0}
\wg(\partial_{l_{\gamma}}, \partial_{l_{\eta}})= 0 \, .
\end{equation}
\end{theorem}

The classical proof of Theorem~\ref{theo w(dl,dl)=0} \cite{Wolpert:1985,Hubbard:2006}, reinterpreted in the language of this article, works as follows. Using the closedness of the Goldman form (or of the Weil--Petersson form, depending on the model of the Teichm\"uller space one is considering) and the Cartan Formula, one shows that one just has to prove the formula at any points $[\rho_0] \in \chifd(\pi_1\S)$ having every twist coordinate $t_{\gamma_i}$ equal to $0$. In that case, the whole hyperbolic surface $(\S,h_0)=\H^2/\rho_0(\pi_1\S)$ is endowed with an orientation-reversing isometric symmetry $\sigma$, obtained by gluing all the symmetries on each hyperbolic pair of pants of the decomposition. Then, using that global isometric symmetry $\sigma$ on $(\S,h_0)$, one can define an involution ${}^*$ on $H^1_{\mbox{\tiny $\Ad \rho_0$}}(\pi_1\S,\isom(\H^2))$ (similarly to what we have done in Section~\ref{sec Induced de Rham forms and symmetries of pairs of pants} with the symmetry on a hyperbolic pair of pants). Then, one shows that for all $[\tau], [\tau'] \in \T_{[\rho_0]}\chifd(\pi_1\S)$,
\begin{equation*}
 \wg([\tau]^*,[\tau']^*) = - \wg([\tau],[\tau']) \, .
\end{equation*}
and that for every $\gamma \in \pi_1\S$ in the fixed pair of pants decomposition, at $[\rho_0]$ one has
\begin{equation}
\label{eq dl*=dl when twist 0}
 \partial_{l_\gamma}^*=\partial_{l_\gamma} \, .
\end{equation}
Then, \eqref{eq w(dl,dl)=0} is satisfied.

\vspace{\baselineskip}

Our proof of Theorem~\ref{theo w(dl,dl)=0} will not rely on the closedness of the Goldman form. With Theorem~\ref{theo existence involution on tangent space}, we have produced, at every point $[\rho] \in \chifd(\pi_1\S)$, an involution ${}^*$ on $\T_{[\rho]}\chifd(\pi_1\S)$ depending on the pair of pants decomposition $\Gamma$, such that for all $[\tau],[\tau'] \in \T_{[\rho]}\chifd(\pi_1\S)$
\begin{equation}
\label{eq wg*=-wg}
 \wg([\tau]^*,[\tau']^*) = - \wg([\tau],[\tau']) \, .
\end{equation}
Then, we will prove (see Theorem~\ref{theo ccl dl*}) that for every element $\gamma \in \pi_1\S$ in the fixed pair of pants decomposition, we have
\begin{equation}
\label{eq sym cocycle}
 \partial_{l_\gamma}+\frac{t_\gamma}{l_\gamma} \partial_{t_\gamma} =\vp{\partial_{l_\gamma}+\frac{t_\gamma}{l_\gamma} \partial_{t_\gamma} }^* \, . 
\end{equation}
Hence, if $\gamma, \eta \in \pi_1 \S$ are two distinct elements of the fixed pair of pants decomposition, equations~\eqref{eq wg*=-wg} and \eqref{eq sym cocycle} imply
\begin{equation*}
 \wg\vp{\partial_{l_\gamma}+\frac{t_\gamma}{l_\gamma} \partial_{t_\gamma} , \partial_{l_\eta}+\frac{t_\eta}{l_\eta} \partial_{t_\eta} } = 0 \, , 
\end{equation*}
so that, using Corollary~\ref{w(dt,dt)=0} and Theorem~\ref{Second Wolpert formula}, we get
\begin{equation*}
 \wg(\partial_{l_\gamma}, \partial_{l_\eta})=\wg\vp{\partial_{l_\gamma}+\frac{t_\gamma}{l_\gamma} \partial_{t_\gamma} , \partial_{l_\eta}+\frac{t_\eta}{l_\eta} \partial_{t_\eta} } = 0 \, . 
\end{equation*}

\begin{remark}
Note that the identity~\eqref{eq sym cocycle} is a quite intuitive fact. Indeed, one expects the ‘‘symmetric way'' to modify the length of a hyperbolic surface along the curve $c^h_\gamma$ to be the one that also proportionally changes the twist along that curve. Moreover, note that if $[\rho]$ has all its twist parameters equal to $0$, we recover the identity \eqref{eq dl*=dl when twist 0} from the classic proof.
\end{remark}

For the rest of this section, we fix a pair of pants decomposition $([\gamma_i])_{1\leq i\leq 3g-3}$ of $\S$ and consider the Fenchel--Nielsen coordinates on $\chifd(\pi_1\S)$ given by that pair of pants decomposition and a multicurve $\Gamma'$. We also let $\rho$ be a Fuchsian representation of $\pi_1\S$ and consider the hyperbolic surface $(\S,h) = \H^2/\rho(\pi_1\S)$ with universal cover $\pi: \H^2 \to \H^2/\rho(\pi_1\S)$.

\subsection{Involutions and infinitesimal lengths}

In this subsection we shall prove the following fact. 

\begin{theorem}
\label{theo ccl dl*}
Let $[\tau] \in \mathrm{span}(\partial_{l_{\gamma_i}})_{1\leq i \leq 3g-3} \subset \T_{[\rho]}\chifd(\pi_1\S)$ and set $(\lambda_i)_{1\leq i \leq 3g-3} \in \R^{3g-3}$, such that 
\begin{equation*}
 [\tau] = \sum_{i=1}^{3g-3} \lambda_i \, \partial_{l_{\gamma_i}} \, .
\end{equation*} 
Then, 
\begin{equation*}
 [\tau]^* - [\tau] = 2 \sum_{i=1}^{3g-3} \lambda_i \frac{t_{\gamma_i}[\rho]}{l_{\gamma_i}[\rho]}\, \partial_{t_{\gamma_i}} \, .
\end{equation*} 
\end{theorem}

Before proving Theorem~\ref{theo ccl dl*}, let us notice some of its consequences.

\begin{corollary}
\label{cor (dl + blabla)*}
Let $\gamma \in \pi_1\S$ be an element of the pair of pants decomposition $([\gamma_i])_{1\leq i \leq 3g-3}$. Then,
\begin{equation*}
 \partial_{l_\gamma}^* = \partial_{l_\gamma} + 2\, \frac{t_{\gamma}[\rho]}{l_{\gamma}[\rho]}\partial_{t_{\gamma}} \, ,
\end{equation*}
so that, by Proposition~\ref{prop dt*=-dt},
\begin{equation*}
 \partial_{l_\gamma}+\frac{t_\gamma[\rho]}{l_\gamma[\rho]} \partial_{t_\gamma} =\vp{\partial_{l_\gamma}+\frac{t_\gamma[\rho]}{l_\gamma[\rho]} \partial_{t_\gamma} }^* \, . 
\end{equation*}
\end{corollary}

\begin{remark}
As noticed in Remark~\ref{rem * depend on pop dec}, the involution ${}^*$ on $\T_{[\rho]}\chifd(\pi_1\S)$ depends only on the pair of pants decomposition $([\gamma_i])_{1\leq i \leq 3g-3}$ and not on the multicurve $\Gamma'$ used in order to define the twists coordinates. Thus, it may seem surprising that the tangent vector
\begin{equation*}
    [\tau]= \partial_{l_\gamma}+\frac{t_\gamma[\rho]}{l_\gamma[\rho]} \partial_{t_\gamma}\, 
\end{equation*}
which seems to depend on the choice of the multicurve $\Gamma'$, could satisfy $[\tau]^*=[\tau]$. In fact it is not, because that tangent vector is actually independent of the choice of the multicurve $\Gamma'$. Indeed, if $(\bar l_{\gamma_i},\bar t_{\gamma_i})_{1\leq i \leq 3g-3}$
are Fenchel--Nielsen coordinates given by the pair of pants decomposition $([\gamma_i])_{1\leq i \leq 3g-3}$ and another multicurve $\bar\Gamma'$, then, using Remark~\ref{remark change of coordinates multicurve} and equation~\eqref{eq change of FN coordinates}, one checks that for all $\gamma $ in the pair of pants decomposition,
\begin{equation*}
    \partial_{\bar l_\gamma}+\frac{\bar t_\gamma[\rho]}{\bar l_\gamma[\rho]}\partial_{\bar t_\gamma} = \partial_{l_\gamma}+\frac{t_\gamma[\rho]}{l_\gamma[\rho]}\partial_{t_\gamma} \, .
\end{equation*}
\end{remark}

Proposition~\ref{prop dt*=-dt} and Corollary~\ref{cor (dl + blabla)*} also imply the following corollary, which, paired with Theorem~\ref{theo existence involution on tangent space}, concludes the proof of Theorem~\ref{theo intro existence involution on tangent space}.

\begin{corollary}
\label{coro inv is tensor}
Let $\Gamma=(\gamma_i)_{1\leq i\leq 3g-3}$ be a pair of pants decomposition of $\S$. Then, setting for all $[\rho] \in \chifd(\pi_1\S)$
\begin{equation*}
    \mathrm{Inv}^{\Gamma}_{[\rho]} \coloneqq {}^*: \T_{[\rho]}\chifd(\pi\S) \longrightarrow
 \T_{[\rho]}\chifd(\pi\S) 
\end{equation*}
defines a smooth $(1,1)$-tensor $\mathrm{Inv^{\Gamma}}$ on $\chifd(\pi_1\S)$, which is expressed, in the Fenchel--Nielsen coordinates $(l_{\gamma_1},\dots,l_{\gamma_{3g-3}},t_{\gamma_1},\dots, t_{\gamma_{3g-3}})$ associated with $\Gamma$, by:
\begin{equation*}
    \mathrm{Inv}^{\Gamma}=
    \begin{pNiceArray}{ccc|ccc}
    \Block{3-3}<\Large>{\Id} &   &   & \Block{3-3}<\Large>{0}
            &   &   \\
            &   &   &   &   &   \\
            &   &   &   &   &   \\
        \hline
        2 \, \frac{t_{\gamma_1}}{l_{\gamma_1}} &   & 0 & \Block{3-3}<\Large>{-\Id} &   &   \\ 
           & \Ddots &   &   &   &   \\
        0 &   & 2 \,\frac{t_{\gamma_{3g-3}}}{l_{\gamma_{3g-3}}} &   &   &
    \end{pNiceArray} \, .
\end{equation*}
\end{corollary}

\begin{proof}[Proof of Theorem~\ref{theo ccl dl*}]

Let $(\lambda_i)_{1\leq i \leq 3g-3} \in \R^{3g-3}$, and set
\begin{equation*}
 [\tau] = \sum_{i=1}^{3g-3} \lambda_{\gamma_i} \, \partial_{l_{\gamma_i}} \, .
\end{equation*}
By Proposition~\ref{prop tau* - tau}, we already know that 
\begin{equation}
\label{eq linear combination tau*-tau}
 [\tau]^*-[\tau] \in \mathrm{span}(\partial_{t_{\gamma_i}})_{1\leq i \leq 3g-3} \, . 
\end{equation}

Let $\gamma$ be an element of the pair of pants decomposition $([\gamma_i])_{1\leq i \leq 3g-3}$. We want to compute the coefficient in front of $\partial t_\gamma$ in the decomposition of $[\tau]^*-[\tau]$ given by \eqref{eq linear combination tau*-tau}. Let us consider the hyperbolic subsurface composed of the hyperbolic pair(s) of pants bounded by $c_\gamma^h$. There are two distinct cases.

\vspace{\baselineskip}

\underline{If $c^h_\gamma$ bounds two distinct hyperbolic pairs of pants $ \pant \neq \pant' $}, we consider the open set $\Sigma \coloneqq \pant \cup c^h_\gamma\cup \pant'$. 
Let us notice the following. By Proposition~\ref{prop restriction of inf twist is coboundary}, if $\eta \in \pi_1\S$ is in the pair of pants decomposition and $[\eta] \neq [\gamma]$, any representative $\tau_{\eta}\in \Zisom$ of $\partial_{t_{\eta}}$ restrict to a coboundary on $\pi_1\Sigma$, i.e. there exist $z \in \isom(\H^2)$ such that
\begin{equation*}
    \forall \xi \in \pi_1\Sigma \, , \qquad \tau_{\eta}(\xi) = \Ad \rho(\xi) \cdot z-z   \, .
\end{equation*}

Thus, as we have \eqref{eq linear combination tau*-tau}, the restriction to $\pi_1\Sigma$ of  $[\tau]^*-[\tau]$ is given by the sum of some coboundaries and a multiple of the restriction of a group cocycle representing the infinitesimal twist $\partial_{t_{\gamma}}$. Setting $k_g$ to be the coefficient in front of $\partial t_\gamma$ in the decomposition of $[\tau]^*-[\tau]$, and using expression \eqref{eq inf twist separating}, that means that $[\tau]^*-[\tau]$ has a representative $\tau' \in \Zisom$ such that
\begin{equation*} 
 \begin{cases}
\tau'(\xi)=0 & \text{if } \xi \in \pi_1 \pant \, , \\
\tau'(\xi)= k_\gamma \bp{\tH(C_\gamma) - \Ad \rho(\xi)\cdot \tH(C_\gamma)} & \text{if } \xi \in \pi_1 \pant' \, .
 \end{cases}
 \end{equation*}

As our aim is to prove that $k_\gamma=2 \lambda_\gamma t_\gamma[\rho]/l_\gamma[\rho]$, we only have to prove the following fact. The group cohomology class $[\tau]^*-[\tau]$ has a representative $\tau' \in \Zisom$ such that
\begin{equation} 
\label{eq proof case 1}
 \begin{cases}
\tau'(\xi)=0 & \text{if } \xi \in \pi_1 \pant \, , \\
\tau'(\xi)= 2 \lambda_\gamma \frac{t_\gamma[\rho]}{l_\gamma[\rho]} \bp{\tH(C_\gamma) - \Ad \rho(\xi)\cdot \tH(C_\gamma)} & \text{if } \xi \in \pi_1 \pant' \, .
 \end{cases}
\end{equation}
Let us now prove that \eqref{eq proof case 1} is satisfied.

\textbullet\ First, let us describe the fundamental group of $\Sigma$. It can be presented as $\pi_1\Sigma = \langle \gamma,\eta,\eta'\rangle$, the free group generated by $\gamma$ and two classes $\eta,\eta'$ of boundary loops of respectively $\pant$ and $\pant'$ (see Figure~\ref{fig pi1 S'1}) which are part of the pair of pants decomposition $([\gamma_i])_{1\leq i \leq 3g-3}$.

\begin{figure}[htbp]
\centering
\includegraphics{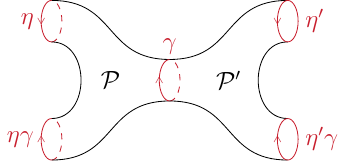} 
\caption{Homotopy classes of loops in $\Sigma = \pant \cup c^h_\gamma\cup \pant'$.}
\label{fig pi1 S'1}
\end{figure}

Let $C_\gamma$, $C_\eta$ and $C_{\eta'}$ be the respective oriented geodesics in $\H^2$ along which $\rho(\gamma)$, $\rho(\eta)$ and $\rho(\eta')$ are positive hyperbolic translations. Then, let $C$ be the geodesic line orthogonal to $C_\gamma$ and $C_\eta$, and let $C'$ be the geodesic line orthogonal to $C_\gamma$ and $C_{\eta'}$ (their orientations do not matter). 

Note that for all $k \in \Z$, we also have $\pi_1\Sigma = \langle \gamma,\eta,\eta'_k\rangle$, where $\eta'_k\coloneqq \gamma^k\eta'\gamma^{-k} \in \pi_1\S$ represents the same boundary loop as $\eta$. By definition of the twist coordinates \cite[Proof of Proposition 7.6.5]{Hubbard:2006}, up to replacing $\eta'$ with some $\eta'_k$ (depending on $\eta$ and the multicurve $\Gamma'$ used in order to define the twist coordinates), we have that the signed distance along $C_\gamma$ between the geodesics $C$ and $C'$ is the twist parameter $t_\gamma[\rho]$ (see Figure~\ref{fig twist universal pants}).

\textbullet\  By assumption, there exists and interval $I$ and a path $[\rho_s]_{s\in I}$ in $\chifd(\pi_1\S)$ with constant twist parameter along $\gamma$, and such that $[\rho_0]=[\rho]$ and
\begin{equation*}
 [\tau]=\ddso[\rho_s] \, .
\end{equation*}
As the twist parameter along $\gamma$ of the path $[\rho_s]_{s\in I}$  is constant, there exists a smooth family of Fuchsian representations $(\mu_s)_{s\in I}$ representing $[\rho_s]_{s\in I}$ such that $\mu_0 = \rho$ and:
\begin{enumerate}[label=(\arabic*)]
 \item for all $s \in I$, $\mu_s(\gamma)$ is a positive hyperbolic translation along a same oriented geodesic $C_\gamma \subset \H^2$ independent of $s$, \label{cdt 1}
 \item for all $s \in I$, $\mu_s(\eta)$ is a positive hyperbolic translation along an oriented geodesic $C_\eta^s \subset \H^2$ orthogonal to $C$ (which is independent of $s$), \label{cdt 2}
 \item for all $s \in I$, $\mu_s(\eta')$ is a positive hyperbolic translation along an oriented geodesic $C_{\eta'}^s \subset \H^2$ orthogonal to $C'$ (which is independent of $s$). \label{cdt 3}
\end{enumerate}
Figure~\ref{fig twist universal pants} depicts the situation.

\begin{figure}[htbp]
\centering
\includegraphics{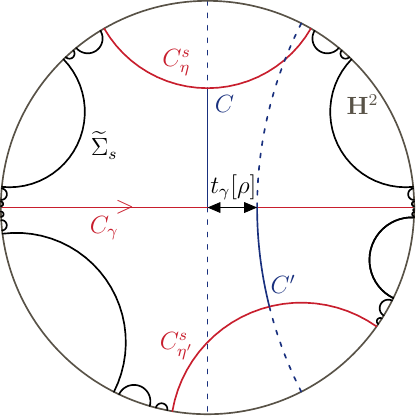}
\caption{The universal cover of the hyperbolic surface $\Sigma_s \subset \H^2/\rho_s(\pi_1\S)$ in the Poincar\'e disk model of $\H^2$.}
\label{fig twist universal pants}
\end{figure}

\textbullet\ From now on, let us use de Rham cohomology in order to prove that equation~\eqref{eq proof case 1} is satisfied. By Proposition~\ref{prop exist rep satisfying P eps}, we can let $\alpha \in \ZdR$ be a representative of $\Psi^{-1}_\rho[\tau] \in \HdR$ satisfying property~\ref{property P epsilon} for some $\varepsilon>0$. Then, by definition, $[\tau]^* = \Psi_\rho \dR{\alpha^*}$. 

Let $\{p\} = C_\gamma \cap C$ and $\{p'\} =C_\gamma \cap C'$. To prove that \eqref{eq proof case 1} is satisfied, we need to compute 
\begin{equation*}
    \int_p^{\xi \cdot p}\alphat^* - \alphat
\end{equation*}
for all $\xi \in \pi_1 \Sigma$.

\textbullet\ First, let us notice that,
\begin{equation}
\label{eq intp = intp'}
\forall \xi \in \pi_1\Sigma \, , \qquad  \int_p^{\xi \cdot p}\alphat^* - \alphat = \int_{p'}^{\xi \cdot p'}\alphat^* - \alphat \, .
\end{equation}
Indeed, as seen in the proof of Lemma~\ref{lem change of base point},
\begin{equation*}
\int_{p'}^{\xi \cdot p'}\alphat^* - \alphat = \int_{p}^{\xi \cdot p}\alphat^* - \alphat +\Ad\rho(\xi)\cdot z -z \, ,
\end{equation*}
where 
\begin{equation*}
z = \int_{p}^{p'} \alphat^* -\alphat \in \isom(\H^2)\, .
\end{equation*}
By Proposition~\ref{prop peps imply gluing}, $\alphat^* - \alphat$ vanishes in a neighbourhood of $C_\gamma$. Thus, using the geodesic path from $p$ to $p'$, which is a geodesic arc included in $C_\gamma$, in the curve integral defining $z$, we get 
$z=0$, implying \eqref{eq intp = intp'}.

\textbullet\ As we have 
\begin{equation*}
 [\tau]=\Psi_\rho\dR{\alpha} = \ddso[\mu_s] \, ,
\end{equation*}
there exists $x \in \isom(\H^2)$ such that
\begin{equation}
\label{eq alpha tau x}
\forall \xi \in \pi_1\S \, , \qquad \int_p^{\xi \cdot p} \alphat = \ddso \mu_s(\xi) \mu(\xi)^{-1}+ \Ad \rho(\xi)\cdot x -x \, .
\end{equation}
Then, as seen in the proof of Lemma~\ref{lem change of base point},
\begin{equation}
\label{eq alpha tau x'}
\forall \xi \in \pi_1\S \, , \qquad \int_{p'}^{\xi \cdot {p'}} \alphat = \ddso \mu_s(\xi) \mu(\xi)^{-1}+ \Ad \rho(\xi)\cdot x' -x' \, , 
\end{equation}
where 
\begin{equation*}
 x' = x + \int_{p}^{p'} \alphat \, .
\end{equation*}

\textbullet\ The geodesic path from $p$ to $p'$ is a geodesic arc of $C_\gamma$ of signed length $t_\gamma [\rho]$ (taking the orientation of $C_\gamma$ into account). Thus, as $\alpha$ satisfies property~\ref{property P epsilon} and, by Proposition~\ref{prop explicit value nice rep}, we have 
\begin{equation}
\label{eq expression x'}
 x' = x + \int_{p}^{p'} \tilde\alpha = x + \lambda_\gamma \frac{t_\gamma[\rho]}{l_\gamma[\rho]} \tH(C_\gamma)\, .
\end{equation}

\textbullet\ We claim that we have
\begin{equation*}
 x,x' \in \mathrm{span}\bp{\tH(C_\gamma)} \, .   
\end{equation*}
Indeed, for $\xi = \gamma$, equation \eqref{eq alpha tau x} becomes 
\begin{equation}
\label{eq alpha for gamma}
\int_p^{\gamma \cdot p} \alphat = \ddso \mu_s(\gamma) \mu(\gamma)^{-1} + \Ad \rho(\gamma)\cdot x -x \, .
\end{equation}
As for all $s \in I$, $\mu_s(\gamma)$ is a positive hyperbolic translation along $C_\gamma$, we have 
\begin{equation*}
 \ddso \mu_s(\gamma) \mu(\gamma)^{-1} \in \mathrm{span}\bp{\tH(C_\gamma)} \, .
\end{equation*}
Moreover, as $\alpha$ is a de Rham form satisfying property~\ref{property P epsilon}, so we also have 
\begin{equation*}
 \int_p^{\gamma \cdot p} \alphat \in \mathrm{span}\bp{\tH(C_\gamma)}\, .
\end{equation*}
Hence, using the direct sum decomposition of $\isom(\H^2)$ given by Corollary~\ref{cor th along gamma}, in equation~\eqref{eq alpha for gamma} we must have 
\begin{equation*}
\label{eq x in span tC}
 x \in \mathrm{span}\bp{\tH(C_\gamma)} \, .
\end{equation*}
Using \eqref{eq expression x'}, we also have 
\begin{equation*}
\label{eq x' in span tC}
 x' \in \mathrm{span}\bp{\tH(C_\gamma)} \, .
\end{equation*}

\textbullet\ The path of representations $(\mu_s)_{s\in I}$ satisfies \ref{cdt 1} and \ref{cdt 2}, thus it satisfies condition~\ref{cdt A} and \ref{cdt B} from subsection~\ref{subsec central lemma}. Let $S\in \isom(H^2)$ be the hyperbolic reflection along $C$. As $p \in C$, we can apply Lemma~\ref{lem central result} and we get, using equation~\eqref{eq alpha tau x}, that for all $\xi \in \{\gamma^{\pm1},\eta^{\pm1}\}$
\begin{equation*}
 \int_p^{\xi \cdot p}\alphat^* - \alphat= \int_p^{\xi \cdot p}\alphat_\pant^* - \alphat_\pant= \Ad \rho(\xi)\cdot y -y \, ,
\end{equation*}
where
\begin{equation*}
y \coloneqq \Ad S \cdot x -x \in \isom(\H^2)\, .
\end{equation*}
Hence, by $\Ad \rho $-equivariance~\eqref{eq ad equiv 1-form} of $\alphat^*$ and $\alphat$,
\begin{equation}
\label{eq p}
\forall \xi \in \pi_1\pant =\langle\gamma,\eta\rangle \, , \qquad \int_p^{\xi \cdot p}\alphat^* - \alphat= \int_p^{\xi \cdot p}\alphat_\pant^* - \alphat_\pant= \Ad \rho(\xi)\cdot y -y \, .
\end{equation}
Let us note that $x \in \mathrm{span}\bp{\tH(C_\gamma)}$ and $C$ is orthogonal to $C_\gamma$, using \eqref{eq Ad(S) t= -t}, we have $\Ad S \cdot x = -x$ so that 
\begin{equation*}
y = \Ad S \cdot x -x = -2x\, .
\end{equation*}

\textbullet\ The path of representations $(\mu_s)_{s\in I}$ satisfies \ref{cdt 1} and \ref{cdt 3}, thus it satisfies condition~\ref{cdt A} and \ref{cdt B} from subsection~\ref{subsec central lemma}, replacing $C$ by $C'$  and $\eta$ by $\eta'$. Using equation~\eqref{eq alpha tau x'} and the exact same arguments as in the previous point, we get that
\begin{equation*}
\forall \xi \in \pi_1\pant' \, , \qquad  \int_{p'}^{\xi \cdot p'}\alphat^* - \alphat = \int_{p'}^{\xi \cdot p'}\alphat^*_\pant - \alphat_\pant= \Ad \rho(\xi)\cdot y' -y' \, ,
\end{equation*}
where
\begin{equation*}
y' \coloneqq -2 x'\in \isom(\H^2)\, .
\end{equation*}
Using \eqref{eq intp = intp'} to change the basepoint from $p$ to $p'$, and the relation \eqref{eq expression x'} between $x$ and $x'$, that becomes
\begin{equation}
\label{eq p'}
\forall \xi \in \pi_1\pant' \, , \qquad \int_{p}^{\xi \cdot p}\alphat^* - \alphat = 2 \lambda_\gamma \frac{t_\gamma[\rho]}{l_\gamma[\rho]} \bp{\tH(C_\gamma) - \Ad \rho(\xi)\cdot \tH(C_\gamma)} + \bp{\Ad \rho(\xi)\cdot y -y}\, ,
\end{equation}

\textbullet\ That concludes the proof in the first case, as then, by equations~\eqref{eq p} and \eqref{eq p'}, $[\tau]^*-[\tau] = \Psi_\rho \dR{\alpha^* - \alpha}$
has a representative $\tau' \in \Zisom$ satisfying~\eqref{eq proof case 1}.

\vspace{\baselineskip}

\underline{If $c^h_\gamma$ bounds only one hyperbolic pair of pants $ \pant $}, we consider $\Sigma \coloneqq \pant \cup c^h_\gamma$. The proof is then similar to the previous case. 

Using, the same argument as in the previous case (based on Proposition~\ref{prop restriction of inf twist is coboundary} and expression~\eqref{eq inf twist non-separating}), we note that we only need to prove that $[\tau]^*-[\tau]$ has a representative $\tau' \in \Zisom$ such that
\begin{equation} 
 \begin{cases}
\label{eq proof case 2}
\tau'(\xi)=0 & \text{if } \xi \in \pi_1 \pant \, , \\
\tau'(\zeta)= 2 \lambda_\gamma\frac{t_\gamma[\rho]}{l_\gamma[\rho]} \tH(C_\gamma) \, .
 \end{cases}
\end{equation}

\textbullet\ The fundamental group of $\Sigma$ is the HNN extension $\pi_1\Sigma = \langle \gamma, \eta, \zeta \ \vert \ \zeta ^{-1} \gamma \zeta= \gamma\eta\rangle$, where $\zeta$ is the class of the loop created by the gluing of the boundaries $\gamma$ and $\gamma' = \gamma\eta$ of $\pant$ (see Figure~\ref{fig pi1 S'2}).

\begin{figure}[htbp]
\centering
\includegraphics{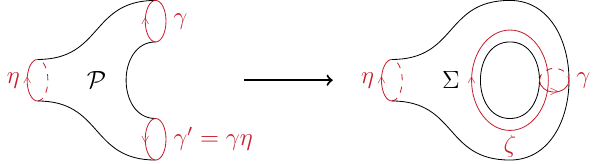} 
\caption{Homotopy classes of loops in $\pant$, and in $\Sigma = \pant \cup c^h_\gamma$ after the gluing.}
\label{fig pi1 S'2}
\end{figure}

Let $C_\gamma$, $C_{\eta}$ and $C_{\eta'}$ be the respective oriented geodesics in $\H^2$ along which $\rho(\gamma)$, $\rho(\eta)$ and $\rho(\eta')$ are positive hyperbolic translations. Then, let $C$ be the geodesic line orthogonal to $C_\gamma$ and $C_{\eta} $. It contains a segment which is a lift of one of the two minimal geodesic arcs from $c^h_\gamma$ to $c^h_\eta$. Let $C'$ be the geodesic line orthogonal to $C_\gamma = C_{\zeta \gamma \zeta^{-1}} = \rho(\zeta^{-1})C_{\gamma}$ and $C_{\eta'} = C_{\zeta \eta \zeta^{-1}} = \rho(\zeta^{-1})C_{\eta}$. It contains a segment which a lift of one of the other minimal geodesic arcs from $c^h_\gamma = c^h_{\gamma\eta}$ to $c^h_\eta = c^h_{\eta'}$. 

By definition of the twist coordinates \cite[Proof of Proposition 7.6.5]{Hubbard:2006}, up to replacing $\zeta$ with some $\gamma^k\zeta\gamma^{-k}$ with $k \in \Z$ (depending on the multicurve $\Gamma'$ used in order to define the twist coordinates), we have that the signed distance along $C_\gamma$ between the geodesics $C$ and $C'$ is $t_\gamma[\rho]$.

\textbullet\ By assumption, there is a path $[\rho_s]_{s\in I}$ in $\chifd(\pi_1\S)$ with constant twist parameter along $\gamma$, and such that $[\rho_0]=[\rho]$ and
\begin{equation*}
 [\tau]=\Psi_\rho\dR{\alpha}=\ddso[\rho_s] \, .
\end{equation*}
As the twist parameter along $\gamma$ of the path $[\rho_s]_{s\in I}$ is constant, there exists a smooth family of Fuchsian representation $(\mu_s)_{s\in I}$ representing $[ \rho_s]_{s\in I}$ such that $\mu_0 = \rho$ and:
\begin{enumerate}[label=(\roman*)]
 \item for all $s \in I$, $\mu_s(\gamma)$ is a positive hyperbolic translation along a same oriented geodesic $C_\gamma \subset \H^2$ independent of $s$, 
 \item for all $s \in I$, $\mu_s(\eta)$ is a positive hyperbolic translation along an oriented geodesic $C_{\eta}^s \subset \H^2$ orthogonal to $C$ (which is independent of $s$),
 \item for all $s \in I$, $\mu_s(\eta')$ is a positive hyperbolic translation along an oriented geodesic $C^s_{\eta'} \subset \H^2$ orthogonal to $C'$ (which is independent of $s$). 
\end{enumerate}

\textbullet\ By Proposition~\ref{prop exist rep satisfying P eps}, we can let $\alpha \in \ZdR$ be a representative of $\Psi^{-1}_\rho[\tau] \in \HdR$ satisfying property~\ref{property P epsilon} for some $\varepsilon>0$. Then, by definition, $[\tau]^* = \Psi_\rho \dR{\alpha^*}$. Let $\{p\} = C_\gamma \cap C$. Using the exact same arguments as in the previous case (relying on Lemma~\ref{lem central result}), we get that there exists $y \in \isom(\H^2)$ such that for all $\xi \in \pi_1\pant = \langle \gamma, \eta \rangle$,
\begin{equation*}
\int_p^{\xi \cdot p}\alphat^* - \alphat = \Ad \rho(\xi)\cdot y -y \, ,
\end{equation*}
and for all $\xi' \in \zeta (\pi_1\pant) \zeta ^{-1} = \langle \zeta\gamma \zeta^{-1} , \zeta \eta \zeta^{-1} \rangle = \langle \gamma \eta , \eta' \rangle$,
\begin{equation*}
 \int_{p}^{\xi' \cdot p}\alphat^* - \alphat = 2 \lambda_\gamma\frac{t_\gamma[\rho]}{l_\gamma[\rho]} \bp{\tH(C_\gamma) - \Ad \rho(\xi')\cdot \tH(C_\gamma)} + \bp{\Ad \rho(\xi')\cdot y -y}\, .
\end{equation*}
Thus, $[\tau]^*-[\tau] = \Psi_\rho\dR{\alpha^*-\alpha}$ has a representative $\tau' \in \Zisom$ such that for all $\xi \in \pi_1 \pant$
\begin{equation} 
 \begin{cases}
 \label{eq tau' second case}
\tau'(\xi)=0 \, , \\
\tau'(\zeta \xi \zeta^{-1} )= 2 \lambda_\gamma\frac{t_\gamma[\rho]}{l_\gamma[\rho]} \bp{\tH(C_\gamma) - \Ad \rho(\zeta \xi\zeta^{-1} )\cdot \tH(C_\gamma)} \, .
 \end{cases}
 \end{equation}

\textbullet\ By Proposition~\ref{prop cocy conj}, we also have that for all $\xi \in \pi_1 \pant$, 
\begin{equation}
\label{eq tau conj by zeta}
\tau'(\zeta \xi\zeta^{-1} ) = \bp{\tau(\zeta) - \Ad \rho(\zeta\xi\zeta^{-1} ) \cdot \tau(\zeta)} + \Ad \rho(\zeta) \cdot \tau (\xi) \, .
\end{equation}
Equations~\eqref{eq tau' second case} and \eqref{eq tau conj by zeta} imply that for every $\xi' \in \zeta (\pi_1\pant) \zeta ^{-1}$, 
\begin{equation*}
\tau(\zeta) - \Ad \rho(\xi') \cdot \tau(\xi) = 2 \lambda_\gamma\frac{t_\gamma[\rho]}{l_\gamma[\rho]} \bp{\tH(C_\gamma) - \Ad \rho(\xi')\cdot \tH(C_\gamma)} .
\end{equation*}
As $\pi_1\pant = \langle \gamma , \eta \rangle$, that implies that 
\begin{equation*}
 \tau(\zeta) - 2 \lambda_\gamma\frac{t_\gamma[\rho]}{l_\gamma[\rho]} \tH(C_\gamma) \in \ker \bp{\Ad \rho(\zeta\gamma \zeta^{-1}) -\Id} \cap \bp{\Ad \rho(\zeta\eta \zeta^{-1}) -\Id} \, .
\end{equation*}
By Lemma~\ref{lem decomp isom via Ad}, that is, 
\begin{equation*}
 \tau(\zeta) - 2 \frac{t_\gamma[\rho]}{l_\gamma[\rho]} \tH(C_\gamma) \in \mathrm{span} \bp{\tH(C_{\zeta\gamma \zeta^{-1}})} \cap \mathrm{span} \bp{\tH(C_{\zeta\eta \zeta^{-1}})} \, .
\end{equation*}
As the geodesics $C_{\zeta\gamma \zeta^{-1}} = \rho(\xi) C_\gamma$ and $C_{\zeta\eta \zeta^{-1}} = \rho(\xi) C_\eta$ are distinct as sets, the intersection above is trivial. Hence, we have 
\begin{equation}
\label{eq tau zeta}
 \tau(\zeta) = 2 \lambda_\gamma\frac{t_\gamma[\rho]}{l_\gamma[\rho]}\tH(C_\gamma) \, .
\end{equation}

\textbullet\ That concludes the proof for the second case, as equations~\eqref{eq tau' second case} and \eqref{eq tau zeta} imply that $[\tau]^*-[\tau]$ has a representative $\tau' \in \Zisom$ satisfying~\eqref{eq proof case 2}.
\end{proof}

\subsection{Proof of the formula}
\label{subsec proof of formula dl}

As explained in the beginning of the section, the proof of Theorem~\ref{theo w(dl,dl)=0} is now straightforward.

\begin{proof}
Let $\gamma, \eta \in \pi_1 \S$ be two distinct elements of the fixed pair of pants decomposition, Corollary~\ref{cor wg(*,*)=0} and Corollary~\ref{cor (dl + blabla)*} imply
\begin{equation*}
 \wg\vp{\partial_{l_\gamma}+\frac{t_\gamma}{l_\gamma} \partial_{t_\gamma} , \partial_{l_\eta}+\frac{t_\eta}{l_\eta} \partial_{t_\eta} } = 0 \, , 
\end{equation*}
Then, using Corollary~\ref{w(dt,dt)=0} and Theorem~\ref{Second Wolpert formula}, we get 
\begin{equation*}
 \wg(\partial_{l_\gamma}, \partial_{l_\eta})=\wg\vp{\partial_{l_\gamma}+\frac{t_\gamma}{l_\gamma} \partial_{t_\gamma} , \partial_{l_\eta}+\frac{t_\eta}{l_\eta} \partial_{t_\eta} } = 0 \, . \qedhere
\end{equation*}
\end{proof}

\section{Wolpert's Magic Formula and the closedness of the Goldman form}
\label{Sec WG closedness}

We now have all the tools in order to conclude and prove Wolpert's Magic Formula \cite{Wolpert:1985}.

\begin{theorem}[Corollary~\ref{theo intro WP in FN}]
\label{WG in FN}
In any Fenchel--Nielsen coordinates $(l_{\gamma_i},t_{\gamma_i})_{1\leq i\leq 3g-3}$ on $\chifd(\pi_1\S)$, the Goldman form is expressed as 
\begin{equation*}
\wg = 2\sum_{i=1}^{3g-3} \dif t_{\gamma_i}\wedge \dif l_{\gamma_i} \, .
\end{equation*}
\end{theorem}

\begin{proof}
By Corollary~\ref{w(dt,dt)=0}, for all $1 \leq i,j \leq 3g-3$, 
\begin{equation*}
\wg(\partial_{ t_{\gamma_i}},\partial_{ t_{\gamma_j}})=0 \, .
\end{equation*}
By Theorem~\ref{Second Wolpert formula}, for all $1 \leq i,j \leq 3g-3$, 
\begin{equation*}
\wg (\partial_{ t_{\gamma_i}},\partial_ {l_{\gamma_j}})=2\, \dif l_{\gamma_i} ( \partial_{l_{\gamma_j}}) = \begin{cases} 2 & \text{if} \ i=j \, , \\ 0 & \text{otherwise.} 
\end{cases}
\end{equation*}
By Theorem~\ref{theo w(dl,dl)=0}, for all $1 \leq i,j \leq 3g-3$, 
\begin{equation*}
\wg(\partial_{ l_{\gamma_i}},\partial_{ l_{\gamma_j}})=0 \, . \qedhere
\end{equation*}
\end{proof}

As previously mentioned, Theorem~\ref{WG in FN} is usually proved using the two formulas of Wolpert (Corollary~\ref{w(dt,dt)=0} and Theorem~\ref{Second Wolpert formula}) and the closedness of the Goldman form. As we have managed to prove Theorem~\ref{theo w(dl,dl)=0} directly, our proof does not rely on the closedness of the form and we can thus deduce closedness from Theorem~\ref{WG in FN}. Note that Theorem~\ref{WG in FN} also clearly implies the non-degenerate nature of the Goldman form. Hence, we have proved the following.

\begin{corollary}[Corollary~\ref{theo intro WP closed}]
\label{cor WG closed}
The Goldman form is a closed non-degenerate skew-symmetric $2$-form on $\chifd(\pi_1\S)$, i.e. it is a symplectic form.
\end{corollary}

\end{spacing}

\bibliography{Ablondi_WP}\label{sec:References} 
\bibliographystyle{alpha}

\end{document}